\documentclass[11pt]{amsart}

\usepackage{kotex}
\usepackage{amsmath}
\usepackage{mathrsfs}
\usepackage{amsfonts}
\usepackage{amscd}
\usepackage{slashed}

\usepackage{amssymb}
\usepackage{graphicx}
\usepackage{caption}
\usepackage{subcaption}
\usepackage{dsfont}
\usepackage{color}
\usepackage{hyperref}
\usepackage{bbm}
\usepackage{a4wide}
\usepackage{sseq}
\usepackage{tikz-cd}
\usepackage[abs]{overpic}

\newtheorem{theorem}{Theorem}[section]
\newtheorem{lemma}[theorem]{Lemma}
\newtheorem{proposition}[theorem]{Proposition}
\newtheorem{corollary}[theorem]{Corollary}

\newtheorem{assumption}[theorem]{Assumption}

\theoremstyle{definition}
\newtheorem*{definition*}{Definition}
\newtheorem{definition}[theorem]{Definition}

\theoremstyle{remark}
\newtheorem*{remark*}{Remark}
\newtheorem{remark}[theorem]{Remark}

\theoremstyle{question}
\newtheorem*{question*}{Question}

\numberwithin{equation}{section}

\theoremstyle{definition}

\newcommand{\myproof}[2]{Proof of {#1} {#2}}

\begin{document}

\title[Brane quantization and SYZ mirror symmetry]{Brane quantization and SYZ mirror symmetry}

\author[Chan, Leung, Li, Suen and Yau]{Kwokwai Chan, Naichung Conan Leung, Qin Li, Yat-Hin Suen and YuTung Yau}
\address{Department of Mathematics, The Chinese University of Hong Kong, Shatin, Hong Kong}
\email{kwchan@math.cuhk.edu.hk}
\address{The Institute of Mathematical Sciences and Department of Mathematics, The Chinese University of Hong Kong, Shatin, Hong Kong}
\email{leung@ims.cuhk.edu.hk}
\address{Shenzhen Institute for Quantum Science and Engineering, Southern University of Science and Technology, Shenzhen, China}
\email{liqin@sustech.edu.cn}
\address{Department of Mathematics, National Cheng Kung University}
\email{yhsuen@gs.ncku.edu.tw}
\address{Kavli Institute for the Physics and Mathematics of the Universe (WPI), The University of Tokyo Institutes for Advanced Study, The University of Tokyo, Kashiwa, Chiba 277-8583, Japan}
\email{yu-tung.yau@ipmu.jp}

\thanks{}

\maketitle

\begin{abstract}
	Coisotropic A-branes were introduced by Kapustin--Orlov to enlarge the Fukaya category of a symplectic manifold in a way that aligns with predictions from homological mirror symmetry. From a mathematical perspective, however, the categorical framework governing such branes remains largely undeveloped.
	On the other hand, Gukov--Witten's brane quantization suggests that a holomorphic deformation quantization of a holomorphic symplectic manifold $X$ arises from the endomorphism algebra $\operatorname{Hom}_A(\mathcal{B}_{\operatorname{cc}},\mathcal{B}_{\operatorname{cc}})$ of a canonical coisotropic A-brane $\mathcal{B}_{\operatorname{cc}}$, which naturally acts on the morphism space $\operatorname{Hom}_A(\mathcal{B},\mathcal{B}_{\operatorname{cc}})$ with a Lagrangian A-brane $\mathcal{B}$ that in turn gives precisely the geometric quantization of $\mathcal{B}$.
	\par
	In this paper, we consider a holomorphic symplectic manifold $X$ which admits an SYZ fibration and apply SYZ mirror symmetry to study its brane quantization. 
	Given any semi-affine, space-filling coisotropic A-brane $\mathcal{B}_{\operatorname{cc}}$ on $X$, we construct the mirror B-brane $\check{\mathcal{B}}_{\operatorname{cc}}$ on the mirror manifold $\check{X}$ by an SYZ transform.
	We then present a mathematical definition of the endomorphism algebra $\operatorname{Hom}_A(\mathcal{B}_{\operatorname{cc}},\mathcal{B}_{\operatorname{cc}})$ by constructing a distinguished non-formal holomorphic deformation quantization of $X$.
	Using a twisted family Toeplitz construction, we transform $\operatorname{Hom}_A(\mathcal{B}_{\operatorname{cc}},\mathcal{B}_{\operatorname{cc}})$ to the mirror B-side and prove that this induces an isomorphism 
	$\operatorname{Hom}_A(\mathcal{B}_{\operatorname{cc}},\mathcal{B}_{\operatorname{cc}})\cong \operatorname{Hom}_B(\check{\mathcal{B}}_{\operatorname{cc}},\check{\mathcal{B}}_{\operatorname{cc}})$
	between the endomorphism algebras. 
	Furthermore, taking any torus fiber of $X$ as the Lagrangian A-brane $\mathcal{B}$, we fully realize Gukov--Witten's proposal, namely, there is a natural action of $\operatorname{Hom}_A(\mathcal{B}_{\operatorname{cc}},\mathcal{B}_{\operatorname{cc}})$ on $\operatorname{Hom}_A(\mathcal{B},\mathcal{B}_{\operatorname{cc}})$ which is precisely mirror to the natural action on the mirror B-side. This provides a mathematical framework which is compatible with Gukov--Witten's brane quantization proposal, SYZ mirror symmetry as well as family Floer theory.
\end{abstract}

\section{Introduction}
\label{Section: Introduction}

\subsection{Background and the main result}
\quad\par

Kontsevich's homological mirror symmetry \cite{HMS} asserts that the Fukaya category $Fuk(X,\omega)$ of a symplectic manifold $(X,\omega)$ is quasi-equivalent to the derived category $D^b(\check{X},\check{J})$ of coherent sheaves of its mirror complex manifold $(\check{X},\check{J})$. 
However, Kapustin--Orlov \cite{KapOrl2003} observed that for some special symplectic structure, the Fukaya category of $(X,\omega)$ could be smaller than $D^b(\check{X},\check{J})$. An algebraic way to resolve this asymmetry is to add more objects by taking the split closure of $Fuk(X,\omega)$, as was done in e.g., \cite{HMS_4torus, HMS_for_log_CY_surfaces, Sheriden_HMS_CY_hypersurfaces}. 
A more geometric approach was proposed by Kapustin--Orlov \cite{KapOrl2003}, who suggested that coisotropic submanifolds, not just Lagrangian submanifolds, should be considered as objects in the Fukaya category. These new objects are called \emph{coisotropic A-branes}. 
Nevertheless, a definition of morphism spaces between coisotropic A-branes remains elusive, despite much effort \cite{AldZas2005, BisGua2022, GaiWit2022, KapLi2005, LeuYau2024, Qin2020, Qiu2024, Qui2012, Yau2024}.

On the other hand, Gukov--Witten \cite{GukWit2009} introduced the notion of a \emph{canonical coisotropic A-brane} of a symplectic manifold $(X,\omega)$. This is a space-filling brane, meaning that it is a coisotropic A-brane $\mathcal{B}_{\operatorname{cc}}$ supported on the whole $X$, and it is equipped with a $U(1)$-line bundle $(L,\nabla^L)$ over $X$ such that 
$I := \omega^{-1}F_{\nabla^L}$ defines a complex structure on $X$; here $-2\pi\sqrt{-1} F_{\nabla^L}$ is the curvature 2-form of the $U(1)$-connection $\nabla^L$. 
We thus obtain a structure sheaf $\mathcal{O}_X$ and a Dolbeault operator $\bar{\partial}_I$ on $X$. Also, the 2-form $\Omega_X:=F_{\nabla^L}+\sqrt{-1}\omega$ 
defines a holomorphic symplectic structure on $(X, I)$. 
Gukov and Witten proposed that 
\begin{itemize}
	\item The endomorphism algebra $\operatorname{Hom}_A(\mathcal{B}_{\operatorname{cc}},\mathcal{B}_{\operatorname{cc}})$ of the coisotropic A-brane $\mathcal{B}_{\operatorname{cc}}$ is a \emph{holomorphic deformation quantization (DQ)} of the sheaf cohomology $H^*(X,\mathcal{O}_X)$ of holomorphic functions on $X$ with respect to $\Omega_X$.
	\item Furthermore, for any Lagrangian brane $\mathcal{B}=(\underline{\mathcal{B}},\mathcal{L})$ such that $F_{\nabla^L}|_{\underline{\mathcal{B}}}$ is non-degenerate, the morphism space $\operatorname{Hom}_A(\mathcal{B},\mathcal{B}_{\operatorname{cc}})$ is the \emph{geometric quantization (GQ)} of $\underline{\mathcal{B}}$ equipped with the pre-quantum line bundle $L|_{\underline{\mathcal{B}}}\otimes\mathcal{L}^*$.
\end{itemize}
In particular, there is a natural action ``$DQ\curvearrowright GQ$'' given by the categorical composition
$$\operatorname{Hom}_A(\mathcal{B}_{\operatorname{cc}},\mathcal{B}_{\operatorname{cc}})\otimes \operatorname{Hom}_A(\mathcal{B},\mathcal{B}_{\operatorname{cc}})\to \operatorname{Hom}_A(\mathcal{B},\mathcal{B}_{\operatorname{cc}}).$$
This action is called a \emph{brane quantization}.

In this paper, we study Gukov--Witten's brane quantization proposal via SYZ mirror symmetry \cite{SYZ}. Suppose $X$ admits a Lagrangian torus fibration $p: X \to B$, or an \emph{SYZ fibration}, without singular fibers. The \emph{SYZ mirror} $\check{X}$ of $X$ is then constructed as the total space of the fiberwise dual torus fibration. We obtain a pair of fiberwise dual torus fibrations \cite{ChanLeung10, Leung05, LeuYauZas2000, SYZ}
\begin{center}
	\begin{tikzcd}
		X=T^*B/\Lambda \ar[rd, "p"'] & & \check{X}=TB/\check{\Lambda} \ar[ld, "\check{p}"]\\
		& B
	\end{tikzcd}
\end{center}

Consider a canonical coisotropic A-brane $\mathcal{B}_{\operatorname{cc}}=(X,L,\nabla^L)$ on $X$ which is \emph{semi-affine}, meaning that the restriction $F_{\nabla^L}|_{F_x}$, where $F_x := p^{-1}(x)$ is the torus fiber over $x \in B$, is a constant 2-form independent of $x \in B$; this notion was introduced in \cite{ChaLeuZha2018}. We impose an assumption that for each $x \in B$, $F_{\nabla^L} \vert_{F_x}$ is non-degenerate. For any positive integer $k \in \mathbb{Z}_{>0}$, we also consider the canonical coisotropic A-brane $\mathcal{B}_{\operatorname{cc}}^{(k)}:=(X,L^{\otimes k},\nabla^{L^{\otimes k}})$ with respect to the scaled symplectic structure $k\omega$.

By a slight modification of the \emph{SYZ transform} introduced in \cite{ChaLeuZha2018}, we construct a family $\{\check{\mathcal{B}}_{\operatorname{cc}}^{(k)}\}_{k\in\mathbb{Z}_{>0}}$ of B-branes on $\check{X}$ which is SYZ mirror to the family $\{\mathcal{B}_{\operatorname{cc}}^{(k)}\}_{k\in\mathbb{Z}_{>0}}$ of A-branes on $X$. 
We prove that the mirror B-brane $\check{\mathcal{B}}_{\operatorname{cc}}^{(k)}$ is a holomorphic vector bundle $\check{E}^{(k)}$ on $\check{X}$ whose rank is given by $rk(\check{E}^{(k)})=\int_{F_x}ch(L)^k$, which tends to infinity as $k \to \infty$. 

Our main focus is on the A-model morphism spaces. Following Gukov--Witten's proposal \cite{GukWit2009}, we propose a mathematically rigorous definition of the endomorpism algebra of $\mathcal{B}_{\operatorname{cc}}^{(k)}$, as a \emph{non-formal holomorphic deformation quantization} of $X$. 
Here comes the key novelty of the paper.
By considering fiberwise geometric quantization with K\"ahler polarization on $X$ and introducing a suitable twist of family of Toeplitz operators, we are able to define a mirror transform from the endomorphism algebra of $\mathcal{B}_{\operatorname{cc}}^{(k)}$ to that of the mirror B-brane $\check{\mathcal{B}}_{\operatorname{cc}}^{(k)}$. 
Our main result is the following mirror isomorphism.
\begin{theorem}
	For every $k\in\mathbb{Z}_{>0}$, the mirror transform, defined as a family twisted Toeplitz operators, induces an isomorphism
	$$\operatorname{Hom}_A(\mathcal{B}_{\operatorname{cc}}^{(k)},\mathcal{B}_{\operatorname{cc}}^{(k)})\cong \operatorname{Hom}_B(\check{\mathcal{B}}_{\operatorname{cc}}^{(k)},\check{\mathcal{B}}_{\operatorname{cc}}^{(k)})$$
	of graded algebras.
\end{theorem}
Note that the ring structure on the LHS is given by a \emph{star product}, while $\operatorname{Hom}_B(\check{\mathcal{B}}_{\operatorname{cc}}^{(k)},\check{\mathcal{B}}_{\operatorname{cc}}^{(k)}) = H^*(\check{X},\operatorname{End}(\check{E}^{(k)}))$ so the ring structure on the RHS is the \emph{matrix multiplication} in $\operatorname{End}(\check{E}^{(k)})$. 
Explicit transformations of morphism spaces which induce mirror isomorphisms are very rarely achieved. To the best of our knowledge, our construction is the first such transform for coisotropic A-branes.

As we will see shortly, the fiber of the holomorphic vector bundle $\check{E}^{(k)}$ at the point $(F_x,\mathcal{L})\in\check{X}$ is by construction given by
$$\check{E}_{(F_x,\mathcal{L})}^{(k)}=H^0(F_x,L^{\otimes k}|_{F_x}\otimes\mathcal{L}^*\otimes\sqrt{K_{\Omega}}).$$
This is nothing but the geometric quantization of $(F_x,\Omega_X|_{F_x})$, which, again according to Gukov--Witten's proposal \cite{GukWit2009}, should exactly be the A-model morphism space $\operatorname{Hom}_A(\mathcal{B},\mathcal{B}_{\operatorname{cc}}^{(k)})$, where $\mathcal{B}$ denotes the Lagrangian A-brane $(F_x, \mathcal{L})$. In other words, we should set
$$\operatorname{Hom}_A(\mathcal{B},\mathcal{B}_{\operatorname{cc}}^{(k)}):=H^0(F_x,L^{\otimes k}|_{F_x}\otimes\mathcal{L}^*\otimes\sqrt{K_{\Omega}}).$$

Now the natural action of $\operatorname{Hom}_A(\mathcal{B}_{\operatorname{cc}}^{(k)},\mathcal{B}_{\operatorname{cc}}^{(k)})$ on $\operatorname{Hom}_A(\mathcal{B},\mathcal{B}_{\operatorname{cc}}^{(k)})$ by Toeplitz operators \cite{ChaLeuLi2023, Yau2024}
corresponds, through the above mirror isomorphism, to the simple and natural B-model action of $\operatorname{Hom}_B(\check{\mathcal{B}}_{\operatorname{cc}}^{(k)},\check{\mathcal{B}}_{\operatorname{cc}}^{(k)}) = H^*(\check{X},\operatorname{End}(\check{E}^{(k)}))$ on $\operatorname{Hom}_B(\check{\mathcal{B}},\check{\mathcal{B}}_{\operatorname{cc}}^{(k)}) = \check{E}_{(F_x,\mathcal{L})}^{(k)}$.
This shows that the $DQ$-module structure on $GQ$ (of the fiber $F_x$), which fully and rigorously realizes Gukov--Witten's brane quantization proposal in the current setting, is compatible with SYZ mirror symmetry as well as family Floer theory.

\subsection{The SYZ transform and construction of the mirror B-brane}
\quad\par

Let us explain our constructions in more detail.
Recall that in SYZ mirror symmetry \cite{ChanLeung10, Leung05, LeuYauZas2000, SYZ}, the SYZ mirror $\check{X}$ of $X$ is constructed as the moduli space of pairs $\{(F_x, \mathcal{L})\}$ consisting of a Lagrangian torus fiber $F_x = p^{-1}(x)$ of $p:X\to B$ equipped with a flat $U(1)$-local system $\mathcal{L}$. Given a semi-affine canonical coisotropic A-brane $\mathcal{B}_{\operatorname{cc}} = (X,L,\nabla^L)$ on $X$ whose curvature is non-degenerate on every fiber of $p: X \to B$, we can then construct the mirror B-brane $\check{\mathcal{B}}_{\operatorname{cc}} = \check{E}$, as a holomorphic vector bundle on $\check{X}$, using a slight modification of the \emph{SYZ transform} defined in \cite{ChaLeuZha2018}. 

More precisely, the fiber of $\check{E}$ over $(F_x,\mathcal{L})\in\check{X}$ is given by
$$\check{E}_{(F_x,\mathcal{L})} := H^0(F_x, L|_{F_x}\otimes\mathcal{L}^*\otimes\sqrt{K_{\Omega}}),$$
where $\sqrt{K_{\Omega}}$ is a holomorphic half-form bundle on $F_x$ with respect to some auxiliary locally constant (with respect to $x \in B$) flat complex structure $\Omega$ on $F_x$. This complex structure turns each fiber into an abelian variety and $L|_{F_x} \otimes \mathcal{L}^*$ gives a pre-quantum line bundle on $(F_x,\Omega)$. Thus, $\check{E}_{(F_x,\mathcal{L})}$ is nothing but the \emph{Hilbert space with half-form corrections} in the geometric quantization of $F_x$ with respect to the pre-quantum line bundle $L|_{F_x}\otimes\mathcal{L}^*$ and the K\"ahler polarization.

To glue these fibers together, we need to consider a change of affine charts in the base $B$.
Let $\Omega'$ be an auxiliary locally constant flat complex structure on $F_x$ over another chart. 
Then the two spaces of K\"ahler polarized sections are related by the \emph{BKS isomorphism}
$$\operatorname{BKS}_{\Omega',\Omega}:H^0(F_x,L\otimes\mathcal{L}^*\otimes\sqrt{K_{\Omega}})\xrightarrow{\sim}H^0(F_x,L\otimes\mathcal{L}^*\otimes\sqrt{K_{\Omega'}}).$$
It is well-known that BKS maps satisfy \cite{And2005, BaiMouNun2010}
$$\operatorname{BKS}_{\Omega'',\Omega}\circ \operatorname{BKS}_{\Omega',\Omega}=\operatorname{BKS}_{\Omega'',\Omega}.$$
This yields the desired cocycle condition and independence of the choice of the fiber K\"ahler polarization. 

This construction can easily be generalized to the \emph{level $k$} canonical coisotropic A-brane $$\mathcal{B}_{\operatorname{cc}}^{(k)}:=(X,L^{\otimes k},\nabla^{L^{\otimes k}})$$ 
with respect to the $k$-scaled symplectic manifold $(X,k\omega)$ for any $k \in \mathbb{Z}_{>0}$. The mirror B-brane $\check{\mathcal{B}}_{\operatorname{cc}}^{(k)}$ is a holomorphic vector bundle $\check{E}^{(k)}$ on $\check{X}^{(k)}$ (with complex coordinates $\check{z}_i:=x_i+\frac{\sqrt{-1}}{k}\check{y}_i$, $i=1,\dots,2n$), whose rank is given by the Riemann-Roch formula:
$$rk(\check{E}^{(k)})=\int_{F_x}ch(L)^k.$$
The detailed construction of $\check{\mathcal{B}}_{\operatorname{cc}}^{(k)}$ will be given in Section \ref{Section: The mirror B-brane via Fourier--Mukai type transforms} (see Theorems \ref{Theorem: unitary connection on the local mirror B brane} and \ref{Proposiiton: Global mirror B-brane}).

\subsection{Non-formal holomorphic deformation quantization}
\quad\par

The next step is to define the endomorphism algebra $\operatorname{Hom}_A(\mathcal{B}_{\operatorname{cc}}^{(k)},\mathcal{B}_{\operatorname{cc}}^{(k)})$ and construct the mirror transform. 

First of all, Gukov--Witten \cite{GukWit2009} proposed that $\operatorname{Hom}_A(\mathcal{B}_{\operatorname{cc}}, \mathcal{B}_{\operatorname{cc}})$ carries a natural noncommutative algebra structure arising from a holomorphic deformation quantization of $(X, \Omega_X)$. They further suggested that when $X$ admits a ``good A-model'' (see \cite[pages 2-3]{GukWit2009}), this deformation quantization should be \emph{non-formal} and governed by an actual complex parameter rather than a formal one; see \cite[page 11]{GukWit2009}.

In Section \ref{Section: Holomorphic deformation quantization}, we construct a distinguished \emph{non-formal holomorphic deformation quantization} of $(X, \Omega)$ depending on a real parameter $\hbar$. 
For each level $k \in \mathbb{Z}_{>0}$, the specialization $\hbar = 1/k$ realizes the A-model endomorphism algebra $\operatorname{Hom}_A(\mathcal{B}_{\operatorname{cc}}^{(k)}, \mathcal{B}_{\operatorname{cc}}^{(k)})$ for the A-model on $(X, k\omega)$. 
Such a non-formal definition is also necessary to make sense of the desired mirror statement $\operatorname{Hom}_A(\mathcal{B}_{\operatorname{cc}}^{(k)}, \mathcal{B}_{\operatorname{cc}}^{(k)}) \cong \operatorname{Hom}_B(\check{\mathcal{B}}_{\operatorname{cc}}^{(k)}, \check{\mathcal{B}}_{\operatorname{cc}}^{(k)})$. 

Let us briefly describe the construction. Using data from $\mathcal{B}_{\operatorname{cc}}^{(k)}$, we first write down a non-formal Moyal-type star product $\star_q$ on each torus fiber of $X$.
On the other hand, following \cite{ChaLeuZha2018}, we construct some distinguished local models for the semi-affine A-brane $\mathcal{B}_{\operatorname{cc}}^{(k)}$. These so-called \emph{skew-Smith forms} of $\mathcal{B}_{\operatorname{cc}}^{(k)}$ produce an atlas of complex affine coordinates on $X$. Now for each affine chart $U \subset B$, let $X_U := p^{-1}(U)$. We expand each smooth function $f\in C^{\infty}(X_U)$ as a fiberwise Fourier series whose Fourier coefficients $\widehat{f}_m(x)$ are smooth functions on $U\subset B$ and then we define $\star_{U,\hbar}:C^{\infty}(X_U)\times C^{\infty}(X_U)\to C^{\infty}(X_U)$ by extending the star product $\star_q$ $C^{\infty}(U)$-bilinearly. The Schwartz's condition (in the lattice variable $m$) of the Fourier coefficients ensures (absolute) convergence. We denote this non-formal star product by $\star_{U,\hbar}$ and extend it to a graded star product
$$\star_{U,\hbar}:\Omega^{0,*}(X_U)\times \Omega^{0,*}(X_U)\to \Omega^{0,*}(X_U)$$
in the most natural way. 

The transition maps of $(X,I)$ being affine allows us to glue these graded algebras to a global graded algebra
$$\star_{\hbar}:\Omega^{0,*}(X)\times \Omega^{0,*}(X)\to \Omega^{0,*}(X),$$
and we can establish the Leibniz rule:
$$\bar{\partial}_I(\alpha\star_{\hbar}\beta)=(\bar{\partial}_I\alpha)\star_{\hbar}\beta+(-1)^{|\alpha|}\alpha\star_{\hbar}\bar{\partial}_I\beta.$$
As a result, we obtain a differential graded algebra (dga) (see Theorem \ref{Theorem: non-commutative Dolbeault complex})
$$\mathcal{A}_{X,\hbar}:=(\Omega^{0,*}(X),\bar{\partial}_I,\star_{\hbar}).$$
A key feature of this Moyal-type star product $\star_\hbar$ is that, although it is non-local along the fiber direction of the SYZ fibration $p: X \to B$, it remains \emph{local along the base $B$}.
We then define
$$\operatorname{Hom}_A(\mathcal{B}_{\operatorname{cc}}^{(k)},\mathcal{B}_{\operatorname{cc}}^{(k)}):=(H_{\bar{\partial}_I}^{0,*}(X),\star_{k^{-1}}).$$

Some justification of this definition is in order. Let $F_x$ be a torus fiber of $X$ over a point $x \in B$. Our standing assumption \ref{Setup} on the space-filling A-brane $\mathcal{B}_{\operatorname{cc}}^{(k)}$ implies that $(X, \Omega_X)$ is a \emph{complexification} of the symplectic torus $(F_x, \Omega_X \vert_{F_X})$.
According to the Gukov--Witten picture, we shall realize $\operatorname{Hom}_A(\mathcal{B}, \mathcal{B}_{\operatorname{cc}}^{(k)})$ as the geometric quantization of $(F_x, \Omega_X \vert_{F_x})$ in a \emph{torus-invariant K\"ahler polarization}.

As explained in \cite{Yau2024}, this places us in a setting where the \emph{Berezin--Toeplitz star product} yields a \emph{canonical but formal} holomorphic deformation quantization on $X_U$ for a sufficiently small neighborhood $U \subset B$ of $x$, uniquely determined (up to equivalence) by Toeplitz operators acting on these morphism spaces \cite{ChaLeuLi2023, Yau2024}.
This is also in line with the philosophy of family Floer theory, namely, the noncommutative product on $\operatorname{Hom}_A(\mathcal{B}_{\operatorname{cc}}^{(k)}, \mathcal{B}_{\operatorname{cc}}^{(k)})$ should be determined by its action on $\operatorname{Hom}_A(\mathcal{B}, \mathcal{B}_{\operatorname{cc}}^{(k)})$, where $\mathcal{B}$ ranges over all the Lagrangian A-branes supported on torus fibers of $X$ (cf. \cite{GaiWit2022}).

Now, as shown in \cite{And2005} (see also \cite{LeuYau2023}), the Berezin--Toeplitz star product and the Moyal star product on a symplectic torus are actually gauge equivalent (as formal star products). So it seems that we can use the fiberwise Berezin--Toeplitz quantization to define the required holomorphic deformation quantization on $X$. Unfortunately, for different open subsets $X_U$'s, the choices of torus invariant K\"ahler polarization could be very different (in fact they could be arbitrary). To glue these Berezin--Toeplitz star products together, we need to choose gauge transformations between different polarizations. The issue is that these gauge transformations need \emph{not} converge at the non-formal level (i.e., when we evaluate at $\hbar = 1/k$). 
In contrast, our construction gives a holomorphic deformation quantization of $X$ which is \emph{non-formal but locally formally equivalent to the fiberwise Berezin--Toeplitz deformation quantization}. This is consistent with the Gukov--Witten picture. More importantly, we are able to construct the mirror transform of the morphism spaces under such a definition.

\subsection{The mirror isomorphism}
\quad\par

In defining the mirror transform for $\mathcal{B}_{\operatorname{cc}}^{(k)}$, our idea is to apply the \emph{family Toeplitz construction}. Recall that for a general compact K\"ahler manifold $(M,\omega,J)$ equipped with a pre-quantum line bundle $L$, the \emph{level $k$ Toeplitz operator} $T_f^{(k)}\in\operatorname{End}_{\mathbb{C}}(H^0(M,L^{\otimes k}))$ associated to $f\in C^{\infty}(M)$ is defined by
$$T_f^{(k)}:s\mapsto \Pi^{(k)}(f\cdot s),$$
where $\Pi^{(k)}:L^2(M,L^{\otimes k})\to H^0(M,L^{\otimes k})$ is the orthogonal projection \cite{BT_operators}. 

Applying this construction to each fiber of $p:X\to B$, we obtain a $\mathbb{C}$-linear map
$$T_U^{(k)}:C^{\infty}(X_U)\to C^{\infty}(\check{X}_U^{(k)},\operatorname{End}(\check{E}^{(k)})).$$
By sending $d\bar{z}^{i_1}\wedge\cdots\wedge d\bar{z}^{i_l}$ to $d\bar{\check{z}}_{i_1}\wedge\cdots\wedge d\bar{\check{z}}_{i_l}$, $T_U^{(k)}$ extends to a graded $\mathbb{C}$-linear map $$T_U^{(k)}:\Omega^{0,*}(X_U)\to \Omega^{0,*}(\check{X}_U^{(k)},\operatorname{End}(\check{E}_U^{(k)})).$$
The most surprising property of $T_U^{(k)}$ is that it interchanges the Dolbeault operators $\bar{\partial}_I$ and $\bar{\partial}_{\operatorname{End}(\check{E}^{(k)})}$:
\begin{proposition}[= Lemma \ref{Lemma: holomorphicity}]\label{proposition : holomorphicity_intro}
	The $\mathbb{C}$-linear map
	$$T_U^{(k)}:(\Omega^{0,*}(X_U),\bar{\partial}_I)\to (\Omega^{0,*}(\check{X}_U^{(k)},\operatorname{End}(\check{E}^{(k)})),\bar{\partial}_{\operatorname{End}(\check{E}^{(k)})})$$
	is a cochain homomorphism.
\end{proposition}

However, we encounter a difficulty when we try to globalize the construction. Namely, the collection of maps $\{T_U^{(k)}\}_U$ does \emph{not} commute with the transition maps (i.e., the BKS maps) of $\check{E}^{(k)}$. 
It turns out that this problem can be solved by applying the gauge transformation $e^{\frac{1}{k}\Delta_{\Omega}}$, where $\Delta_{\Omega}$ is a complexification of the $\bar{\partial}_{\Omega}$-Laplacian of $(F_x,\Omega)$.
Since this gauge transformation involves differential operators, it is \emph{not} well-defined on $C^{\infty}(X_U)$. 
Nevertheless, a careful calculation shows that the composition $$\Phi_U^{(k)}:=T_U^{(k)}\circ e^{\frac{1}{k}\Delta_{\Omega}}$$
is actually well-defined. 
Moreover, the collection of maps $\{\Phi_U^{(k)}\}_U$ now commutes with the BKS maps. Extending $\Phi^{(k)}$ as in Proposition \ref{proposition : holomorphicity_intro}, we obtain a cochain map between the Dolbeault complex of $(X,\mathcal{O}_X)$ and that of $\operatorname{End}(\check{E}^{(k)})$:
$$\Phi^{(k)}:(\Omega^{0,*}(X),\bar{\partial}_I)\to (\Omega^{0,*}(\check{X},\operatorname{End}(\check{E}^{(k)})),\bar{\partial}_{\operatorname{End}(\check{E}^{(k)})}),$$
called the \emph{mirror transform}.
Below is a more precise version of our main theorem.

\begin{theorem}[=Theorem \ref{Theorem: mirror statement}]
	For every $k \in \mathbb{Z}_{>0}$, we have
	$$\Phi^{(k)}(\alpha\star_{k^{-1}}\beta)=\Phi^{(k)}(\alpha)\circ\Phi^{(k)}(\beta)$$
	for any $\alpha,\beta\in \Omega^{0,*}(X)$, and $\Phi^{(k)}$ is bijective. 
	In other words, the mirror transform
	$$\Phi^{(k)}:(\Omega^{0,*}(X),\bar{\partial}_I)\to (\Omega^{0,*}(\check{X},\operatorname{End}(\check{E}^{(k)})),\bar{\partial}_{\operatorname{End}(\check{E}^{(k)})})$$
	is an isomorphism of differential graded algebras for every $k\in\mathbb{Z}_{>0}$.
	In particular, we have an isomorphism
	$$\operatorname{Hom}_A(\mathcal{B}_{\operatorname{cc}}^{(k)},\mathcal{B}_{\operatorname{cc}}^{(k)})\cong \operatorname{Hom}_B(\check{\mathcal{B}}_{\operatorname{cc}}^{(k)},\check{\mathcal{B}}_{\operatorname{cc}}^{(k)}),$$
	of graded algebras for every $k\in\mathbb{Z}_{>0}$.
\end{theorem}

All constructions in this paper do not require any choice of metrics on $X$. In particular, our main theorem remains valid even when $X$ is non-K\"ahler (cf. Subsection \ref{Subsubsection: Kodaira--Thurston manifold}).

A key technical tool in our proof is the family Weil--Brezin transform, which renders the construction explicit and computationally tractable. To illustrate, consider the case $B = \mathbb{R}^2$. Then $(X, \omega)$ admits a semi-affine coisotropic A-brane $\mathcal{B}_{\operatorname{cc}}$ that induces a complex structure on $X$ with complex coordinates $(e^{2\pi z^1}, e^{2\pi z^2}) = (e^{2\pi(x_1 - \sqrt{-1} y^2)}, e^{2\pi(x_2 + \sqrt{-1} y^1)})$, where $(y^1, y^2)$ are the fiber periodic coordinates on $X$ associated with the standard affine coordinates $(x_1, x_2)$ on $B$ (cf. Subsection \ref{Subsubsection: transcendental mirror symmetry}). In particular, $X \cong (\mathbb{C}^\times)^2$. For these coordinate functions, the non-formal star product is a phase deformation of the usual commutative product:
\begin{equation*}
	e^{2\pi z^1} \star_{k^{-1}} e^{2\pi z^2} = e^{-\frac{\pi\sqrt{-1}}{k}} e^{2\pi (z^1 + z^2)}.
\end{equation*}
The family Weil--Brezin transform identifies $\Gamma(\check{X}, \check{E}^{(k)})$ with the space $\mathcal{C}^\infty(\mathbb{R}^2, \mathcal{S}(\mathbb{R}))$ of smooth functions on $\mathbb{R}^2 \times \mathbb{R}$ (with coordiantes $(x, \check{y}_1)$) that are rapidly decreasing in the $\mathbb{R}$-variable (see Section \ref{Section: The mirror B-brane via Fourier--Mukai type transforms}). Under this identification, the mirror transform can be expressed in terms of multiplication and shift operators. More precisely, for any $s \in \Gamma(\check{X}, \check{E}^{(k)}) \cong \mathcal{C}^\infty(\mathbb{R}^2, \mathcal{S}(\mathbb{R}))$,
\begin{align*}
	(\Phi^{(k)}(e^{2\pi z^1}) (s)) (x, \check{y}_1) = & e^{ 2\pi \check{z}_1} \cdot s (x, \check{y}_1),\\
	(\Phi^{(k)}(e^{2\pi z^2}) (s)) (x, \check{y}_1) = & e^{2\pi x_2} \cdot s (x, \check{y}_1 + 1),\\
	(\Phi^{(k)} (e^{2\pi (z^1 + z^2)})(s)) (x, \check{y}_1) = & e^{2\pi \left( \left( \check{z}_1 + \frac{\sqrt{-1}}{2k} \right) + x_2 \right)} \cdot s(x, \check{y}_1 + 1).
\end{align*}
It follows directly that $\Phi^{(k)} (e^{2\pi z^1} \star_{k^{-1}} e^{2\pi z^2}) = \Phi^{(k)}(e^{2\pi z^1}) \circ \Phi^{(k)}(e^{2\pi z^2})$.

In this mirror isomorphism, the A-model action of $\operatorname{Hom}_A(\mathcal{B}_{\operatorname{cc}}^{(k)}, \mathcal{B}_{\operatorname{cc}}^{(k)})$ on $\operatorname{Hom}_A(\mathcal{B}, \mathcal{B}_{\operatorname{cc}}^{(k)})$ precisely corresponds to the B-model action of $H^*(\check{X},\operatorname{End}(\check{E}^{(k)}))$ on $\check{E}_{(F_x,\mathcal{L})}^{(k)}$. As far as we know, this is the first mirror theorem for coisotropic A-branes as well as for brane quantization. This also provides another instance where we see that SYZ transforms can provide information which is beyond the reach of homological mirror symmetry (cf. \cite{CS_SYZ_immersed}).

\subsection{Organization of the paper}
\quad\par

In Section \ref{Section: Preliminaries}, we review the notion of coisotropic A-branes and in particular the semi-affine condition introduced in \cite{ChaLeuZha2018}.
In Section \ref{Section: Holomorphic deformation quantization}, we construct a non-formal Moyal-type holomorphic star product on $\mathcal{O}_X[[\hbar]]$. We also present three concrete examples (including a non-K\"ahler example).
In Section \ref{Section: The mirror B-brane via Fourier--Mukai type transforms}, we make a slight modification of the SYZ transform constructed in \cite{ChaLeuZha2018} and use it to construct the family $\{\check{\mathcal{B}}_{\operatorname{cc}}^{(k)}\}_{k\in\mathbb{Z}_{>0}}$ of mirror B-branes.
Section \ref{Section: Brane quantization via SYZ transforms} is the heart of the paper, where we define the mirror transform for morphism spaces via a twisted family Toeplitz construction. Finally we prove that the mirror transform induces a mirror isomorphism for every level $k \in \mathbb{Z}_{>0}$.

\subsection{Setup and notations}
\quad\par
\label{Subsection: Setup and notations}
Throughout this paper (except in Section \ref{Section: Preliminaries}), let $B$ be a $2n$-dimensional connected integral affine manifold, and define $X := T^*B / \Lambda$ and $\check{X} := TB / \check{\Lambda}$. Here, $\Lambda = T_\mathbb{Z}^*B \subset T^*B$ and $\check{\Lambda} = T_\mathbb{Z}B \subset TB$ denote the lattice bundle and the dual lattice bundle, respectively, induced by the integral affine structure on $B$. Moreover, let $\mathcal{B}_{\operatorname{cc}} = (X, L, \nabla^{L})$ be a rank-$1$ coisotropic A-brane on the symplectic manifold $(X, \omega)$ satisfying the following standing assumption:

\begin{assumption}
	\label{Setup}
	The coisotropic A-brane $\mathcal{B}_{\operatorname{cc}}$ is \emph{semi-affine} (see Definition \ref{Definition: semi-affine A-brane}) and the curvature of $\nabla^L$ is non-degenerate on every fiber of the Lagrangian torus fibration $p: X \to B$.
\end{assumption}

For each $k \in \mathbb{Z}_{>0}$, $\mathcal{B}_{\operatorname{cc}}^{(k)} = (X, L^{\otimes k}, \nabla^{L^{\otimes k}})$ is then a rank-$1$ coisotropic A-brane on the symplectic manifold $(X, k\omega)$. We denote by $\check{\mathcal{B}}_{\operatorname{cc}}^{(k)}$ be the mirror brane of $\mathcal{B}_{\operatorname{cc}}^{(k)}$, which will be constructed in Section \ref{Section: The mirror B-brane via Fourier--Mukai type transforms}. We furthermore adopt the following notations:
\begin{itemize}
	\item Let $\mathcal{C}^\infty(M)$ be the algebra of smooth $\mathbb{C}$-valued functions on a smooth manifold $M$.
	\item For all open subset $U$ of $B$, define the open subsets $X_U := X \vert_U$ of $X$ and $\check{X}_U := \check{X} \vert_U$ of $\check{X}$. Similarly, define $\Lambda_U := \Lambda \vert_U$ and $\check{\Lambda}_U := \check{\Lambda} \vert_U$.
	\item When $x = (x_1, ..., x_{2n})$ are integral affine coordinates on $B$, we denote by
	\begin{equation*}
		(x, y) = (x_1, ..., x_{2n}, y^1, ..., y^{2n}) \quad \text{and} \quad (x, \check{y}) = (x_1, ..., x_{2n}, \check{y}_1, ..., \check{y}_{2n})
	\end{equation*}
	the induced real coordinates on $T^*B$ and on $TB$ respectively.
	\item Let $\omega$ be the canonical symplectic form on $X$ given by $\omega = dx_i \wedge dy^i$.
	\item A $k$-tuple of complex valued functions on a manifold is identified with the corresponding column vector, and vice versa. If $a = (a_1, ..., a_k)$ and $b = (b_1, ..., b_k)$ are such $k$-tuples, then we denote their dot product by $a \cdot b$, i.e.
	\begin{equation*}
		a \cdot b := a_1b_1 + \cdots + a_k b_k.
	\end{equation*}
	If $\alpha = (\alpha_1, ..., \alpha_k)$ and $\beta = (\beta_1, ..., \beta_k)$ are $k$-tuples of $1$-forms, then we define
	\begin{equation*}
		\alpha \wedge \beta := \alpha_1 \wedge \beta_1 + \cdots + \alpha_k \wedge \beta_k.
	\end{equation*}
	\item For a $k \times k$ matrix $A$, $A^T$ denotes the transpose of $A$ and $A^{-T} := (A^T)^{-1}$. For a $k$-tuple $a$, $Aa$ denotes the matrix multiplication of $A$ and $a$ whenever it is well defined.
	\item If $a = (a_1, ..., a_{2n})$, then we write
	\begin{equation*}
		\boldsymbol{a}_1 = (a_1, ..., a_n) \quad \text{and} \quad \boldsymbol{a}_2 = (a_{n+1}, ..., a_{2n}).
	\end{equation*}
	The notations $\boldsymbol{a}^1$ and $\boldsymbol{a}^2$ are similarly defined for $a = (a^1, ..., a^{2n})$.
	\item Denote by $\mathbb{H}_n$ the \emph{Segal upper half space}, i.e. the set of symmetric $n \times n$ matrices $\Omega$ over $\mathbb{C}$ such that $\operatorname{Im} \Omega$ is positive definite.
\end{itemize}

\subsection*{Acknowledgement}
\quad\par
The work of K. Chan was substantially supported by grants of the Hong Kong Research Grants Council (Project No. CUHK14305023, CUHK14302524, CUHK14310425). The work of N. C. Leung was substantially supported by grants of the Hong Kong Research Grants Council (Project No. 14305923, 14302224, 14302225). The work of Q. Li was supported by grants from National Natural Science Foundation of China (Project No. 12471061). The work of Y.-H. Suen was supported by the National Science and Technology Council (Project No. 113-2115-M-006-016-MY2) and partially supported by the Yushan Fellow Program by the Ministry of Education (MOE), Taiwan. (MOE-114-YSFMS-0005-001-P1). Y. Yau was supported by World Premier International Research Center Initiative (WPI), MEXT, Japan, and gratefully acknowledges the support of the National Center for Theoretical Sciences, where part of this work was carried out.\par
The authors thank Si Li for helpful discussions and interest in this work.

\section{Preliminaries}
\label{Section: Preliminaries}

In this section, we recall the notion of coisotropic A-branes and the semi-affine condition introduced by K.-L. Chan--Leung--Zhang \cite{ChaLeuZha2018} and give a rough idea on what the mirror B-brane should be. Such an idea will be made rigorous in Section \ref{Section: The mirror B-brane via Fourier--Mukai type transforms}.

\subsection{A- and B-branes}
\quad\par
\label{Subsection: Semi-affine A- and B-branes}

In this paper, we restrict our attention to the definition of \emph{rank-$1$} coisotropic A-branes. For a more general treatment involving higher-rank cases, we refer the reader to \cite{ChaLeuZha2018, GukKorNawPeiSab2023, Her2012, LeuXieYau2025}.

\begin{definition}\label{definition : coisotropic A brane}
	Let  $(X,\omega)$ be a symplectic manifold. A coisotropic \emph{A-brane} is a coisotropic submanifold $C\subset X$ together with a $U(1)$-line bundle $(L,\nabla^L)$ on $C$ such that
	\begin{itemize}
		\item [(a)] The curvature $-2\pi\sqrt{-1} F_{\nabla^L}$ vanishes on $\ker(\omega|_C)$. In particular, the $2$-form $F_{\nabla^L}$ descends to a map $F_{\nabla^L}:TC/\ker(\omega|_C)\to T^*C$.
		\item [(b)] The composition $I:=\omega^{-1}F_{\nabla^L}:TC/\ker(\omega|_C) \to TC/\ker(\omega|_C)$ defines a complex structure, i.e. $I^2=-\operatorname{Id}$.
	\end{itemize}
\end{definition}

\begin{remark}
	The complex structure $I$ on $TC/\ker(\omega|_C)$ associated with $(C, L, \nabla^L)$ is always integrable \cite{KapOrl2003}. When $C=L\subset X$ is a Lagrangian, $\ker(\omega|_L)=TL$, so the brane condition boils down to the well-known condition for a Lagrangian brane, i.e. $F_{\nabla^L}=0$.
\end{remark}

Specifically, by Condition (b) in Definition \ref{definition : coisotropic A brane}, the brane structure $(L,\nabla^L)$ of a \emph{space-filling} (i.e. $C=X$) coisotropic A-brane turns $X$ into a complex manifold with complex structure $I:=\omega^{-1}F_{\nabla^L}$. We use $\mathcal{O}_X$ to denote the structure sheaf corresponding to this complex structure. The 2-form
$$\Omega_X:=F_{\nabla^L}+\sqrt{-1}\omega$$
is an $I$-holomorphic 2-form on $X$, turning $X$ into a holomorphic symplectic manifold. Conversely, given a holomorphic symplectic manifold $(X,\Omega_X)$ so that $\operatorname{Im}(\Omega_X)$ is non-degenerate and $-2\pi\sqrt{-1} \operatorname{Re}(\Omega_X)$ is the curvature of some Hermitian line bundle $(L,\nabla^L)$ on $X$, $\mathcal{B}_{\operatorname{cc}} = (X,L,\nabla^L)$ defines a space-filling coisotropic A-brane, called the \emph{canonical coisotropic brane on $(X,\operatorname{Im}(\Omega_X))$}.

Coisotropic A-branes show up at very specific points in the K\"ahler moduli. For example, if we scale $\omega$ by an arbitrary constant $k$, Condition (b) in Definition \ref{definition : coisotropic A brane} will no longer be fulfilled by any $U(1)$-line bundle because $F_{\nabla^L}$ needs to be integral. However, when $k\in\mathbb{Z}_{>0}$, we do have the canonical coisotropic brane $\mathcal{B}_{\operatorname{cc}}^{(k)} = (X,L^{\otimes k},\nabla^{L^{\otimes k}})$ with respect to the scaled symplectic manifold $(X,k\omega)$.

Kapustin--Orlov \cite{KapOrl2003} suggested that non-Lagrangian coisotropic A-branes should also be included as objects in the Fukaya category of $(X,\omega)$. In particular, by mirror symmetry, every coisotropic A-brane should admit a mirror B-brane. In general, B-branes on a complex manifold are described as objects of its derived category of coherent sheaves. For the purposes of this paper, however, we restrict attention to a more concrete class of B-branes that is compatible with the SYZ framework:

\begin{definition}
	Let $\check{X}$ be a complex manifold. A \emph{B-brane} is a complex submanifold $\check{C}\subset\check{X}$ together with a Hermitian holomorphic vector bundle $\check{E}$ on $\check{C}$.
\end{definition}

\subsection{Semi-affine branes}
\quad\par
\label{Subsection: Semi-flat SYZ picture}
Let $p:(X, \omega)\to B$ be a Lagrangian torus fibration without singular fibers, which in particular induces an integral affine structure on the base $B$. Such a fibration is often called \emph{semi-flat}, reflecting the fact that its fibers are affine tori; although one can equip $X$ with a compatible Riemannian metric making these fibers flat, no such metric will be needed in this paper. Assume moreover that $p:X\to B$ admits a Lagrangian section.\par
Fix a parameter $k \in \mathbb{Z}_{>0}$, and consider the A-model of $(X, k\omega)$ together with its mirror B-model. It is then well-known that $(X,k\omega)$ can be identified with $(T^*B/T_\mathbb{Z}^*B,k\omega_{std})$, where
$$\omega_{std}=\textstyle\sum_{i=1}^{2n}dx_i\wedge dy^i.$$
Let $\check{X}:=TB/T_\mathbb{Z}B$ and $\check{p}:\check{X}\to B$ be the dual torus fibration. The total space $\check{X}$ admits a complex structure $\check{J}_k$, which has local complex coordinates
\begin{equation*}
	\check{z}^{(k)} = (\check{z}_1^{(k)}, ..., \check{z}_{2n}^{(k)}) := x + \tfrac{\sqrt{-1}}{k} \check{y}.
\end{equation*}
This can be verified by observing how the pure spinors are transformed by the \emph{Poincar\'e bundle} $\mathcal{P}$ on $X\times_B\check{X}$. The complex manifold $\check{X}^{(k)} := (\check{X},\check{J}_k)$ is called the \emph{SYZ mirror of} $(X,k\omega)$. Recall the following pullback diagram in the usual semi-flat SYZ setting:
\begin{center}
	\begin{tikzcd}
		& X \times_B \check{X} \ar[ld, "\pi"'] \ar[rd, "\check{\pi}"]\\
		X \ar[rd, "p"'] & & \check{X} \ar[ld, "\check{p}"]\\
		& B
	\end{tikzcd}
\end{center}

In what follows, take $k = 1$ as an illustration. The SYZ philosophy says that any symplecto-geometric object on $X$ is transformed to a complex-geometric object on $\check{X}$ via the ``pullback-tensor-pushforward" construction. For instance, if $L\subset X$ is a Lagrangian section equipped with a flat $U(1)$-bundle $\mathcal{L}$, then the SYZ transform of $L$ is the line bundle
$$\check{L}:=\check{\pi}_*(\pi^*\mathcal{L}\otimes\mathcal{P}).$$
In other words, the fiber of $\check{L}$ at the point $(x,\check{y})\in\check{X}$ is given by
$$H^0(L\cap F_x,\mathcal{L}\otimes\mathcal{P}|_{(L\cap F_x)\times\{\check{y}\}})=\mathcal{L}_x\otimes\mathcal{P}_{(p_x,\check{y})},$$
where $L\cap F_x=\{p_x\}$. See \cite{LeuYauZas2000, ChaSue2020, CS_SYZ_immersed} for more details and applications. K.-L. Chan--Leung--Zhang \cite{ChaLeuZha2018} applied the same idea to construct the mirror B-brane of a so-called \emph{semi-affine} coisotropic A-brane.

\begin{definition}
	\label{Definition: semi-affine A-brane}
	A (rank-$1$) coisotropic A-brane $(C,L,\nabla^L)$ on $(X, \omega)$ is called \emph{semi-affine} if
	\begin{itemize}
		\item [(a)] The restriction $p \vert_C: C \to p(C)$ is a torus bundle such that for any $x\in B$, $C\cap F_x$ is an affine subtorus of the fiber $F_x\subset X$ over $x$.
		\item [(b)] The curvature of $\nabla^L$ is a constant $2$-form along each fiber, that is, if $(y^i)$ are affine coordinates of $C\cap F_x$, then $F_{\nabla^L}|_{C\cap F_x}=\sum_{i,j}h_{ij}dy^i\wedge dy^j$ for some integers $h_{ij}\in\mathbb{Z}$.
	\end{itemize}
\end{definition}

On the mirror side, the semi-affine condition becomes the following

\begin{definition}
	A B-brane $(\check{C},\check{E},\nabla^{\check{E}})$ on $\check{X}$ is called \emph{semi-affine} if
	\begin{itemize}
		\item [(a)] The restriction $\check{p} \vert_{\check{C}}: \check{C} \to \check{p}(\check{C})$ is a torus bundle such that for any $x\in B$, $\check{C}\cap\check{F}_x$ is an affine subtorus of the fiber $\check{F}_x$ of $\check{p}:\check{X}\to B$.
		\item [(b)] The curvature of $\nabla^{\check{E}}$ is a constant 2-form multiple of the identity map along each fiber.
	\end{itemize}
\end{definition}

K.-L. Chan--Leung--Zhang \cite{ChaLeuZha2018} established a correspondence
$$\{\text{Semi-affine A-branes}\}\longleftrightarrow\{\text{Semi-affine B-branes}\}$$
via family Nahm transform, which involves the choice of a fiberwise spinor bundle. For simplicity, we recall only their construction of the mirror B-brane associated with the coisotropic A-brane $\mathcal{B}_{\operatorname{cc}}=(X,L,\nabla^L)$ under our standing assumption \ref{Setup}. In this case, the mirror B-brane is again space-filling, and the fiber of its associated vector bundle $\check{E}$ over a point $(x,\check{y})\in\check{X}$ is given by
$$\ker(\slashed{D}_{(x,\check{y})}:\Gamma(F_x,L|_{F_x} \otimes\mathcal{P}|_{F_x\times\{\check{y}\}}\otimes\slashed{S}_x)\to\Gamma(F_x,L|_{F_x} \otimes\mathcal{P}|_{F_x\times\{\check{y}\}}\otimes\slashed{S}_x)),$$
where $\slashed{S}_x$ is a spinor bundle on $F_x$ and $\slashed{D}_{(x,\check{y})}$ is the associated Dirac operator.

Our construction of the mirror B-brane is a slight modification of this approach: we replace the spinor bundle by the half-form bundle $\sqrt{K}$, defined with respect to a choice of fiberwise torus-invariant K\"ahler polarizations. With such a choice, the line bundle $L|_{F_x} \otimes \mathcal{P} \vert_{F_x \times \{\check{y}\}}$ may be viewed as a pre-quantum line bundle on the K\"ahler manifold $(F_x,F_{\nabla^L}|_{F_x})$, becuase the Poincar\'e bundle $\mathcal{P}$ is flat along $F_x \times \{\check{y}\}$. Accordingly, the Dirac operator is replaced by the fiberwise $\bar{\partial}$-operator and the fiber $\check{E}_{(x, \check{y})}$ is given by
$$H^0(F_x,L|_{F_x}\otimes\mathcal{P}|_{F_x\times\{\check{y}\}}\otimes\sqrt{K}).$$
This is precisely where geometric quantization enters the picture. Analogous constructions apply for $k > 1$. We refer the reader to Section \ref{Section: The mirror B-brane via Fourier--Mukai type transforms} for further details.

\subsection{Quotient bundles by fiberwise factors of automorphy}
\quad\par
\label{Subsection: Quotient bundles by fibrewise factors of automorphy}
Consider our setup stated in Subsection \ref{Subsection: Setup and notations}. In the following (sub)sections we repeatedly use a standard construction to describe line bundles over $X_U$ and $X_U \times_U \check{X}_U$ for an open subset $U \subset B$.\par
Consider $m \in \mathbb{Z}_{>0}$ and the trivial Hermitian line bundle $(U \times \mathbb{R}^m) \times \mathbb{C}$ over $U \times \mathbb{R}^m$, equipped with a unitary connection $\nabla$. Suppose that the bundle of groups $U \times \mathbb{Z}^m$ over $U$ acts smoothly on $(U \times \mathbb{R}^m) \times \mathbb{C}$ via
\begin{equation*}
	\bullet : (U \times \mathbb{Z}^m) \times (U \times \mathbb{R}^m) \times \mathbb{C} \longrightarrow (U \times \mathbb{R}^m) \times \mathbb{C},
\end{equation*}
covering the natural translation action of $U \times \mathbb{Z}^m$ on $U \times \mathbb{R}^m$, such that $\nabla$ is $\bullet$-equivariant. We refer to such an action $\bullet$ as a \emph{fiberwise factor of automorphy}. By descending $((U \times \mathbb{R}^m) \times \mathbb{C}, \nabla)$ via $\bullet$, we obtain a Hermitian line bundle with unitary connection over $U \times \mathbb{T}^m$, where $\mathbb{T}^m := \mathbb{R}^m / \mathbb{Z}^m$. This is called the \emph{quotient bundle} on $U \times \mathbb{T}^m$ \emph{associated with $\nabla$ and $\bullet$}.\par
As an important example, the Poincar\'e bundle $(\mathcal{P}, \nabla^\mathcal{P})$ on $X \times_B \check{X}$ arises from this construction. Let $U \subset B$ be equipped with integral affine coordinates $x$, so that $X_U \times_U \check{X}_U$ is identified with $U \times \mathbb{R}^{4n}$ via coordinates $(x, y, \check{y})$. On $U \times \mathbb{R}^{4n}$ we consider the trivial Hermitian line bundle endowed with the unitary connection
\begin{equation*}
	d - 2\pi\sqrt{-1} \check{y} \cdot dy.
\end{equation*}
Define a fiberwise factor of automorphy by: for $(\lambda, \check{\lambda}) \in \mathbb{Z}^{2n} \times \mathbb{Z}^{2n}$,
\begin{equation*}
	(\lambda, \check{\lambda}) \bullet (x, y, \check{y}, v)
	:= \left(x, y + \lambda, \check{y} + \check{\lambda}, e^{2\pi\sqrt{-1} \check{\lambda} \cdot y} v \right).
\end{equation*}
The associated quotient bundle on $U \times \mathbb{T}^{4n}$ defines the local model of $(\mathcal{P}, \nabla^{\mathcal{P}})$ on $X_U \times_U \check{X}_U$. These quotient bundles glue together to yield the global Poincar\'e bundle $(\mathcal{P}, \nabla^{\mathcal{P}})$.

\section{Non-formal quantization on $\mathcal{B}_{\operatorname{cc}}^{(k)}$}
\label{Section: Holomorphic deformation quantization}
Equip $X$ with the holomorphic symplectic structure induced by $\mathcal{B}_{\operatorname{cc}}$ and let $\Omega_X$ denote the corresponding holomorphic symplectic form, i.e. $-2\pi\sqrt{-1} \cdot \operatorname{Re}\Omega_X$ is the curvature of $\nabla^L$ and $\operatorname{Im} \Omega_X = \omega$. Under our standing assumption \ref{Setup}, the main statement of this section is the existence of a canonical strict quantization of holomorphic functions on $X$ for each level $k$. In particular, the asymptotic expansion of these deformed products yields a holomorphic deformation quantization of the holomorphic symplectic manifold $(X, \Omega_X)$.\par
We briefly outline the section. Following \cite{ChaLeuZha2018}, Subsection \ref{Subsection: Skew-Smith forms and invariant factors} constructs distinguished local models for the bundle $(L, \nabla^L)$, referred to as \emph{skew-Smith forms} of $\mathcal{B}_{\operatorname{cc}}$. These local models show that the fibers of $p$, equipped with the restriction of $\operatorname{Re} \Omega_X$, are mutually symplectomorphic. As a consequence, the induced symplectic structures on the fibers are classified by a diagonal integer matrix $\boldsymbol{H}$, which we call the \emph{invariant factor matrix} of $\mathcal{B}_{\operatorname{cc}}$. In Subsection \ref{Subsection: holomorphic Moyal product}, $\boldsymbol{H}$ provides an explicit coordinate description of the holomorphic symplectic structure on $X$. This yields an atlas of complex coordinates with affine transitions, from which we construct a Moyal-type quantization. Subsection \ref{Subsection: Examples} presents examples of SYZ fibrations $p: X \to B$ (without singularities) admitting space-filling A-branes that satisfy the standing assumption \ref{Setup}.
Finally, in Subsection \ref{Subsection: proof of skew smith form}, we give a proof of the existence of skew-Smith forms of $\mathcal{B}_{\operatorname{cc}}$.

\subsection{Skew-Smith forms, invariant factors, and complex coordinates}
\quad\par
\label{Subsection: Skew-Smith forms and invariant factors}
The semi-affine A-brane $\mathcal{B}_{\operatorname{cc}} = (X, L, \nabla^L)$ admits an explicit local coordinate description. By Proposition 27 in \cite{ChaLeuZha2018}, every point in $B$ has a neighbourhood $U$ equipped with data
\begin{equation}
	\label{Equation: original data}
	(x, f, \widetilde{g}, H, \chi),
\end{equation}
where $x = (x_1, ..., x_{2n})$ are integral affine coordinates on $U$, $f, \widetilde{g} \in \mathcal{C}^\infty(U, \mathbb{R}^{2n})$, $H$ is a constant skew-symmetric $2n \times 2n$ matrix over $\mathbb{Z}$ and $\chi: \mathbb{Z}^{2n} \to \operatorname{U}(1)$ is an $H$-semi-character \footnote{That is, $\chi(\lambda + \lambda') = (-1)^{\lambda \cdot H \lambda'} \chi(\lambda) \chi(\lambda')$ for all $\lambda, \lambda' \in \mathbb{Z}^{2n}$.}. These data determine $(L, \nabla^L) \vert_{X_U}$ as the quotient bundle on $U \times \mathbb{T}^{2n}$ associated with
\begin{equation}
	\label{Equation: first connection for local A brane}
	d - 2\pi\sqrt{-1} \left( \tfrac{1}{2} f \cdot dx + \widetilde{g} \cdot dy + \tfrac{1}{2} y \cdot H dy \right)
\end{equation}
and the fiberwise factor of automorphy $\lambda \bullet (x, y, v) = (x, y + \lambda, \chi(\lambda) e^{\pi\sqrt{-1} \lambda \cdot H y} v)$.\par
In this subsection, we simplify this description and introduce local models for $\mathcal{B}_{\operatorname{cc}}$ that are better suited for the proof of our main results. These models are described in Proposition \ref{Proposition: skew smith form of A brane}.\par
For an open subset $U \subset B$ equipped with integral affine coordinates $x$ and an $\mathbb{R}^{2n}$-valued function $g \in \mathcal{C}^\infty(U, \mathbb{R}^{2n})$, denote by $Dg$ the Jacobian matrix of $g$ with respect to $x$. In what follows, we will frequently decompose a $2n$-tuple into two $n$-tuples, As described in Section \ref{Subsection: Setup and notations}, we denote the resulting components in boldface with appropriate indices. For example, we write $\boldsymbol{x}_1 = (x_1, ..., x_n)$ and $\boldsymbol{x}_2 = (x_{n+1}, ..., x_{2n})$.

\begin{proposition}
	\label{Proposition: skew smith form of A brane}
	Every point in $B$ admits a neighbourhood $U$ together with data
	\begin{equation*}
		(x, g, \check{g}, \boldsymbol{H}),
	\end{equation*}
	where $x = (x_1, ..., x_{2n})$ is an integral affine chart on $U$, $g \in \mathcal{C}^\infty(U, \mathbb{R}^{2n})$ satisfies $Dg = (Dg)^T$, $\check{g} \in \mathbb{R}^{2n}$ is constant, and $\boldsymbol{H} = \operatorname{diag}(h_1, ..., h_n)$ is an $n \times n$ diagonal matrix with $h_1, ..., h_n \in \mathbb{Z}_{>0}$ satisfying $h_1 \mid h_2 \mid \cdots \mid h_n$. With respect to these data, $(L, \nabla^L) \vert_{X_U}$ is isomorphic to the quotient bundle on $U \times \mathbb{T}^{2n}$ associated with the unitary connection
	\begin{equation*}
		d - 2\pi\sqrt{-1} ( \boldsymbol{x}_1 \cdot \boldsymbol{H}^{-1} d\boldsymbol{x}_2 + (\boldsymbol{y}^2 + \boldsymbol{g}^2) \cdot \boldsymbol{H} d(\boldsymbol{y}^1 + \boldsymbol{g}^1) ),
	\end{equation*}
	and the fiberwise factor of automorphy
	\begin{equation}
		\lambda \bullet (x, y, v) = (x, y + \lambda, e^{2\pi\sqrt{-1} (\check{g} \cdot \lambda + \boldsymbol{\lambda}^2 \cdot \boldsymbol{H} (\boldsymbol{y}^1 + \boldsymbol{g}^1))} v).
	\end{equation}
\end{proposition}

The proof is largely computational, involving a sequence of coordinate changes and gauge transformations based on the coordinate description introduced at the beginning of this subsection. We postpone the details to Subsection~\ref{Subsection: proof of skew smith form}.\par
As a consequence of Proposition \ref{Proposition: skew smith form of A brane}, the matrix $\boldsymbol{H}$ is an invariant of the brane $\mathcal{B}_{\operatorname{cc}}$.

\begin{proposition}
	\label{Proposition: invariant factor matrix}
	Let $U, U' \subset B$ be open subsets with $U \cap U' \neq \emptyset$, equipped respectively with the data $(x, g, \check{g}, \boldsymbol{H})$ and $(x', g', \check{g}', \boldsymbol{H}')$ as in Proposition \ref{Proposition: skew smith form of A brane}. Then
	\begin{equation*}
		\boldsymbol{H} = \boldsymbol{H}'.
	\end{equation*}
\end{proposition}
\begin{proof}
	The charts $x$ and $x'$ are related by the coordinate change $x = A ( x' + b )$ on $U \cap U'$ for some $A \in \operatorname{GL}(2n, \mathbb{Z})$ and $b \in \mathbb{R}^{2n}$.  The induced coordinate changes on $X$ are $y = A^{-T} y'$. Let
	\begin{equation*}
		H = \begin{pmatrix}
			0 & -\boldsymbol{H}\\
			\boldsymbol{H} & 0
		\end{pmatrix}
		\quad \text{and} \quad 
		H' = \begin{pmatrix}
			0 & -\boldsymbol{H}'\\
			\boldsymbol{H}' & 0
		\end{pmatrix}.
	\end{equation*}
	Comparing the coordinate expression for the restriction of the curvature of $\nabla^L$ stated in Proposition \ref{Proposition: skew smith form of A brane} onto a fiber of $X_{U \cap U'} \to U \cap U'$:
	\begin{align*}
		-\pi\sqrt{-1} dy\wedge H dy = -\pi\sqrt{-1} dy' \wedge H' dy',
	\end{align*}
	we obtain $H = A H' A^T$. The uniqueness of the skew-Smith norm form yields $\boldsymbol{H} = \boldsymbol{H}'$.
\end{proof}

This motivates the following definitions.

\begin{definition}
	A \emph{skew-Smith form} of $\mathcal{B}_{\operatorname{cc}}$ is a Hermitian line bundle with a unitary connection over $U \times \mathbb{T}^n$ of the form as in Proposition \ref{Proposition: skew smith form of A brane}.
\end{definition}

\begin{definition}
	The \emph{invariant factor matrix} of $\mathcal{B}_{\operatorname{cc}}$ is the diagonal matrix $\boldsymbol{H}$ given as in Proposition \ref{Proposition: skew smith form of A brane}. The \emph{invariant factors} of $\mathcal{B}_{\operatorname{cc}}$ are the diagonal entries of $\boldsymbol{H}$.
\end{definition}

Here and in the sequel, we always denote by $\boldsymbol{H}$ the invariant factor matrix of $\mathcal{B}_{\operatorname{cc}}$, and by $h_1, ..., h_n$ the invariant factors of $\mathcal{B}_{\operatorname{cc}}$ with $h_1 \mid h_2 \mid \cdots \mid h_n$.\par
We conclude this subsection by providing a coordinate description of the complex structure $I$ induced by $\mathcal{B}_{\operatorname{cc}}$.

\begin{proposition}
	\label{Proposition: complex coordinates induced by A brane}
	Let $(x, g, \check{g}, \boldsymbol{H})$ be data as in Proposition \ref{Proposition: skew smith form of A brane}. Then the functions
	\begin{equation}
		z := (z^1, ..., z^{2n}) = ( \boldsymbol{x}_1 - \sqrt{-1} \boldsymbol{H} (\boldsymbol{y}^2 + \boldsymbol{g}^2), \boldsymbol{x}_2 + \sqrt{-1} \boldsymbol{H} (\boldsymbol{y}^1 + \boldsymbol{g}^1) )
	\end{equation}
	form complex coordinates \footnote{Strictly speaking, $z$ defines complex coordinates on the covering space $T^*U$ of $X_U$, since the fiber coordinates $y$ are periodic. By abuse of notation, we still treat $z$ as coordinates on $X_U$.} on $X$ with respect to the complex structure $I$. Moreover, the holomorphic symplectic form $\Omega_X$ is locally given by
	\begin{equation}
		\label{Equation: coordinate description of holomorpihc symplectic form}
		\Omega_X = d\boldsymbol{z}^1 \wedge \boldsymbol{H}^{-1} d\boldsymbol{z}^2.
	\end{equation}
\end{proposition}
\begin{proof}
	Recall that $(x, y)$ are coordinates on $X_U$. Since $\boldsymbol{H}$ is non-degenerate, the functions $(z^1, \dots, z^{2n})$ define a system of complex coordinates on $X_U$, and hence determine a complex structure $I_U$ on it. Using Proposition \ref{Proposition: skew smith form of A brane}, we compute the curvature of $\nabla^L$. Together with the condition $(Dg)^T = Dg$, a direct calculation shows that $\Omega_X$ can be expressed in the form \eqref{Equation: coordinate description of holomorpihc symplectic form}. In particular, $\Omega_X$ is of type $(2,0)$ with respect to $I_U$. The non-degeneracy of $\Omega_X$ forces that $I_U$ is the restriction $I$ to the open subset $X_U \subset X$.
\end{proof}

\begin{corollary}
	\label{Corollary: affine change in complex coorindates}
	Let $U, U' \subset B$ be connected open subsets with $U \cap U' \neq \emptyset$, equipped with the data $(x, g, \check{g}, \boldsymbol{H})$ and $(x', g', \check{g}', \boldsymbol{H}')$ respectively, as in Proposition \ref{Proposition: skew smith form of A brane}. Suppose $A \in \operatorname{GL}(2n, \mathbb{Z})$ and $b \in \mathbb{R}^{2n}$ define the transition between integral affine charts on $B$ via
	\begin{equation*}
		x = A(x' + b).
	\end{equation*}
	Then there exists $c \in \mathbb{R}^{2n}$ such that the induced change of the corresponding complex coordinates $z, z'$ as described in Proposition \ref{Proposition: complex coordinates induced by A brane} is given by the affine transformation
	\begin{equation*}
		z = A(z' + b + \sqrt{-1} c).
	\end{equation*}
\end{corollary}
\begin{proof}
	Recall that $H = AHA^T$, where $H$ is given by \ref{Equation: Skew-Smith form of H}. Comparing the mixed terms in the coordinate expression for the curvature of $\nabla^L$ given in Proposition \ref{Proposition: skew smith form of A brane}, we deduce that $d(g(x)) = d(A^{-T} g'(x'))$. Using the relation $y = A^{-T} y'$, together with the definitions of $z$ and $z'$ from Proposition \ref{Proposition: complex coordinates induced by A brane}, we then obtain the desired affine transformation.
\end{proof}

\subsection{A non-formal quantization of the holomorphic symplectic manifold $(X, \Omega_X)$}
\quad\par
\label{Subsection: holomorphic Moyal product}
We begin by outlining the main constructions in this subsection. Since the coordinate changes in $(x, y)$ are affine and compatible with the matrix $\boldsymbol{H}$, the following local expressions for smooth functions $f, g \in \mathcal{C}^\infty(X)$ patches together to define a global Poisson bracket on $X$:
\begin{equation*}
	\{f, g\} := \frac{1}{2\pi} \sum_{i=1}^n h_i^{-1} \left( \frac{\partial f}{\partial y^i} \frac{\partial g}{\partial y^{n+i}} - \frac{\partial f}{\partial y^{n+i}} \frac{\partial g}{\partial y^i} \right).
\end{equation*}
We may regard $\{ \,\cdot\,, \,\cdot\, \}$ as a smooth family of translation-invariant Poisson brackets on $2n$-dimensional symplectic tori, parametrized by the base manifold $B$. Using this structure, we obtain a corresponding family of quantum tori \cite{Rie1989}, realized as a non-formal deformation $\star_\hbar$ of the commutative product on $\mathcal{C}^\infty(X)$ for $\hbar \in \mathbb{R}$. A key feature of the deformation $\star_\hbar$ is that, although it is nonlocal in the fiber directions of the SYZ fibration $X \to B$, it remains \emph{local along the base}: if either $f$ or $g$ vanishes on $X_U$ for some open subset $U \subset B$, then $f \star_\hbar g$ also vanishes on $X_U$.\par
Next, observe that the holomorphic coordinate vector fields on $X$ are given by
\begin{align}
	\label{Equation: first complex coordinate derivative}
	\tfrac{\partial}{\partial z^i} = \tfrac{1}{2} \left( \tfrac{\partial}{\partial x_i} - \textstyle\sum_{j=1}^n \tfrac{\partial g^{n+j}}{\partial x_i} \tfrac{\partial}{\partial y^{n+j}} + \sqrt{-1} h_i^{-1} \tfrac{\partial}{\partial y^{n+i}} \right),\\
	\label{Equation: second complex coordinate derivative}
	\tfrac{\partial}{\partial z^{n+i}} = \tfrac{1}{2} \left( \tfrac{\partial}{\partial x_{n+i}} - \textstyle\sum_{j=1}^n \frac{\partial g^i}{\partial x_{n+i}} \frac{\partial}{\partial y^j} -\sqrt{-1} h_i^{-1} \frac{\partial}{\partial y^i} \right),
\end{align}
for $1 \leq i \leq n$. Consequently, for local holomorphic functions $f, g$ on $X$, the bracket becomes
\begin{equation*}
	\{f, g\} = \frac{1}{4\pi} \sum_{i=1}^n h_i \left( \frac{\partial f}{\partial z^{n+i}} \frac{\partial g}{\partial z^i} - \frac{\partial f}{\partial z^i} \frac{\partial g}{\partial z^{n+i}} \right),
\end{equation*}
which coincides with the holomorphic Poisson bracket induced by the holomorphic symplectic form $\Omega_X = d\boldsymbol{z}^1 \wedge \boldsymbol{H}^{-1} d\boldsymbol{z}^2$. We further show that the product $\star_\hbar$ naturally extends to the space $\Omega^{0, *}(X)$ of $(0, *)$-forms in such a way that the Dolbeault operator $\overline{\partial}$ continues to act as a graded derivation. In particular, the subspace $\mathcal{O}(X)$ of holomorphic functions is closed under $\star_\hbar$. The resulting subalgebra $\mathcal{A}_\hbar(X) := (\mathcal{O}(X), \star_\hbar)$ yields a non-formal quantization of the holomorphic symplectic manifold $(X, \Omega_X)$.\par
We now turn to the details of the construction.

\subsubsection{Family of quantum tori}
\quad\par
Fix $\hbar \in \mathbb{R}$. We begin by working over an open subset $U \subset B$ for which $\mathcal{B}_{\operatorname{cc}}$ admits a skew-Smith form. Let $(x, y)$ (resp. $z$) be the corresponding real (resp. complex) coordinates on $X_U$. Let $f, g \in \mathcal{C}^\infty(X_U)$ be smooth functions, and consider their fiberwise Fourier expansions:
\begin{equation*}
	f(x, y) = \sum_{m \in \mathbb{Z}^{2n}} \widehat{f}_m(x) e^{2\pi \sqrt{-1} m \cdot y}, \quad g(x, y) = \sum_{m \in \mathbb{Z}^{2n}} \widehat{g}_m (x) e^{2\pi\sqrt{-1} m \cdot y}.
\end{equation*}
Following the standard construction of quantum tori via twisted convolution products of Fourier modes, we define a binary operation
\begin{equation*}
	\star_\hbar: \mathcal{C}^\infty(X_U) \times \mathcal{C}^\infty(X_U) \to \mathcal{C}^\infty(X_U), \quad (f, g) \mapsto f \star_\hbar g,
\end{equation*}
by
\begin{equation}
	\label{Equation: quantum tori}
	(f \star_\hbar g)(x, y) := \sum_{q \in \mathbb{Z}^{2n}} \left( \sum_{m + m' = q} e^{-\hbar\pi\sqrt{-1} m \cdot H^{-1} m'} \widehat{f}_m(x) \widehat{g}_{m'}(x) \right) e^{2\pi\sqrt{-1} q \cdot y},
\end{equation}
where the bicharacter $(m, m') \mapsto m \cdot H^{-1} m'$ on $\mathbb{Z}^{2n}$ with $H$ given in \eqref{Equation: Skew-Smith form of H} is obtained from the fiberwise Fourier transform of the Poisson bracket $\{ \,\cdot\, ,\,\cdot\, \}$.\par
Equip $\mathcal{C}^\infty(X_U)$ with its compact-open $\mathcal{C}^\infty$-topology. While Rieffel \cite{Rie1993} utilized this construction to develop a framework for \emph{strict deformation quantization} in the context of $C^*$-algebras, our interests lie in the differential geometric properties of the smooth algebra. For our purposes, it is sufficient to view the one-parameter family $\{ (\mathcal{C}^\infty(X_U), \star_\hbar) \}_{\hbar \in \mathbb{R}}$ as a deformation quantization of the Fr\'echet Poisson algebra $(\mathcal{C}^\infty(X_U), \{\,\cdot\,,\,\cdot\,\})$ in the sense of Omori--Maeda--Miyazaki--Yoshioka \cite{OmoHaeMiyYos2000}.\par
Indeed, by employing seminorm estimates on the rapidly decreasing fiberwise Fourier coefficients (cf. \cite{Rie1989, Rie1993}), one can show that for each $\hbar \in \mathbb{R}$, $(\mathcal{C}^\infty(X_U), \star_\hbar)$ is a Fr\'echet algebra where the product $\star_\hbar$ is jointly continuous. Furthermore, for all $f, g \in \mathcal{C}^\infty(X_U)$, the following conditions are satisfied:
\begin{enumerate}
	\item The map $\mathbb{R} \to \mathcal{C}^\infty(X_U)$ given by $\hbar \mapsto f \star_\hbar g$ is continuous, and $f \star_0 g = fg$.
	\item As $\hbar \to 0$, $\tfrac{1}{\hbar} (f \star_\hbar g - g \star_\hbar f) \to \sqrt{-1} \{f, g\}$.
\end{enumerate}

Using the following basis of $\Omega^{0, *}(X_U)$ over $\mathcal{C}^\infty(X_U)$: 
\begin{equation*}
	\{ dz^{i_1} \wedge \cdots \wedge dz^{i_l} : 0 \leq l \leq 2n, 1 \leq i_1 < \cdots < i_l \leq 2n\}
\end{equation*}
We extend $\star_\hbar$ to $\Omega^{0, *}(X)$ in a natural way: for all $f, g \in \mathcal{C}^\infty(X_U)$ and indices $i_1, ..., i_a, j_1, ..., j_b$,
\begin{equation*}
	\left( f dz^{i_1} \wedge \cdots \wedge dz^{i_a}\right) \star_\hbar \left( g dz^{j_1} \wedge \cdots \wedge dz^{j_b}\right) := (f \star_\hbar g) dz^{i_1} \wedge \cdots \wedge dz^{i_a} \wedge dz^{j_1} \wedge \cdots \wedge dz^{j_b}.
\end{equation*}
Using the same basis, we extend the Poisson bracket $\{ \,\cdot\,, \,\cdot\, \}$ to $\Omega^{0, *}(X_U)$ analogously.\par
These constructions are compatible with restrictions to smaller open subsets. More precisely, if $U' \subset U$ and $\alpha, \beta \in \Omega^{0, *}(X_U)$, then
\begin{equation*}
	(\alpha \star_\hbar \beta) \vert_{X_{U'}} = (\alpha \vert_{X_{U'}}) \star_\hbar (\beta \vert_{X_{U'}}) \quad \text{and} \quad \{\alpha, \beta\} \vert_{X_{U'}} = \{\alpha \vert_{X_{U'}}, \beta \vert_{X_{U'}}\},
\end{equation*}
where the right-hand sides are defined using the restricted coordinates on $X_{U'}$.\par
We also verify that these constructions are independent of the choice of local coordinates. Suppose $(x', y')$ is another coordinate system on $X_U$ arising from a different skew-Smith form of $\mathcal{B}_{\operatorname{cc}}$. Then the fiber coordinates transform as $y = A^{-T}y'$ for some constant matrix $A \in \operatorname{GL}(2n, \mathbb{Z})$ satisfying $AHA^T = H$. By Corollary \ref{Corollary: affine change in complex coorindates}, the corresponding $(0, 1)$-forms transform as $d\overline{z} = A d\overline{z}'$. It follows from these transformation rules that both $\star_\hbar$ and $\{ \,\cdot\,,\,\cdot\,\}$ are invariant under such coordinate changes. Consequently, they are well defined on $\Omega^{0, *}(X_U)$ for \emph{any} open subset $U \subset B$.

\begin{theorem}
	\label{Theorem: non-commutative Dolbeault complex}
	For each $\hbar \in \mathbb{R}$, the assignment
	\begin{equation*}
		U \to (\Omega^{0, *}(X_U), \star_\hbar, \overline{\partial})
	\end{equation*}
	defines a sheaf of differential graded algebras on $B$. Moreover, for every open subset $U \subset B$ and all $\alpha, \beta \in \Omega^{0, *}(X_U)$, the following convergences hold in the compact-open $\mathcal{C}^\infty$-topology:
	\begin{enumerate}
		\item $\alpha \star_\hbar \beta \to \alpha \star_0 \beta = \alpha \wedge \beta$ as $\hbar \to 0$;
		\item $\tfrac{1}{\hbar} [\alpha, \beta]_{\star_\hbar} \to \sqrt{-1} \{\alpha, \beta\}$ as $\hbar \to 0$, where $[\,\cdot\,,\,\cdot\,]_{\star_\hbar}$ denotes the graded commutator of $\star_\hbar$.
	\end{enumerate}
\end{theorem}

An immediate consequence of the theorem is that, for any open subset $U \subset B$, the space $\mathcal{O}(X_U)$ forms a $\mathbb{C}$-subalgebra of $(\mathcal{C}^\infty(X_U), \star_\hbar)$. We denote this subalgebra by $\mathcal{A}_\hbar(X_U)$.

\begin{remark}
	Although $\mathcal{A}_\hbar(X)$ is non-commutative for sufficiently small $\hbar$, with semiclassical behaviour $\frac{1}{\hbar} [f, g]_{\star} \to \sqrt{-1} \{f, g\}$ as $\hbar \to 0$,
	it need not be non-commutative for all values of $\hbar$. For instance, when $\hbar = 1$ and $\boldsymbol{H}$ is the identity matrix, the phase factor $e^{-\pi\sqrt{-1}\, m \cdot m'}$ in \eqref{Equation: quantum tori} is symmetric in $m$ and $m'$, and hence the resulting product is commutative.
\end{remark}

\begin{proof}[\myproof{Theorem}{\ref{Theorem: non-commutative Dolbeault complex}}]
	As explained above, the assignment $U \mapsto (\Omega^{0,*}(X_U), \star_\hbar)$ defines a sheaf of graded algebras on $B$. It therefore suffices to argue locally on the base. Accordingly, fix an open subset $U \subset B$ equipped with coordinates arising from a skew-Smith form of $\mathcal{B}_{\operatorname{cc}}$.\par
	To prove that $\overline{\partial}$ is a graded derivation of $(\Omega^{0,*}(X_U), \star_\hbar)$, it suffices to verify the Leibniz rule on functions. More precisely, for all $f,g \in \mathcal{C}^\infty(X_U)$ and each coordinate $\overline{z}^i$, we must show
	\begin{equation}
		\label{Equation: Leibniz rule for zbar}
		\frac{\partial}{\partial \overline{z}^i}(f \star_\hbar g)
		=
		\left(\frac{\partial f}{\partial \overline{z}^i}\right)\star_\hbar g
		+
		f \star_\hbar \left(\frac{\partial g}{\partial \overline{z}^i}\right),
		\quad 1 \leq i \leq 2n.
	\end{equation}
	
	To this end, observe that the product $\star_\hbar$ is invariant under translations along the fibers. Furthermore, by its locality in the base direction, $\star_\hbar$ is $\mathcal{C}^\infty(U)$-bilinear, where $\mathcal{C}^\infty(U)$ acts on $\mathcal{C}^\infty(X_U)$ via pullback. It follows that the Leibniz rule for $\star_\hbar$ holds for coordinate vector fields $\tfrac{\partial}{\partial x^i}$, as well as for vector fields of the form $f \tfrac{\partial}{\partial y^j}$, where $f$ is the pullback of a smooth function on $U$. Since each $\frac{\partial}{\partial \overline{z}^i}$ can be expressed as a linear combination of such vector fields, \eqref{Equation: Leibniz rule for zbar} follows.\par
	For open sets $U \subset B$ of the above form, the family $\{(\mathcal{C}^\infty(X_U), \star_\hbar)\}_{\hbar \in \mathbb{R}}$ defines a deformation quantization of the Fr\'echet Poisson algebra $(\mathcal{C}^\infty(X_U), \{\,\cdot\,,\,\cdot\,\})$ in the sense of Definition~1.1 in \cite{OmoHaeMiyYos2000}. In particular, for all $f,g \in \mathcal{C}^\infty(X_U)$, we have
	\begin{itemize}
		\item $f \star_\hbar g \to f \star_0 g = fg$ as $\hbar \to 0$;
		\item $\tfrac{1}{\hbar}[f,g]_{\star_\hbar} \to \sqrt{-1} \{f,g\}$ as $\hbar \to 0$.
	\end{itemize}
	
	By construction, the product $\star_\hbar$ extends to $\Omega^{0,*}(X_U)$ in a manner compatible with the wedge product and the grading. The convergence statements (1) and (2) therefore extend immediately from functions to arbitrary $(0, *)$-forms $\alpha,\beta \in \Omega^{0,*}(X_U)$. Finally, since these statements are formulated in the compact-open $\mathcal{C}^\infty$-topology, they globalize to arbitrary open subsets of $B$.
\end{proof}

\begin{remark}
	We give a string-theoretic insight of the above star product. The fibrewise non-degeneracy of the curvature of $\nabla^L$, as assumed in Assumption \ref{Setup}, provides an identification between the lattice bundles on A- and B-sides. Together with the fiberwise Fourier transform, we can identify holomorphic functions on $X$ as a subclass of rapidly decreasing smooth functions on the lattice bundle $\Lambda$ on the A-side, exchanging the multiplication and the convolution product. As $\Lambda$ is canonically diffeomorphic to the space of fiberwise geodesic loop space of $X$ with base point in the zero section, physically speaking, our construction of a deformed product of holomorphic functions on $X$ can be regarded as a deformation of the product on a based loop space obtained from concatenation of strings. See related ideas in \cite{ChaSue2020, ChaLeuMa2012, KapOrl2004}.
\end{remark}

\subsection{Examples}
\quad\par
\label{Subsection: Examples}
We present three kinds of examples of coisotropic A-branes satisfying Assumption \ref{Setup}.

\subsubsection{$\mathbb{R}^2 \times \mathbb{T}^2$}
\quad\par
\label{Subsubsection: transcendental mirror symmetry}
Let's consider $B=\mathbb{R}^2$ with standard affine structure. In this case, $X\cong(\mathbb{C}^{\times})^2$. Consider the $U(1)$-line bundle $(L,\nabla^L)$ obtained by the connection
$$d-2\pi\sqrt{-1}(x_1dx_2+y^2dy^1)$$
and the fiberwise factor of automorphy
$$\lambda\bullet(x,y,v)=(x,y+\lambda,e^{2\pi\sqrt{-1}\lambda_1y^2}v),$$
for $\lambda=(\lambda_1,\lambda_2)$. The curvature of $(L,\nabla^L)$ is given by
$$-2\pi\sqrt{-1} F_{\nabla^L}=2\pi\sqrt{-1} (dx_1\wedge dx_2-dy^1\wedge dy^2)$$
and $I=\omega^{-1}F_{\nabla^L}$ has complex coordinates
$$z^1:=x_1-\sqrt{-1}y^2,z^2:=x_2+\sqrt{-1}y^1.$$
This gives a canonical coisotropic brane $\mathcal{B}_{\operatorname{cc}}=(X,L,\nabla^L)$ on $X$ satisfying Assumption \ref{Setup}. The invariant factor matrix of this brane is $\boldsymbol{H} = \begin{pmatrix}
	1
\end{pmatrix}$.

\subsubsection{The $4$-torus $\mathbb{T}^4$ and the Kodaira--Thurston manifold}
\label{Subsubsection: Kodaira--Thurston manifold}
\quad\par
Consider $B = \mathbb{T}^2 := \mathbb{R}^2 / \mathbb{Z}^2$. The standard periodic coordinates $\widetilde{x} = (\widetilde{x}_1, \widetilde{x}_2)$ on $B$, together with an integral affine coordinate change $\widetilde{x}' = \widetilde{x} + m$ for each $m \in \mathbb{Z}^2$, define the standard integral affine structure on $B$. Let $a \in \mathbb{N}$. We consider another integral affine structure on $B$ which is determined by the new coordinates
\begin{equation*}
	(x_1, x_2) = (\widetilde{x}_1, \widetilde{x}_2 + a (\widetilde{x}_1)^2)
\end{equation*}
and an integral affine coordinate change for each $(m_1, m_2) \in \mathbb{Z}^2$:
\begin{equation*}
	\begin{pmatrix}
		x'_1\\
		x'_2
	\end{pmatrix}
	=
	\begin{pmatrix}
		1 & 0\\
		2am_1 & 1
	\end{pmatrix}
	\begin{pmatrix}
		x_1\\
		x_2
	\end{pmatrix}
	+
	\begin{pmatrix}
		m_1\\
		m_2 +a (m_1)^2
	\end{pmatrix}.
\end{equation*}
The cotangent bundle $T^*X$ then admits a global frame
\begin{equation*}
	(\alpha_1, \alpha_2, \alpha_3, \alpha_4) := (dx_1, dx_2 - 2ax_1 dx_1, dy^1 + 2ax_1 dy^2, dy^2)
\end{equation*}
satisfying $d\alpha_1 = d\alpha_2 = d\alpha_4 = 0$ and $d\alpha_3 = 2a\alpha_1 \wedge \alpha_4$. When $a = 0$, the integral affine structure reduces to the standard one and $X \cong \mathbb{T}^4$; when $a > 0$, $X$ is diffeomorphic to the Kodaira--Thurston manifold \cite{Thu1976} (which gives an important example of non-K\"ahler mirror symmetry; see e.g., \cite{Abo2014, LauTseYau2015, Pop2020}).\par
Consider the quotient bundle associated with the unitary connection
\begin{equation*}
	d - 2\pi\sqrt{-1} (x_1 dx_2 + y^2 dy^1)
\end{equation*}
and the fiberwise factor of automorphy
\begin{equation*}
	\lambda \bullet (x, y, v) = (x, y + \lambda, e^{2\pi\sqrt{-1} \lambda^2 y^1} v).
\end{equation*}
Note that
\begin{align*}
	& d - 2\pi\sqrt{-1} (x'_1 dx'_2 + (y')^2 d(y')^1)\\
	= & d - 2\pi\sqrt{-1} (x_1dx_2 + y^2 dy^1 + d(am_1(x_1)^2 + 2a(m_1)^2 x_1 + m_1 x_2 - am_2 (y^2)^2))
\end{align*}
and
\begin{align*}
	& \lambda \bullet' (x, y, v) = (x, y + \lambda, e^{2\pi\sqrt{-1} \lambda^2 (y^1 - 2am_2y^2)}).
\end{align*}
Then via the gauge transformation
\begin{equation*}
	\exp \left( 2\pi\sqrt{-1} \left( am_1(x_1)^2 + 2a(m_1)^2 x_1 + m_1 x_2 - am_2 (y^2)^2 \right) \right),
\end{equation*}
we can glue the local quotient bundles to a global Hermitian line bundle with unitary connection $(L, \nabla^L)$ over $X$ such that $\mathcal{B}_{\operatorname{cc}} = (X, L, \nabla^L)$ is a rank-$1$ coisotropic A-brane on $(X, \omega)$ satisfying Assumption \ref{Setup}. The invariant factor matrix of $\mathcal{B}_{\operatorname{cc}}$ is $\boldsymbol{H} = \begin{pmatrix}
	1
\end{pmatrix}$.

\subsubsection{Ooguri--Vafa metric, away from the singularity}
\quad\par
\label{Subsubsection: Ooguri--Vafa metric, away from the singularity}
We give an example that $B$ is the smooth locus of an SYZ fibration with an affine singularity (cf. \cite{Cha2010, Cha2013}). Consider $B = \mathbb{C}^\times$ covered by three contractible open subsets $U_\alpha, U_\beta, U_\gamma$, where
\begin{align*}
	U_\alpha = & \{ re^{\sqrt{-1}\theta}: r > 0, \theta \in (0, \tfrac{3\pi}{2}) \},\\
	U_\beta = & \{ re^{\sqrt{-1}\theta}: r > 0, \theta \in (\pi, 2\pi) \},\\
	U_\gamma = & \{ re^{\sqrt{-1}\theta}: r > 0, \theta \in (-\tfrac{\pi}{2}, \pi) \}.
\end{align*}
For $\eta \in \{\alpha, \beta, \gamma\}$, let $(x_{\eta, 1}, x_{\eta, 2})$ denote the restriction of the standard real coordinates on $\mathbb{C} \cong \mathbb{R}^2$ to $U_\eta$, and let $(L_\eta, \nabla^{L_\eta})$ denote the quotient bundle on $U_\eta \times \mathbb{T}^2$ associated with the unitary connection
\begin{equation*}
	d - 2\pi\sqrt{-1} \left( \tfrac{1}{2} x_{\eta, 1} dx_{\eta, 2} + 2 y_\eta^2 dy_\eta^1 \right)
\end{equation*}
and the fiberwise factor of automorphy
\begin{equation*}
	\lambda_\eta \bullet_\eta (x_\eta, y_\eta, v) = (x_\eta, y_\eta + \lambda, e^{4\pi\sqrt{-1} \lambda_\eta^2 y_\eta^1} v).
\end{equation*}
Equip $B$ with the integral affine structure given by the integral affine coordinate changes:
\begin{equation*}
	x_\alpha = x_\beta \quad \text{on } U_{\alpha\beta}; \quad x_\beta = x_\gamma \quad \text{on } U_{\beta\gamma}; \quad (x_{\gamma, 1}, x_{\gamma, 2}) = (x_{\alpha, 1} + x_{\alpha, 2}, x_{\alpha, 2}) \quad \text{on } U_{\gamma\alpha}.
\end{equation*}
Let $\phi_{\alpha\beta}: (L_\beta, \nabla^{L_\beta}) \vert_{U_{\alpha\beta}} \to (L_\alpha, \nabla^{L_\alpha}) \vert_{U_{\alpha\beta}}$ and $\phi_{\beta\gamma}: (L_\gamma, \nabla^{L_\gamma}) \vert_{U_{\beta\gamma}} \to (L_\beta, \nabla^{L_\beta}) \vert_{U_{\beta\gamma}}$ be the identity maps. Note that with respect to the coordinates $(x_\alpha, y_\alpha)$, $(L_\gamma, \nabla^{L_\gamma}) \vert_{U_{\gamma\alpha}}$ is the quotient bundle associated with the unitary connection
\begin{equation*}
	d - 2\pi\sqrt{-1} \left( \tfrac{1}{2} (x_{\alpha, 1} + x_{\alpha, 2}) dx_{\alpha, 2} + 2(-y_\alpha^1 + y_\alpha^2) dy_\alpha^1 \right)
\end{equation*}
and the fiberwise factor of automorphy
\begin{equation*}
	\lambda_\alpha \bullet_\gamma (x_\alpha, y_\alpha, v) = (x_\alpha, y_\alpha + \lambda_\alpha, e^{4\pi\sqrt{-1} (-\lambda_\alpha^1 + \lambda_\alpha^2) y_\alpha^1} v).
\end{equation*}
Let $\phi_{\gamma\alpha}: (L_\alpha, \nabla^{L_\alpha}) \vert_{U_{\gamma\alpha}} \to (L_\gamma, \nabla^{L_\gamma}) \vert_{U_{\gamma\alpha}}$ be the isomorphism of Hermitian line bundles with unitary connection induced by the gauge transformation $\exp \left( 2\pi\sqrt{-1} \left( \tfrac{1}{4} (x_{\alpha, 2})^2 - (y_\alpha^1)^2 \right) \right)$. Then the data glue to a global Hermitian line bundle with unitary connection $(L, \nabla^L)$ over $B$ such that $\mathcal{B}_{\operatorname{cc}} = (X, L, \nabla^L)$ is a rank-$1$ coisotropic A-brane on $(X, \omega)$ satisfying Assumption \ref{Setup}. The invariant factor matrix of this brane is $\boldsymbol{H} = \begin{pmatrix}
	2
\end{pmatrix}$.

\subsection{The proof of the existence of skew-Smith forms}
\quad\par
\label{Subsection: proof of skew smith form}
In this last subsection, we give a detailed proof of Proposition \ref{Proposition: skew smith form of A brane}.

\begin{proof}[\myproof{Proposition}{\ref{Proposition: skew smith form of A brane}}]
	We start with the data $(x, f, \widetilde{g}, H, \chi)$ in \eqref{Equation: original data}.\par
	\textbf{Step (1):} By Assumption \ref{Setup}, $H$ is invertible. We claim that $H$ can be chosen such that
	\begin{equation}
		\label{Equation: Skew-Smith form of H}
		H = \begin{pmatrix}
			0 & -\boldsymbol{H}\\
			\boldsymbol{H} & 0
		\end{pmatrix},
	\end{equation}
	where $\boldsymbol{H}$ denotes the diagonal matrix $\operatorname{diag}(h_1, ... ,h_n)$ and $h_1, ... ,h_n \in \mathbb{Z}_{>0}$ are the elementary divisors of $H$ with $h_1 \mid h_2 \mid \cdots \mid h_n$. This is because there exists $A \in \operatorname{GL}(2n, \mathbb{Z})$ such that
	\begin{equation*}
		H' := A H A^T = \begin{pmatrix}
			0 & -\boldsymbol{H}\\
			\boldsymbol{H} & 0
		\end{pmatrix}.
	\end{equation*}
	Performing an integral affine coordinate change $x' = Ax$ (hence $y' = A^{-T} y$) and defining $f'(x') = A^{-T} f(x)$, $\widetilde{g}'(x') = A \widetilde{g}(x)$ and $\chi'(\lambda) = \chi(A^T\lambda)$, we obtain a new set of data $(x', f', \tilde{g}', H', \chi')$ of the same kind. Hence, our claim holds.\par
	Now, given $H$ being chosen of the form \eqref{Equation: Skew-Smith form of H}, there is a canonical semi-character with respect to $H$, namely $\chi_0(\lambda) = (-1)^{\boldsymbol{\lambda}^2 \cdot \boldsymbol{H} \boldsymbol{\lambda}^1}$. Hence, $\chi$ must be of the form
	\begin{equation}
		\chi(\lambda) = \chi_0(\lambda) e^{2\pi\sqrt{-1} \check{g} \cdot \lambda}, \quad \text{for some } \check{g} \in \mathbb{R}^{2n}.
	\end{equation}
	
	\textbf{Step (2):} Let $g = H^{-T} \widetilde{g}$ so that \eqref{Equation: first connection for local A brane} becomes
	\begin{equation*}
		d - 2\pi\sqrt{-1} \left( \tfrac{1}{2} f \cdot dx + g \cdot H dy + \tfrac{1}{2} y \cdot H dy \right).
	\end{equation*}
	Via a gauge transformation by $e^{-\pi\sqrt{-1} g \cdot H y}$, $(L, \nabla^L) \vert_{X_U}$ is isomorphic to the quotient bundle on $U \times \mathbb{T}^{2n}$ associated with the unitary connection
	\begin{equation*}
		d - \pi\sqrt{-1} \left( f \cdot dx - g \cdot H dg + (y + g) \cdot H d(y + g) \right)
	\end{equation*}
	and the fiberwise factor of automorphy
	\begin{equation}
		\label{Equation: fibrewise factor of automorphy for local A brane}
		\lambda \bullet (x, y, v) = (x, y + \lambda, \chi(\lambda) e^{\pi\sqrt{-1} \cdot \lambda \cdot H (y + g) } v).
	\end{equation}
	
	\textbf{Step (3):} Define $F := \tfrac{1}{2} ( Df - (Df)^T)$ and $G := Dg$. We claim that $H^{-1} = -(F + GHG)$ and $G^T = G$, leaving its proof to Lemma \ref{Lemma: expression of inverse of H}. Noting that $F^T = -F$, this claim implies
	\begin{equation*}
		d \left( f \cdot dx - g \cdot H dg \right) = -dx \cdot \left( F + GHG \right) dx = dx \cdot H^{-1} dx.
	\end{equation*}
	Shrinking $U$ if necessary, the $1$-form $f \cdot dx - g \cdot H dg - x \cdot H^{-1} dx$ is exact. Hence, via a suitable gauge transformation, $(L, \nabla^L) \vert_{X_U}$ is isomorphic to the quotient bundle on $U \times \mathbb{T}^{2n}$ associated with the fiberwise factor of automorphy \eqref{Equation: fibrewise factor of automorphy for local A brane} and the unitary connection
	\begin{equation}
		d - \pi\sqrt{-1} ( x \cdot H^{-1} dx + (y + g) \cdot H d(y + g) ).
	\end{equation}
	
	\textbf{Step (4):} Via a final gauge transformation $e^{\pi\sqrt{-1} (\boldsymbol{x}_1 \cdot \boldsymbol{H}^{-1} \boldsymbol{x}_2 + (\boldsymbol{y}^2 + \boldsymbol{g}^2) \cdot \boldsymbol{H} (\boldsymbol{y}^1 + \boldsymbol{g}^1))}$, $(L, \nabla^L) \vert_{X_U}$ is isomorphic to the desired quotient bundle on $U \times \mathbb{T}^{2n}$.
\end{proof}

\begin{lemma}
	\label{Lemma: expression of inverse of H}
	Let $F, G, H$ be as in Step (3) of the proof of Proposition \ref{Proposition: skew smith form of A brane}. Then
	\begin{equation*}
		H^{-1} = -(F + GHG) \quad \text{and} \quad G = G^T.
	\end{equation*}
\end{lemma}
\begin{proof}
	The holomorphic symplectic form $\Omega_X$ on $X$ can be expressed as
	\begin{equation*}
		\Omega_X = -\tfrac{1}{2} dx \wedge F dx + dx \wedge G^TH dy + \tfrac{1}{2} dy \wedge H dy + \sqrt{-1} dx \wedge dy.
	\end{equation*}
	It implies that, locally, the complex structure on $X$ induced by $\mathcal{B}_{\operatorname{cc}}$ is given by
	\begin{equation}
		\label{Equation: matrix form of complex structure of X}
		\begin{pmatrix}
			0 & -I_{2n}\\
			I_{2n} & 0
		\end{pmatrix}
		\begin{pmatrix}
			-F & G^TH\\
			-(G^TH)^T & H
		\end{pmatrix}
		=
		\begin{pmatrix}
			-HG & -H\\
			-F & G^TH
		\end{pmatrix}
	\end{equation}
	as a matrix relative to the frame $(\tfrac{\partial}{\partial x_1}, ..., \tfrac{\partial}{\partial x_{2n}}, \tfrac{\partial}{\partial y^1}, ..., \tfrac{\partial}{\partial y^{2n}})$. Denote by $I_m$ the $m \times m$ identity matrix for any $m \in \mathbb{Z}_{>0}$. As the square of \eqref{Equation: matrix form of complex structure of X} is equal to $-I_{4n}$, we obtain 
	\begin{equation*}
		(HG)^2 + HF = -I_{2n}, \quad G = G^T, \quad FHG = G^THF \quad \text{and} \quad FH + (G^TH)^2 = -I_{2n}
	\end{equation*}
	(indeed, the last two equations can be deduced from the first two). In particular, the first equation implies that $H^{-1} = -(F + GHG)$.
\end{proof}

\section{The mirror B-brane $\check{\mathcal{B}}_{\operatorname{cc}}$ via Fourier--Mukai type transforms}
\label{Section: The mirror B-brane via Fourier--Mukai type transforms}

Throughout this section, we fix a positive integer $k \in \mathbb{Z}_{>0}$.\par

In \cite{ChaLeuZha2018}, K.-L. Chan, Zhang and the second author of this paper generalized the real Fourier--Mukai transforms (see {\cite{AriPol2001, LeuYauZas2000}) of Lagrangian A-branes on $X$ to arbitrary semi-affine coisotropic A-branes on $X$. For the canonical coisotropic A-brane $\mathcal{B}_{\operatorname{cc}}^{(k)}$ on $(X,k\omega)$, their construction of the mirror B-brane proceeds by twisting the tensor product bundle $\pi^*L^{\otimes k} \otimes \mathcal{P}$ with a fiberwise spinor bundle and then taking the kernel of the associated family Dirac operator.\par
	In this paper, we slightly modify this construction to better align with the brane quantization proposal of Gukov--Witten. For each sufficiently small open subset $U$ of $B$, we first construct a local mirror B-brane on $\check{X}_U^{(k)}$ associated with $\mathcal{B}_{\operatorname{cc}}^{(k)}$ by taking a \emph{family geometric quantization} of $\pi^*L^{\otimes k} \otimes \mathcal{P}$. We then glue these local objects using family Blattner--Kostant--Sternberg pairing maps to obtain a global mirror B-brane, denoted by $\check{\mathcal{B}}_{\operatorname{cc}}^{(k)}$.\par
	As will be shown in Theorem \ref{Proposiiton: Global mirror B-brane}, under the standing assumption \ref{Setup}, the B-brane $\check{\mathcal{B}}_{\operatorname{cc}}^{(k)}$ on the mirror manifold $\check{X}^{(k)}$ is semi-affine and space-filling, i.e. of the form
	\begin{equation*}
		\check{\mathcal{B}}_{\operatorname{cc}}^{(k)} = (\check{X}^{(k)}, \check{E}^{(k)}, \nabla^{\check{E}^{(k)}}),
	\end{equation*}
	where $(\check{E}^{(k)}, \nabla^{\check{E}^{(k)}})$ is a Hermitian vector bundle on $\check{X}^{(k)}$ with a unitary connection. In general, $\check{E}^{(k)}$ has rank greater than one, except when $k$ and all invariant factors of $\mathcal{B}_{\operatorname{cc}}$ are equal to $1$.
	
	\subsection{Local construction of the mirror B-brane: outline}
	\quad\par
	\label{Subsection: Local construction of mirror brane: outline}
	Consider a fixed open subset $U \subset B$ for which $(L, \nabla^L)$ has a skew-Smith form determined by data $x, g, \check{g}$ described in Proposition \ref{Proposition: skew smith form of A brane}.
	
	\subsubsection{An explicit local form of $L^{\otimes k} \otimes \mathcal{P}$}
	\quad\par
	\label{Subsubsection: An explicit local form}
	Identify $U$ with an open subset of $\mathbb{R}^{2n}$ via the integral affine chart $x$. Via the coordinates $(x, y, \check{y})$, we obtain a local trivialization $T^*U \oplus TU \cong U \times \mathbb{R}^{4n}$ which identifies the lattice bundle $\Lambda_U \oplus \check{\Lambda}_U$ with $U \times \mathbb{Z}^{4n}$ and hence descends to a diffeomorphism $X_U \times_U \check{X}_U \cong U \times \mathbb{T}^{4n}$. To avoid confusion, we henceforth let $\check{\mathbb{T}}^{2n}$ be another copy of $\mathbb{T}^{2n}$ and write
	\begin{equation}
		\label{Equation: coordinate description of fibre product}
		X_U \times_U \check{X}_U \cong U \times \mathbb{T}^{2n} \times \check{\mathbb{T}}^{2n}.
	\end{equation}
	
	On this manifold, it is convenient to introduce the new coordinates $(x, u, \check{u})$, where
	\begin{equation*}
		u = (u^1, ..., u^{2n}) := y + g(x) \quad \text{and} \quad \check{u} = (\check{u}_1, ..., \check{u}_{2n}) := \check{y} - k\check{g},
	\end{equation*}
	and take the quotient bundle $(\mathbf{L}, \nabla^\mathbf{L})$ associated with the unitary connection
	\begin{equation}
		\label{Equation: unitary connection over fibre product}
		d - 2\pi\sqrt{-1} ( k \boldsymbol{x}_1 \cdot \boldsymbol{H}^{-1} d\boldsymbol{x}_2 + k \boldsymbol{u}^2 \cdot \boldsymbol{H} d\boldsymbol{u}^1 + \check{y} \cdot dy)
	\end{equation}
	and the fiberwise factor of automorphy
	\begin{equation*}
		(\lambda, \check{\lambda}) \bullet (x, y, \check{y}, v) = (x, y + \lambda, \check{y} + \check{\lambda}, e^{2\pi\sqrt{-1} (k\check{g} \cdot \lambda + k \boldsymbol{\lambda}^2 \cdot \boldsymbol{H} \boldsymbol{u}^1 + \check{\lambda} \cdot y)} v).
	\end{equation*}
	By Proposition \ref{Proposition: skew smith form of A brane} and the construction of the Poincar\'e bundle, there exists an isomorphism $\pi^*L^{\otimes k} \otimes \mathcal{P} \vert_{X_U \times_U \check{X}_U} \cong \mathbf{L}$ of Hermitian line bundles with unitary connection, covering the diffeomorphism \eqref{Equation: coordinate description of fibre product}, so that
	\begin{equation*}
		\Gamma(X_U \times_U \check{X}_U, \pi^*L^{\otimes k} \otimes \mathcal{P}) \cong \Gamma(U \times \mathbb{T}^{2n} \times \check{\mathbb{T}}^{2n}, \mathbf{L}).
	\end{equation*}
	
	\subsubsection{Prequantum line bundles on a symplectic torus parametrized by $\check{X}_U$}
	\quad\par
	\label{Subsubsection: Prequantum line bundles on a symplectic torus parametrized by mirror manifold}
	The manifold $\check{X}_U \cong U \times \check{\mathbb{T}}^{2n}$ parametrizes a class of prequantum line bundles of the symplectic torus $(\mathbb{T}^{2n}, d\boldsymbol{y}^2 \wedge \boldsymbol{H} d\boldsymbol{y}^1)$ as follows. Given a point in $U \times \check{\mathbb{T}}^{2n}$, we identity the fiber of the canonical projection $U \times \mathbb{T}^{2n} \times \check{\mathbb{T}}^{2n} \to U \times \check{\mathbb{T}}^{2n}$ over this point with $\mathbb{T}^{2n}$. Then the restriction of $\mathbf{L}$ onto this fiber is a prequantum line bundle of $(\mathbb{T}^{2n}, d\boldsymbol{y}^2 \wedge \boldsymbol{H} d\boldsymbol{y}^1)$.\par
	One can obtain an explicit description of this prequantum line bundle via a fixed value of $(x, \check{y})$ representing the given point -- such a value is unique up to transforms $(x, \check{y}) \mapsto (x, \check{y} + \check{\lambda})$ for $\check{\lambda} \in \mathbb{Z}^{2n}$. The prequantum line bundle is indeed isomorphic to the quotient bundle $\mathbf{L}_{(x, \check{y})}$ of the trivial Hermitian line bundle $\mathbb{R}^{2n} \times \mathbb{C}$ equipped with the unitary connection
	\begin{equation*}
		d - 2\pi\sqrt{-1} \left( k \boldsymbol{u}^2 \cdot \boldsymbol{H} d\boldsymbol{u}^1 + \check{y} \cdot dy \right)
	\end{equation*}
	by the factor of automorphy
	\begin{equation*}
		\lambda \bullet (y, v) = \left(y + \lambda, e^{2\pi \sqrt{-1} k \left( \check{g} \cdot \lambda + \boldsymbol{\lambda}^2 \cdot \boldsymbol{H} \boldsymbol{u}^1 \right)} v \right).
	\end{equation*}
	
	\subsubsection{Invariant K\"ahler polarizations and half-form bundles}
	\quad\par
	\label{Subsection: Invariant polarizarions and half form bundles}
	Every element $\Omega = (\Omega_{ij}) \in \mathbb{H}_n$ in the Segal upper half space determines a $\mathbb{T}^{2n}$-invariant complex structure on $\mathbb{T}^{2n}$ via the complex coordinates
	\begin{equation}
		\label{Equation: complex coordinates on fibre torus}
		w = (w^1, ..., w^n) := \boldsymbol{H} \boldsymbol{y}^1 - \Omega \boldsymbol{y}^2,
	\end{equation}
	which is compatible with the symplectic form $d\boldsymbol{y}^2 \wedge \boldsymbol{H} d\boldsymbol{y}^1 = \tfrac{\sqrt{-1}}{2} dw \wedge (\operatorname{Im} \Omega)^{-1} d\overline{w}$. Hence the antiholomorphic tangent bundle $T^{0, 1}\mathbb{T}^{2n}$ defines a K\"ahler polarization. It admits the global frame
	\begin{equation*}
		\frac{\partial}{\partial \overline{w}} := \left( \frac{\partial}{\partial \overline{w}^1}, ..., \frac{\partial}{\partial \overline{w}^n} \right)= \frac{1}{2} (\operatorname{Im} \Omega)^{-1} \left( \frac{\partial}{\partial \boldsymbol{y}^2} + \Omega \boldsymbol{H}^{-1} \frac{\partial}{\partial \boldsymbol{y}^1} \right).
	\end{equation*}
	
	We briefly review the construction of half-form bundles for such polarizations (cf. \cite{BaiMouNun2010}). The canonical bundle $K := \bigwedge^n T^{*(1, 0)} \mathbb{T}^{2n}$ admits the holomorphic frame $d^n w$. For another element $\Omega' \in \mathbb{H}_n$, let $w'$ and $K'$ denote the corresponding complex coordinates and canonical bundle. There is a natural sesquilinear pairing
	\begin{equation*}
		h^{K, K'}: \Gamma(\mathbb{T}^{2n}, K) \times \Gamma(\mathbb{T}^{2n}, K') \to \mathcal{C}^\infty(\mathbb{T}^{2n})
	\end{equation*}
	defined by
	\begin{equation*}
		h^{K, K'} (\eta, \eta') \cdot \det (-\boldsymbol{H}) \cdot dy^1 \wedge \cdots \wedge dy^{2n} = \left( \sqrt{-1} \right)^{n^2} \eta \wedge \overline{\eta}'.
	\end{equation*}
	In particular, $h^{K, K'} ( d^nw, d^nw' ) = \det \left( \tfrac{1}{\sqrt{-1}} (\Omega - \overline{\Omega'}) \right)$.\par
	Indeed, one can choose a family of holomorphic line bundles over $\mathbb{T}^{2n}$ parametrized by $\mathbb{H}_n$ such that, for each $\Omega \in \mathbb{H}_n$, the corresponding bundle $\sqrt{K}$ is a square root of $K$ and admits a global holomorphic frame $\sqrt{d^n w}$ satisfying $d^nw = \sqrt{d^nw} \otimes \sqrt{d^nw}$. For each pair $\Omega, \Omega' \in \mathbb{H}_n$, this determines a sesquilinear pairing
	\begin{equation*}
		h^{\sqrt{K}, \sqrt{K'}} : \Gamma(\mathbb{T}^{2n}, \sqrt{K}) \times \Gamma(\mathbb{T}^{2n}, \sqrt{K'}) \to \mathcal{C}^\infty(\mathbb{T}^{2n}),
	\end{equation*}
	defined on generators by
	\begin{equation*}
		h^{\sqrt{K}, \sqrt{K'}} \left( \sqrt{d^nw}, \sqrt{d^nw'} \right) = \sqrt{\det \left( \tfrac{1}{\sqrt{-1}} (\Omega - \overline{\Omega'}) \right)},
	\end{equation*}
	where the square root is taken in the principal branch.\par
	For each $\Omega$, the pairing $h^K := h^{K,K}$ defines a Hermitian metric on $K$ for which $d^n w$ is parallel with respect to the Chern connection $\nabla^K$. Similarly, $h^{\sqrt{K}} := h^{\sqrt{K},\sqrt{K}}$ defines a Hermitian metric on $\sqrt{K}$ with Chern connection $\nabla^{\sqrt{K}}$, under which $\sqrt{d^n w}$ is parallel. In this way, $(\sqrt{K}, h^{\sqrt{K}}, \nabla^{\sqrt{K}})$ is a square root of $(K, h^{K}, \nabla^{K})$. We continue to denote by $(\sqrt{K}, h^{\sqrt{K}}, \nabla^{\sqrt{K}})$, as well as by $\sqrt{d^n w}$, their pullbacks along the projection
	\begin{equation*}
		U \times \mathbb{T}^{2n} \times \check{\mathbb{T}}^{2n} \to \mathbb{T}^{2n}.
	\end{equation*}
	
	For any $\Omega, \Omega' \in \mathbb{H}_n$, the Hermitian metric on $\mathbf{L}$ together with $h^{\sqrt{K},\sqrt{K'}}$ induces a fiberwise sesquilinear pairing
	\begin{equation}
		\label{Equation: fibrewise pairing of prequantum spaces}
		\langle \, \cdot \, , \, \cdot \, \rangle: \Gamma(U \times \mathbb{T}^{2n} \times \check{\mathbb{T}}^{2n}, \mathbf{L} \otimes \sqrt{K}) \times \Gamma(U \times \mathbb{T}^{2n} \times \check{\mathbb{T}}^{2n}, \mathbf{L} \otimes \sqrt{K'}) \to \mathcal{C}^\infty(U \times \check{\mathbb{T}}^{2n})
	\end{equation}
	given by
	\begin{equation*}
		\langle s \otimes \sqrt{d^nw}, s' \otimes \sqrt{d^nw'} \rangle (x, \check{y}) := \sqrt{\det \left( \tfrac{1}{\sqrt{-1}} (\Omega - \overline{\Omega'}) \right)} \int_{[0, 1]^{2n}} s(x, y, \check{y}) \overline{s'(x, y, \check{y})} d^{2n}y.
	\end{equation*}
	In particular, this defines a fiberwise inner product on $\Gamma(U \times \mathbb{T}^{2n} \times \check{\mathbb{T}}^{2n}, \mathbf{L} \otimes \sqrt{K})$.
	
	\subsubsection{Local mirror B-branes as bundles of quantum Hilbert spaces}
	\quad\par
	\label{Subsubsection: Local mirror B-branes as bundles of quantum Hilbert spaces}
	For a given $\Omega \in \mathbb{H}_n$, the pullback of $T^{0, 1}\mathbb{T}^{2n}$ defines a CR structure on $U \times \mathbb{T}^{2n} \times \check{\mathbb{T}}^{2n}$. The induced unitary connection on $\mathbf{L} \otimes \sqrt{K}$ is flat along this CR structure, allowing us to define
	\begin{equation*}
		H^0(U \times \mathbb{T}^{2n} \times \check{\mathbb{T}}^{2n}, \mathbf{L} \otimes \sqrt{K})
	\end{equation*}
	as the space of smooth sections of $\mathbf{L} \otimes \sqrt{K}$ that are flat along this CR structure.\par
	We then construct a complex vector bundle $\mathbf{\check{E}}$ over $U \times \check{\mathbb{T}}^{2n}$ by specifying its fiber over any point in $U \times \check{\mathbb{T}}^{2n}$, represented by $(x, \check{y}) \in U \times \check{\mathbb{R}}^{2n}$, to be the space
	\begin{equation*}
		H^0(\mathbb{T}^{2n}, \mathbf{L}_{(x, \check{y})} \otimes \sqrt{K})
	\end{equation*}
	of holomorphic sections of $\mathbf{L}_{(x, \check{y})} \otimes \sqrt{K}$ with respect to the complex structure associated with $\Omega$. This construction yields a natural identification
	\begin{equation*}
		\Gamma(U \times \check{\mathbb{T}}^{2n}, \mathbf{\check{E}}) \cong H^0(U \times \mathbb{T}^{2n} \times \check{\mathbb{T}}^{2n}, \mathbf{L} \otimes \sqrt{K}).
	\end{equation*}
	Specifically, this identification induces a continuous $\mathcal{C}^\infty(U \times \check{\mathbb{T}}^{2n})$-linear embedding
	\begin{equation*}
		\Gamma(U \times \check{\mathbb{T}}^{2n}, \mathbf{\check{E}}) \hookrightarrow \Gamma(U \times \mathbb{T}^{2n} \times \check{\mathbb{T}}^{2n}, \mathbf{L} \otimes \sqrt{K}),
	\end{equation*}
	where both of the spaces are equipped with the compact-open $\mathcal{C}^\infty$-topology. In this context, $\mathcal{C}^\infty(U \times \check{\mathbb{T}}^{2n})$ acts on the right-hand space via pullback.\par
	The fiberwise inner product \eqref{Equation: fibrewise pairing of prequantum spaces} on $\Gamma(U \times \mathbb{T}^{2n} \times \check{\mathbb{T}}^{2n}, \mathbf{L} \otimes \sqrt{K})$ gives rise to a Hermitian metric on $\mathbf{\check{E}}$ and a family orthogonal projection
	\begin{equation}
		\label{Equation: family orthogonal projection}
		\Pi: \Gamma(U \times \mathbb{T}^{2n} \times \check{\mathbb{T}}^{2n}, \mathbf{L} \otimes \sqrt{K}) \to \Gamma(U \times \check{\mathbb{T}}^{2n}, \mathbf{\check{E}}),
	\end{equation}
	which is $\mathcal{C}^\infty(U \times \check{\mathbb{T}}^{2n})$-linear. Denote by $\nabla$ the tensor product of the unitary connections on $\mathbf{L}$ and $\sqrt{K}$. For any complex vector field $\xi$ on $U \times \check{\mathbb{T}}^{2n}$, define
	\begin{equation}
		\label{Equation: B brane connectino}
		\check{\nabla}_\xi := \Pi \circ \nabla_{\widehat{\xi}}: \Gamma(U \times \check{\mathbb{T}}^{2n}, \mathbf{\check{E}}) \to \Gamma(U \times \check{\mathbb{T}}^{2n}, \mathbf{\check{E}}),
	\end{equation}
	where $\widehat{\xi}$ denotes the canonical lift of $\xi$ with respect to the canonical decomposition
	\begin{equation*}
		T(U \times \mathbb{T}^{2n} \times \check{\mathbb{T}^{2n}}) \cong T(U \times \check{\mathbb{T}^{2n}}) \times T\mathbb{T}^{2n}.
	\end{equation*}
	
	While our construction closely parallels that of \cite{ChaLeuZha2018}, it yields the following result.
	
	\begin{theorem}
		\label{Theorem: unitary connection on the local mirror B brane}
		The triple $(U \times \check{\mathbb{T}}^{2n}, \mathbf{\check{E}}, \check{\nabla})$ defines a semi-affine B-brane on the complex manifold $U \times \check{\mathbb{T}}^{2n}$, endowed with the complex structure determined by the coordinates $\check{z}$. We call it a \emph{local mirror B-brane} of $\mathcal{B}_{\operatorname{cc}}$. Moreover, the curvature of $\check{\nabla}$ is locally given by
		\begin{equation*}
			-2\pi\sqrt{-1} \cdot d\check{z} \wedge \tfrac{k}{2} \left( H^{-1} + \sqrt{-1} Dg \right) d\overline{\check{z}} \otimes \operatorname{Id}_{\mathbf{\check{E}}},
		\end{equation*}
		where $H$ is as in \eqref{Equation: Skew-Smith form of H} and $Dg$ is the Jacobian matrix of $g$ with respect to $x$.
	\end{theorem}
	
	We will provide a computational proof of this theorem in the next subsection.
	
	\subsection{Local construction of the mirror B-brane: computational proof}
	\quad\par
	\label{Subsection: Local construction of mirror brane: computational proof}
	We continue to adopt notations introduced in the previous subsection. The key idea in the proof of Theorem \ref{Theorem: unitary connection on the local mirror B brane} is to transform the space of sections $\Gamma(U \times \check{\mathbb{T}}^{2n}, \mathbf{\check{E}})$ into a more tractable space in which all relevant data of $\mathbf{\check{E}}$ become explicitly computable.\par
	To this end, we first identify $\Gamma(U \times \mathbb{T}^{2n} \times \check{\mathbb{T}}^{2n}, \mathbf{L})$ with the space of smooth $\mathbb{C}$-valued functions $s \in \mathcal{C}^\infty(U \times \mathbb{R}^{4n})$ satisfying the following quasi-periodicity condition: for all $\lambda, \check{\lambda} \in \mathbb{Z}^{2n}$,
	\begin{equation}
		\label{Equation: quasi periodicity of sections in prequantum space}
		s(x, y + \lambda, \check{y} + \check{\lambda}) = e^{2\pi\sqrt{-1} (k\check{g} \cdot \lambda + k \boldsymbol{\lambda}^2 \cdot \boldsymbol{H} \boldsymbol{u}^1 + \check{\lambda} \cdot y)} s(x, y, \check{y}).
	\end{equation}
	
	\subsubsection{Family Weil--Brezin transform}
	\quad\par
	\label{Subsubsection: Family Weil--Brezin transform}
	We adapt a standard technique from geometric quantization on symplectic tori --- namely, the \emph{Weil--Brezin transform} \cite{BaiMouNun2010} --- to a family setting. This yields an isomorphism of topological vector spaces
	\begin{equation*}
		\Gamma(U \times \mathbb{T}^{2n} \times \check{\mathbb{T}}^{2n}, \mathbf{L} \otimes \sqrt{K})
		\;\cong\;
		\mathcal{C}^\infty\bigl(U, \mathcal{S}(\mathbb{R}^n \times \check{\mathbb{R}}^n)\bigr).
	\end{equation*}
	Here, the space $\Gamma(U \times \mathbb{T}^{2n} \times \check{\mathbb{T}}^{2n}, \mathbf{L} \otimes \sqrt{K})$ is endowed with the compact-open $\mathcal{C}^\infty$-topology. On the right-hand side, $\check{\mathbb{R}}^n$ denotes an additional copy of $\mathbb{R}^n$, and $\mathcal{C}^\infty(U, \mathcal{S}(\mathbb{R}^n \times \check{\mathbb{R}}^n))$ consists of smooth $\mathbb{C}$-valued functions on $U \times \mathbb{R}^n \times \check{\mathbb{R}}^n$ that are rapidly decreasing in the $\mathbb{R}^n \times \check{\mathbb{R}}^n$ variables. This space is endowed with its natural Fr\'echet topology, defined by seminorms controlling uniform convergence on compact subsets of $U$ together with the Schwartz decay of all derivatives in the $\mathbb{R}^n \times \check{\mathbb{R}}^n$ directions (see Appendix \ref{Sectoin: Remarks on fibrewise Schwartz spaces}).\par
	To facilitate the construction, we introduce a complete set of representatives for the quotient $\mathbb{Z}^n / (k\boldsymbol{H}\mathbb{Z}^n)$:
	\begin{equation*}
		\mathbb{Z}_{k\boldsymbol{H}}^n
		:=
		\{0,\dots,kh_1-1\} \times \cdots \times \{0,\dots,kh_n-1\}.
	\end{equation*}
	
	\begin{proposition}
		\label{Proposition: fibrewise Weil Brezin transform}
		Let $\Omega \in \mathbb{H}_n$. Then there exists an isomorphism of topological vector spaces
		\begin{equation}
			\Gamma(U \times \mathbb{T}^{2n} \times \check{\mathbb{T}}^{2n}, \mathbf{L} \otimes \sqrt{K}) \to \mathcal{C}^\infty(U, \mathcal{S}(\mathbb{R}^n \times \check{\mathbb{R}}^n)), \quad s \mapsto [s],
		\end{equation}
		uniquely characterized by the following condition. Writing $s = \underline{s} \otimes \sqrt{d^nw}$, we require that for all $(x, y, \check{y}) \in U \times \mathbb{R}^{4n}$,
		\begin{align*}
			\underline{s}(x, y, \check{y}) = & \sum_{l \in \mathbb{Z}_{k\boldsymbol{H}}^n} \sum_{q, \check{q} \in \mathbb{Z}^n} [s] (x, \boldsymbol{y}^2 - q + \check{q}, \boldsymbol{\check{u}}_1 - l - k\boldsymbol{H}\check{q})\\
			& \cdot \exp \left( 2\pi\sqrt{-1} \left( k\check{g} \cdot y + \boldsymbol{\check{u}}_2 \cdot (\boldsymbol{y}^2 - q + \check{q}) - \boldsymbol{\check{u}}_1 \cdot \boldsymbol{g}^1(x) + (l + k\boldsymbol{H}q) \cdot \boldsymbol{u}^1 \right) \right).
		\end{align*}
		Moreover, this isomorphism is $\mathcal{C}^\infty(U)$-linear.
	\end{proposition}
	\begin{proof}
		We begin by transforming $s$ into a function $\widetilde{s} \in \mathcal{C}^\infty(U \times \mathbb{R}^{4n})$ via a coordinate shift and phase factor:
		\begin{equation}
			\label{Equation: transform of sections in prequantum space}
			\widetilde{s}(x, y, \check{y}) := \underline{s}(x, y, \check{y} + k\check{g}) \cdot \exp\left( 2\pi\sqrt{-1} (\boldsymbol{\check{y}}_1 \cdot \boldsymbol{g}^1(x) - \boldsymbol{\check{y}}_2 \cdot \boldsymbol{y}^2 - k\check{g} \cdot y) \right).
		\end{equation}
		This transformation converts the quasi-periodicity condition into a real-polarized form suitable for the application of the (fiberwise) Weil--Brezin transform: for all $\lambda, \check{\lambda} \in \mathbb{Z}^{2n}$,
		\begin{align*}
			\widetilde{s}(x, \boldsymbol{y}^1 + \boldsymbol{\lambda}^1, \boldsymbol{y}^2, \boldsymbol{\check{y}}_1, \boldsymbol{\check{y}}_2 + \boldsymbol{\check{\lambda}}_2) = & \widetilde{s}(x, y, \check{y}),\\
			\widetilde{s}(x, \boldsymbol{y}^1, \boldsymbol{y}^2 + \boldsymbol{\lambda}^2, \boldsymbol{\check{y}}_1 + \boldsymbol{\check{\lambda}}_1, \boldsymbol{\check{y}}_2) = & \exp\left( 2\pi\sqrt{-1} (\boldsymbol{\check{\lambda}}_1 \cdot \boldsymbol{u}^1 + \boldsymbol{\lambda}^2 \cdot (k \boldsymbol{H} \boldsymbol{u}^1 - \boldsymbol{\check{y}}_2)) \right) \cdot \widetilde{s}(x, y, \check{y}).
		\end{align*}
		The map $s \mapsto \widetilde{s}$ is an isomorphism of topological vector spaces with respect to the chosen $\mathcal{C}^\infty$-topologies. Observe that the map $(\boldsymbol{y}^1, \boldsymbol{\check{y}}_2) \mapsto (\boldsymbol{y}^1, k\boldsymbol{H} \boldsymbol{y}^1 - \boldsymbol{\check{y}}_2)$ induces an automorphism of $\mathbb{Z}^{2n}$. Using this observation and applying the fiberwise Weil--Brezin transform, we obtain a unique function $[s] \in \mathcal{C}^\infty(U, \mathcal{S}(\mathbb{R}^n \times \check{\mathbb{R}}^n))$ such that
		\begin{equation*}
			\widetilde{s}(x, y, \check{y}) = \sum_{\boldsymbol{\lambda}^2, \boldsymbol{\check{\lambda}}_1 \in \mathbb{Z}^n} [s](x, \boldsymbol{y}^2 - \boldsymbol{\lambda}^2, \boldsymbol{\check{y}}_1 - \boldsymbol{\check{\lambda}}_1) \cdot \exp \left( 2\pi\sqrt{-1}(\boldsymbol{\check{\lambda}}_1 \cdot \boldsymbol{u}^1 + \boldsymbol{\lambda}^2 \cdot (k \boldsymbol{H} \boldsymbol{u}^1 - \boldsymbol{\check{y}}_2)) \right).
		\end{equation*}
		Decomposing $\boldsymbol{\check{\lambda}}_1 = l + k\boldsymbol{H}\check{q}$ with $l \in \mathbb{Z}_{k\boldsymbol{H}}^n$ and $\check{q} \in \mathbb{Z}^n$, and performing the change of variables $(q,\check{q}) = (\boldsymbol{\lambda}^2 + \check{q}, \check{q})$, we obtain
		\begin{equation*}
			\widetilde{s}(x, y, \check{y}) = \sum_{l \in \mathbb{Z}_{k\boldsymbol{H}}^n} \sum_{q, \check{q} \in \mathbb{Z}^n} [s] (x, \boldsymbol{y}^2 - q + \check{q}, \boldsymbol{\check{y}}_1 - l - k\boldsymbol{H} \check{q}) \cdot e^{2\pi\sqrt{-1}( \boldsymbol{\check{y}}_2 \cdot (-q + \check{q}) + (l + k \boldsymbol{H} q) \cdot \boldsymbol{u}^1)}.
		\end{equation*}
		Finally, substituting back using \eqref{Equation: transform of sections in prequantum space} yields the stated formula for $s_\Omega$.\par
		Since both the transformation $s \mapsto \widetilde{s}$ and the fiberwise Weil--Brezin transform are isomorphisms of topological vector spaces, their composition defines the claimed isomorphism. Finally, the $\mathcal{C}^\infty(U)$-linearity is evident.
	\end{proof}
	
	Based on the expression of $\underline{s}$ stated in Proposition \ref{Proposition: fibrewise Weil Brezin transform}, a direct computation yields the following formula of the fiberwise sesquilinear pairing \eqref{Equation: fibrewise pairing of prequantum spaces}, which we omit the proof.
	
	\begin{proposition}
		\label{Proposition: fibrewsie sesquilinear pairing}
		Let $\Omega, \Omega' \in \mathbb{H}_n$. Let $s \mapsto [s]_\Omega$ and $s' \mapsto [s']_{\Omega'}$ denote the isomorphisms corresponding to $\Omega, \Omega'$, respectively, as in Proposition \ref{Proposition: fibrewise Weil Brezin transform}. Then
		\begin{align*}
			\langle s, s' \rangle (x, \check{y}) = & \sqrt{\det \left( \tfrac{1}{\sqrt{-1}} (\Omega - \overline{\Omega'}) \right)} \sum_{\check{m} \in \mathbb{Z}^{2n}} e^{2\pi\sqrt{-1} \boldsymbol{\check{m}}_2 \cdot \boldsymbol{\check{u}}_2}\\
			& \cdot \int_{\mathbb{R}^n} [s]_\Omega \left( x, \boldsymbol{y}^2 + \boldsymbol{\check{m}}_2, \boldsymbol{\check{u}}_1 - \boldsymbol{\check{m}}_1 - k\boldsymbol{H} \boldsymbol{\check{m}}_2 \right) \overline{[s']_{\Omega'} \left( x, \boldsymbol{y}^2, \boldsymbol{\check{u}}_1 - \boldsymbol{\check{m}}_1 \right)} d^n \boldsymbol{y}^2.
		\end{align*}
	\end{proposition}
	
	\subsubsection{An explicit description of local mirror B-branes}
	\quad\par
	\label{Subsubection: An explicit description of local mirror B-branes}
	Since $\Gamma(U \times \check{\mathbb{T}}^{2n}, \mathbf{\check{E}})$ is constructed from a family of geometric quantizations in K\"ahler polarization (namely, the fibers of $\mathbf{\check{E}}$), a natural first attempt to describe this space is to construct a smooth family of orthonormal bases for these quantizations, parametrized by $U \times \check{\mathbb{T}}^{2n}$, thereby yielding a unitary frame of $\mathbf{\check{E}}$. However, due to monodromies along the fiber $\check{\mathbb{T}}^{2n}$, such a globally smooth choice does not exist in general.\par
	As we shall see, it is nevertheless convenient to introduce a distinguished family of orthonormal bases of the fibers of $\mathbf{\check{E}}$, parametrized by $U \times \check{\mathbb{R}}^{2n}$, as follows.
	
	\begin{definition}
		\label{Definition: Guassian base}
		For each $l \in \mathbb{Z}_{k\boldsymbol{H}}^n$ and $(x, y, \check{y}) \in U \times \mathbb{R}^{4n}$, define
		\begin{equation*}
			\sigma_l(x, y, \check{y}) := \sum_{q \in \mathbb{Z}^n} N_\Omega(\boldsymbol{u}^2 - q, \boldsymbol{\check{u}}_1 - l) \cdot e^{2\pi\sqrt{-1}( k \check{g} \cdot y - \boldsymbol{g}^1(x) \cdot \boldsymbol{\check{u}}_1 + \boldsymbol{\check{u}}_2 \cdot (\boldsymbol{y}^2 - q) + (l + k\boldsymbol{H}q) \cdot \boldsymbol{u}^1)} \otimes \sqrt{d^nw},
		\end{equation*}
		where for all $a, b \in \mathbb{C}^n$,
		\begin{equation}
			\label{Equation: Gaussian distribution}
			N_\Omega(a, b) := k^{\frac{n}{4}} \exp \left( \pi\sqrt{-1} (a + (k\boldsymbol{H})^{-1} b) \cdot (k\Omega) (a + (k\boldsymbol{H})^{-1} b) \right).
		\end{equation}
	\end{definition}
	
	For each fixed $(x, \check{y})$, $\sigma_l(x, \,\cdot\,, \check{y})$ defines a smooth section of $\mathbf{L}_{(x, \check{y})} \otimes \sqrt{K}$, which is essentially a classical theta function.
	
	\begin{proposition}
		\label{Proposition: orthonormal basis of geometric quantization}
		For each $\Omega \in \mathbb{H}_n$, the collection $\{ \sigma_l(x, \,\cdot\,, \check{y}) \}_{l \in \mathbb{Z}_{k\boldsymbol{H}}^n}$ forms an orthonormal basis of the Hilbert space $H^0(\mathbb{T}^{2n}, \mathbf{L}_{(x, \check{y})} \otimes \sqrt{K})$. If $\Omega' \in \mathbb{H}_n$, with corresponding orthonormal basis $\{ \sigma'_l(x, \,\cdot\,, \check{y}) \}_{l \in \mathbb{Z}_{k\boldsymbol{H}}^n}$, then for all $l, l' \in \mathbb{Z}_{k\boldsymbol{H}}^n$,
		\begin{equation*}
			\left\langle \sigma_l(x, \,\cdot\,, \check{y}), \sigma'_{l'}(x, \,\cdot\,, \check{y}) \right\rangle = \delta_{ll'},
		\end{equation*}
		where
		\begin{equation*}
			\langle \,\cdot\,,\,\cdot\, \rangle: H^0(\mathbb{T}^{2n}, \mathbf{L}_{(x, \check{y})} \otimes \sqrt{K}) \times H^0(\mathbb{T}^{2n}, \mathbf{L}_{(x, \check{y})} \otimes \sqrt{K'}) \to \mathbb{C}
		\end{equation*}
		denotes the sesquilinear pairing induced by \eqref{Equation: fibrewise pairing of prequantum spaces}.
	\end{proposition}
	
	This result is essentially contained in \cite{And2005, BaiMouNun2010}, where the authors treat the case of unimodular symplectic forms on $\mathbb{T}^{2n}$. Our setting is only slightly more general, allowing for \emph{non-unimodular} integral symplectic forms and a more intricate parametrization of the prequantum line bundle $\mathbf{L}_{(x, \check{y})}$. The proof of the above proposition is nevertheless standard, and we therefore omit the details.\par
	Using this family of orthonormal bases, together with the fiberwise Weil–Brezin transform established in Proposition \ref{Proposition: fibrewise Weil Brezin transform}, we obtain a characterization of the space $\Gamma(U \times \check{\mathbb{T}}^{2n}, \mathbf{\check{E}})$ as
	\begin{align*}	
		\mathcal{C}^\infty(U, \mathcal{S}(\check{\mathbb{R}}^n)),
	\end{align*}
	the space of $\mathbb{C}$-valued smooth functions on $U \times \check{\mathbb{R}}^n$ that are rapidly decreasing in the $\check{\mathbb{R}}^n$-variables, endowed with its standard Fr\'echet topology (see Appendix \ref{Sectoin: Remarks on fibrewise Schwartz spaces}).
	
	\begin{proposition}
		\label{Proposition: sections of local mirror B-brane}
		Let $\Omega \in \mathbb{H}_n$. There exists an isomorphism of topological vector spaces
		\begin{equation}
			\label{Equation: charaterization of sections of mirror brane}
			\Gamma(U \times \check{\mathbb{T}}^{2n}, \mathbf{\check{E}}) \to \mathcal{C}^\infty(U, \mathcal{S}(\check{\mathbb{R}}^n)), \quad s \mapsto \langle s \rangle,
		\end{equation}
		uniquely characterized by the following condition: for all $(x, y, \check{y}) \in U \times \mathbb{R}^{4n}$,
		\begin{equation}
			\label{Equation: sections in term of non-smooth unitary frame}
			s(x, y, \check{y}) = \sum_{l \in \mathbb{Z}_{k\boldsymbol{H}}^n} \sigma_l(x, y, \check{y}) \sum_{\check{q} \in \mathbb{Z}^n} \langle s \rangle \left( x, \boldsymbol{\check{u}}_1 - l - k\boldsymbol{H}\check{q} \right) \cdot e^{2\pi\sqrt{-1} \check{q} \cdot \boldsymbol{\check{u}}_2},
		\end{equation}
		where $s$ is identified as a section in $H^0(U \times \mathbb{T}^{2n} \times \check{\mathbb{T}}^{2n}, \mathbf{L} \otimes \sqrt{K})$. Moreover, this isomorphism is $\mathcal{C}^\infty(U)$-linear.
	\end{proposition}
	
	\begin{remark}
		Formula \eqref{Equation: sections in term of non-smooth unitary frame} shows that for any open subset $V\subset\mathbb{R}^{2n}$ contained in the interior of a fundamental domain of $\check{\mathbb{T}}^{2n}$, one can construct a unitary frame $\{ s_l \}_{l \in \mathbb{Z}_{k\boldsymbol{H}}^n}$ of $\mathbf{\check{E}}$ over $U \times V$ by choosing $\langle s_l \rangle$ to be suitable bump functions in the $\boldsymbol{\check{y}}_1$-direction. However, such a frame cannot be extended to $U \times \check{\mathbb{T}}^{2n}$ while preserving smoothness of the associated functions $\langle s_l \rangle$; this reflects the nontrivial topology of $\mathbf{\check{E}}$ along the $\check{\mathbb{T}}^{2n}$-direction.
	\end{remark}
	
	Observe that the coordinate transformation $(x, \boldsymbol{y}^2, \boldsymbol{\check{y}}_1) \mapsto (x, \boldsymbol{u}^2 + (k \boldsymbol{H})^{-1} \boldsymbol{\check{y}}_1, \boldsymbol{\check{y}}_1)$ induces an automorphism of the topological vector space $\mathcal{C}^\infty(U, \mathcal{S}(\mathbb{R}^n \times \check{\mathbb{R}}^n))$. Under this transformation, the Gaussian factor $N_\Omega(\boldsymbol{u}^2, \boldsymbol{\check{y}}_1)$ becomes independent of $\boldsymbol{\check{y}}_1$ and is rapidly decreasing with respect to the variable $\boldsymbol{u}^2 + (k \boldsymbol{H})^{-1} \boldsymbol{\check{y}}_1$. Consequently, for any smooth function $s$ on $U \times \check{\mathbb{R}}^n$ that is rapidly decreasing in the variable $\boldsymbol{\check{y}}_1$, its product with the Gaussian factor resides in the space $\mathcal{C}^\infty(U, \mathcal{S}(\mathbb{R}^n \times \check{\mathbb{R}}^n))$. This allows us to define a continuous $\mathbb{C}$-linear embedding $\iota: \mathcal{C}^\infty(U, \mathcal{S}(\check{\mathbb{R}}^n)) \to \mathcal{C}^\infty(U \times \mathbb{R}^n \times \check{\mathbb{R}}^n)$ as:
	\begin{equation*}
		(\iota(s))(x, \boldsymbol{y}^2, \boldsymbol{\check{y}}_1) := s(x, \boldsymbol{\check{y}}_1) \cdot N_\Omega(\boldsymbol{u}^2, \boldsymbol{\check{y}}_1).
	\end{equation*}
	
	\begin{lemma}
		\label{Lemma: two kinds of fibrewise Schwartz functions}
		A function $s \in \mathcal{C}^\infty(U, \mathcal{S}(\mathbb{R}^n \times \check{\mathbb{R}}^n))$ lies in the image of $\iota$ if and only if the expression $s(x, \boldsymbol{y}^2, \boldsymbol{\check{y}_1}) \cdot \left( N_\Omega(\boldsymbol{u}^2, \boldsymbol{\check{y}}_1) \right)^{-1}$ is independent of the variable $\boldsymbol{y}^2$.
	\end{lemma}
	\begin{proof}
		The ``only if'' part follows directly from the definition of $\iota$. For the ``if'' part, suppose the given expression is independent of $\boldsymbol{y}^2$ and denote it by $f(x, \boldsymbol{\check{y}}_1)$. In the transformed coordinates introduced above, the Gaussian factor $N_\Omega(\boldsymbol{u}^2, \boldsymbol{\check{y}}_1)$ depends solely on the coordinate $\boldsymbol{u}^2 + (k \boldsymbol{H})^{-1} \boldsymbol{\check{y}}_1$ and is independent of $\boldsymbol{\check{y}}_1$. Since $s$ is rapidly decreasing in all fiber variables, it follows that $f$ inherits the necessary rapid decay in the $\boldsymbol{\check{y}}_1$-variable. Thus, $f \in \mathcal{C}^\infty(U, \mathcal{S}(\check{\mathbb{R}}^n))$ and $s = \iota(f)$, which completes the proof.
	\end{proof}
	
	\begin{proof}[\myproof{Proposition}{\ref{Proposition: sections of local mirror B-brane}}]
		For each $s \in \Gamma(U \times \mathbb{T}^{2n} \times \check{\mathbb{T}}^{2n}, \mathbf{L} \otimes \sqrt{K})$, assign it to a smooth function $\langle s \rangle$ on $U \times \mathbb{R}^n \times \check{\mathbb{R}}^n$ by
		\begin{equation}
			\label{Equation: modified Weil-Brezin coefficient of sections over fibre product}
			\langle s \rangle (x, \boldsymbol{y}^2, \boldsymbol{\check{y}}_1) := [s](x, \boldsymbol{y}^2, \boldsymbol{\check{y}}_1) \left( N_\Omega(\boldsymbol{u}^2, \boldsymbol{\check{y}}_1) \right)^{-1}.
		\end{equation}
		We claim that $s$ lies in $\Gamma(U \times \check{\mathbb{T}}^{2n}, \mathbf{\check{E}}) = H^0(U \times \mathbb{T}^{2n} \times \check{\mathbb{T}}^{2n}, \mathbf{L} \otimes \sqrt{K})$ if and only if $\langle s \rangle$ is independent of the variable $\boldsymbol{y}^2$. Once this claim holds, by Lemma \ref{Lemma: two kinds of fibrewise Schwartz functions}, the assignment $s \mapsto \langle s \rangle$ defines a bijective $\mathbb{C}$-linear map $\Gamma(U \times \check{\mathbb{T}}^{2n}, \mathbf{\check{E}}) \to \mathcal{C}^\infty(U, \mathcal{S}(\check{\mathbb{R}}^n))$.\par
		To establish that this is a topological isomorphism, we note that $\Gamma(U \times \check{\mathbb{T}}^{2n}, \mathbf{\check{E}})$ and $\mathcal{C}^\infty(U, \mathcal{S}(\check{\mathbb{R}}^n))$ embed into the Fr\'echet spaces $\Gamma(U \times \mathbb{T}^{2n} \times \check{\mathbb{T}}^{2n}, \mathbf{L} \otimes \sqrt{K})$ and $\mathcal{C}^\infty(U, \mathcal{S}(\mathbb{R}^n \times \check{\mathbb{R}}^n))$, respectively, as closed subspaces. It follows from the construction \eqref{Equation: modified Weil-Brezin coefficient of sections over fibre product} that the assignment $s \mapsto \langle s \rangle$ is continuous with respect to the compact-open $\mathcal{C}^\infty$-topology and the standard Fr\'echet topology on these spaces. By the Open Mapping Theorem for Fr\'echet spaces, this continuous linear bijection is necessarily a homeomorphism.\par
		Now, we prove the claim. Recall that $\nabla$ denotes the tensor product of the unitary connections on $\mathbf{L}$ and $\sqrt{K}$. Since the required CR structure on $U \times \mathbb{T}^{2n} \times \check{\mathbb{T}}^{2n}$ is the pullback of $T^{0, 1} \mathbb{T}^{2n}$, the section $s$ lies in $\Gamma(U \times \check{\mathbb{T}}^{2n}, \mathbf{\check{E}})$ if and only if 
		\begin{equation*}
			\left\langle \textstyle\sum_{j=1}^n h_j^{-1} \Omega_{ji} \nabla_{\partial_{y^j}} s + \nabla_{\partial_{y^{n+i}}} s \right\rangle (x, \boldsymbol{y}^2, \boldsymbol{\check{y}}_1) = 0 \quad \text{for all } 1 \leq i \leq n.
		\end{equation*}
		Fix $1 \leq i \leq n$. It follows from \eqref{Equation: unitary connection over fibre product} that
		\begin{equation*}
			\nabla_{\partial_{y^i}} = \tfrac{\partial}{\partial y^i} - 2\pi\sqrt{-1} (\check{y}_i + kh_i u^{n+i}) \quad \text{and} \quad \nabla_{\partial_{y^{n+i}}} = \tfrac{\partial}{\partial y^{n+i}} - 2\pi\sqrt{-1} \check{y}_{n+i}.
		\end{equation*}
		Direct computations using Proposition \ref{Proposition: fibrewise Weil Brezin transform} and \eqref{Equation: modified Weil-Brezin coefficient of sections over fibre product} yield
		\begin{align*}
			\left\langle \nabla_{\partial_{y^i}} s \right\rangle (x, \boldsymbol{y}^2, \boldsymbol{\check{y}}_1) = & -2\pi\sqrt{-1} (\check{y}_i + kh_iu^{n+i}) \cdot \langle s \rangle (x, \boldsymbol{y}^2, \boldsymbol{\check{y}}_1),\\
			\left\langle \nabla_{\partial_{y^{n+i}}} s \right\rangle (x, \boldsymbol{y}^2, \boldsymbol{\check{y}}_1) = & \left( \tfrac{\partial}{\partial y^{n+i}} + 2\pi\sqrt{-1} \textstyle\sum_{j=1}^n \Omega_{ij} \left( \tfrac{\check{y}_j}{h_j} + ku^{n+j} \right) \right) \langle s \rangle (x, \boldsymbol{y}^2, \boldsymbol{\check{y}}_1).
		\end{align*}
		Since $\Omega$ is symmetric, i.e. $\Omega_{ij} = \Omega_{ji}$, the above equalities imply that $s \in \Gamma(U \times \check{\mathbb{T}}^{2n}, \mathbf{\check{E}})$ if and only if $\tfrac{\partial \langle s \rangle}{\partial y^{n+i}} (x, \boldsymbol{y}^2, \boldsymbol{\check{y}}_1) = 0$ for all $1 \leq i \leq n$, i.e. $\langle s \rangle$ is independent of $\boldsymbol{y}^2$. Therefore, our claim holds. Moreover, the desired formula follows from Proposition \ref{Proposition: fibrewise Weil Brezin transform}, \eqref{Equation: modified Weil-Brezin coefficient of sections over fibre product} and Definition \ref{Definition: Guassian base}. The $\mathcal{C}^\infty(U)$-linearity is then immediate.
	\end{proof}
	
	For $\Omega, \Omega' \in \mathbb{H}_n$, let $\mathbf{\check{E}}$ and $\mathbf{\check{E}}'$ denote the corresponding bundles. The fiberwise pairing $\langle \, \cdot \, , \, \cdot \, \rangle$ in \eqref{Equation: fibrewise pairing of prequantum spaces} restricts to a pairing
	\begin{equation*}
		\Gamma(U \times \check{\mathbb{T}}^{2n}, \mathbf{\check{E}}) \times \Gamma(U \times \check{\mathbb{T}}^{2n}, \mathbf{\check{E}}') \to \mathcal{C}^\infty(U \times \check{\mathbb{T}}^{2n}).
	\end{equation*}
	Let $s \mapsto \langle s \rangle_\Omega$ and $s' \mapsto \langle s' \rangle_{\Omega'}$ denote the isomorphisms corresponding to $\Omega, \Omega'$, respectively, as in Proposition \ref{Proposition: sections of local mirror B-brane}. Combining Proposition \ref{Proposition: fibrewsie sesquilinear pairing} with standard Gaussian integral computations, one obtains, for $s \in \Gamma(U \times \check{\mathbb{T}}^{2n}, \mathbf{\check{E}})$ and $s' \in \Gamma(U \times \check{\mathbb{T}}^{2n}, \mathbf{\check{E}}')$, the formula
	\begin{equation}
		\label{Equation: Hermitian metric on local mirror B-brane}
		\langle s, s' \rangle (x, \check{y}) = \sum_{\check{m} \in \mathbb{Z}^{2n}} e^{2\pi\sqrt{-1} \boldsymbol{\check{m}}_2 \cdot \boldsymbol{\check{u}}_2} \cdot \langle s \rangle_\Omega \left( x, \boldsymbol{\check{u}}_1 - \boldsymbol{\check{m}}_1 - k\boldsymbol{H} \boldsymbol{\check{m}}_2 \right) \overline{\langle s' \rangle_{\Omega'} \left( x, \boldsymbol{\check{u}}_1 - \boldsymbol{\check{m}}_1 \right)}.
	\end{equation}
	In particular, this defines a Hermitian metric on $\mathbf{\check{E}}$.
	
	\subsubsection{Proof of Theorem \ref{Theorem: unitary connection on the local mirror B brane}}
	\quad\par
	\label{Subsubsection: Proof of Theorem one}
	Our strategy for proving Theorem \ref{Theorem: unitary connection on the local mirror B brane} is to compute the connection $\check{\nabla}$ explicitly. This requires a detailed understanding of the family orthogonal projection $\Pi$ defined in \eqref{Equation: family orthogonal projection}. Since $\Pi$ is $\mathcal{C}^\infty(U \times \check{\mathbb{T}}^{2n})$-linear, it suffices to carry out the computation fiberwise using an orthonormal basis of each fiber of $\mathbf{\check{E}}$.\par
	More precisely, at each point $(x, \check{y})$, the operator $\Pi$ restricts to the orthogonal projection
	\begin{equation*}
		\Gamma(\mathbb{T}^{2n}, \mathbf{L}_{(x, \check{y})} \otimes \sqrt{K}) \to H^0(\mathbb{T}^{2n}, \mathbf{L}_{(x, \check{y})} \otimes \sqrt{K}) \cong \mathbf{\check{E}}_{(x, \check{y})},
	\end{equation*}
	In practice, for any section $s \in \Gamma(\mathbb{T}^{2n}, \mathbf{L}_{(x, \check{y})} \otimes \sqrt{K})$, the projection $\Pi s$ is given by
	\begin{equation}
		\label{Equation: orthogonal projection in terms of orthonormal basis}
		\Pi s = \sum_{l \in \mathbb{Z}_{k\boldsymbol{H}}^n} \langle s, \sigma_l(x, \,\cdot\,, \check{y}) \rangle \cdot \sigma_l(x, \,\cdot\,, \check{y}),
	\end{equation}
	where $\sigma_l(x, \,\cdot\,, \check{y})$ are defined in Definition \ref{Definition: Guassian base}, forming an orthonormal basis of $\mathbf{\check{E}}_{(x, \check{y})}$.
	
	\begin{proof}[\myproof{Theorem}{\ref{Theorem: unitary connection on the local mirror B brane}}]
		The proof consists of explicit local computations of the connection $\check{\nabla}$. We outline the main steps.\par
		\textbf{Step 1.} 
		Let $s \in \Gamma(U \times \check{\mathbb{T}}^{2n}, \mathbf{\check{E}})$. Using \eqref{Equation: unitary connection over fibre product} together with Proposition \ref{Proposition: fibrewise Weil Brezin transform}, we compute the terms $[ \nabla_Y s ]$ for vector fields of the form $Y = \tfrac{\partial}{\partial x_j}$, $\tfrac{\partial}{\partial y^j}$ or $\tfrac{\partial}{\partial \check{y}_j}$.\par
		\textbf{Step 2.} 
		Using the results of Step 1 and \eqref{Equation: modified Weil-Brezin coefficient of sections over fibre product}, we compute
		$\langle \nabla_{Z_j} s \rangle$ and $\langle \nabla_{\overline{Z}_j} s \rangle$, where
		\begin{equation*}
			Z_j = \frac{\partial}{\partial \check{z}_j} = \frac{1}{2} \left( \frac{\partial}{\partial x_j} - \sqrt{-1} k \frac{\partial}{\partial \check{y}_j} \right) \quad \text{and} \quad \overline{Z}_j = \frac{\partial}{\partial \overline{\check{z}}_j} = \frac{1}{2} \left( \frac{\partial}{\partial x_j} + \sqrt{-1} k \frac{\partial}{\partial \check{y}_j} \right),
		\end{equation*}
		for $1 \leq j \leq 2n$. These expressions involve the terms $\langle \nabla_{\partial_{y^j}} s \rangle$.\par
		\textbf{Step 3.} 
		Combining Step 2 with the definition of $\check{\nabla}$ in \eqref{Equation: B brane connectino}, which involves the fiberwise orthogonal projection $\Pi$, we obtain for $1 \leq i \leq n$:
		\begin{align*}
			\left\langle \check{\nabla}_{Z_i} s \right\rangle = \left( \frac{\partial}{\partial \check{z}_i} - k\pi g^i \right) \langle s \rangle, \quad & \left\langle \check{\nabla}_{Z_{n+i}} s \right\rangle = \left( \frac{1}{2} \frac{\partial}{\partial x_{n+i}} - \pi\sqrt{-1} \frac{k}{h_i} \overline{\check{z}}_i - k\pi g^{n+i} \right) \langle s \rangle,\\
			\left\langle \check{\nabla}_{\overline{Z}_i} s \right\rangle = \left( \frac{\partial}{\partial \overline{\check{z}}_i} + k\pi g^i \right) \langle s \rangle, \quad & \left\langle \check{\nabla}_{\overline{Z}_{n+i}} s \right\rangle = \left( \frac{1}{2} \frac{\partial}{\partial x_{n+i}} - \pi\sqrt{-1} \frac{k}{h_i} \check{z}_i + k\pi g^{n+i} \right) \langle s \rangle.
		\end{align*}
		A key point in this step is the identity
		\begin{equation*}
			\Pi (\nabla_{\partial_{y^j}} s) = 0, \quad 1 \leq j \leq 2n.
		\end{equation*}
		This identity follows from the explicit expressions for $\langle \nabla_{\partial_{y^j}} s \rangle$ (see the proof of Proposition \ref{Proposition: sections of local mirror B-brane}), together with \eqref{Equation: orthogonal projection in terms of orthonormal basis} and standard Gaussian integral computations.\par
		\textbf{Conclusion.}
		The above formulas, combined with \eqref{Equation: Hermitian metric on local mirror B-brane}, show that $\check{\nabla}$ is unitary. Moreover, by Proposition \ref{Proposition: skew smith form of A brane}, we have $\tfrac{\partial g^j}{\partial x_i} = \tfrac{\partial g^i}{\partial x_j}$ for all $i,j$, which implies that the $(2,0)$- and $(0,2)$-components of the curvature vanish, and that the curvature has the stated form. Hence $\mathbf{\check{E}}$ is a Hermitian holomorphic vector bundle with Chern connection $\check{\nabla}$. Therefore, $(U \times \check{\mathbb{T}}^{2n}, \mathbf{\check{E}}, \check{\nabla})$ defines a B-brane. Since it is space-filling and its curvature is of the above form, it is semi-affine. This completes the proof.
	\end{proof}
	
	\subsection{Global construction of the mirror B-brane}
	\quad\par
	\label{Subsection: Global construction of mirror brane via gluing}
	In this subsection, we show that the local mirror B-branes glue to form a global one.\par
	We start with fixing an open cover $\{U_\alpha\}_{\alpha \in \mathcal{I}}$ of the base manifold $B$ together with a collection of triples $(x_\alpha, g_\alpha, \check{g}_\alpha)$ given as in Proposition \ref{Proposition: skew smith form of A brane} with respect to each $U_\alpha$. In particular, $x_\alpha$ are integral affine coordinates on $U_\alpha$. When $U_{\alpha\beta} := U_\alpha \cap U_\beta \neq \emptyset$, the coordinate change from $x_\alpha$ to $x_\beta$ induces a matrix $A_{\alpha\beta} \in \operatorname{GL}(2n, \mathbb{Z})$ with $A_{\alpha\beta} = A_{\beta\alpha}^{-1}$ such that $dx_\beta = A_{\alpha\beta} dx_\alpha$ (and thus $y_\alpha = A_{\alpha\beta}^{-T} y_\beta$). Furthermore, the cocycle condition
	\begin{equation*}
		A_{\gamma\alpha} A_{\beta\gamma} A_{\alpha\beta} = \operatorname{Id}
	\end{equation*}
	holds on each $U_{\alpha\beta\gamma} := U_\alpha \cap U_\beta \cap U_\gamma \neq \emptyset$.\par
	For each index $\alpha$, let $\mathbf{L}_\alpha$ be the local model of $\pi^*L^{\otimes k} \otimes \mathcal{P}$ over $X_{U_\alpha} \times_{U_\alpha} \check{X}_{U_\alpha}$ given as in the last subsection. Transition maps between the skew-Smith forms of $L$ induce a collection $\left\{\phi_{\alpha\beta}\right\}$ with $(\alpha, \beta)$ ranging over pairs of indices satisfying $U_{\alpha\beta} \neq \emptyset$, where \begin{equation*}
		\phi_{\alpha\beta}: \mathbf{L}_\beta \vert_{U_{\alpha\beta} \times \mathbb{T}^{2n} \times \check{\mathbb{T}}^{2n}} \to \mathbf{L}_\alpha \vert_{U_{\alpha\beta} \times \mathbb{T}^{2n} \times \check{\mathbb{T}}^{2n}}
	\end{equation*}
	is an isomorphism of Hermitian line bundles with unitary connection covering the coordinate change from $(x_\beta, y_\beta, \check{y}_\beta)$ to $(x_\alpha, y_\alpha, \check{y}_\alpha)$, such that $\phi_{\alpha\beta} = \phi_{\beta\alpha}^{-1}$ and that for each triple of indices $(\alpha, \beta, \gamma)$ with $U_{\alpha\beta\gamma} \neq \emptyset$,
	\begin{equation*}
		\phi_{\alpha\beta} \circ \phi_{\beta\gamma} \circ \phi_{\gamma\alpha} = \operatorname{Id}.
	\end{equation*}
	We refine $\{U_\alpha\}_{\alpha \in \mathcal{I}}$ to a new open cover $\{U_{\alpha; \Omega}\}_{\alpha \in \mathcal{I}, \Omega \in \mathbb{H}_n}$ of $B$, where we simply set $U_{\alpha; \Omega} := U_\alpha$. The additional index $\Omega$ serves to record the auxiliary data used in the construction of local models. For each $\alpha \in \mathcal{I}$ and each $\Omega \in \mathbb{H}_n$, let $\left( \mathbf{\check{E}}_{\alpha, \Omega}, \check{\nabla}^{\alpha, \Omega} \right)$ denote the corresponding local mirror B-brane constructed in the previous section.\par
	The construction of transition maps between these local mirror B-branes proceeds in two steps. In the first step, we fix an index $\alpha$ and consider two arbitrary elements $\Omega, \Omega' \in \mathbb{H}_n$.
	
	\begin{definition}
		\label{Definition: family BKS pairing map}
		Define the \emph{family Blattner--Kostant--Sternberg (BKS) pairing map}
		\begin{equation*}
			\operatorname{BKS}_{\Omega, \Omega'}^\alpha: \left( \Gamma\left(U_\alpha \times \check{\mathbb{T}}^{2n}, \mathbf{\check{E}}_{\alpha, \Omega'}\right), \check{\nabla}^{\alpha, \Omega'} \right) \to \left( \Gamma\left(U_\alpha \times \check{\mathbb{T}}^{2n}, \mathbf{\check{E}}_{\alpha, \Omega}\right), \check{\nabla}^{\alpha, \Omega} \right)
		\end{equation*}
		to be the $\mathcal{C}^\infty(U_\alpha \times \check{\mathbb{T}}^{2n})$-linear map uniquely determined by the following property: for each $(x, \check{y}) \in U_\alpha \times \mathbb{R}^{2n}$, the induced $\mathbb{C}$-linear map on fibers:
		\begin{equation*}
			\left( \mathbf{\check{E}}_{\alpha, \Omega'} \right)_{(x, \check{y})} = H^0(\mathbb{T}^{2n}, \mathbf{L}_{(x, \check{y})} \otimes \sqrt{K_{\Omega'}}) \to H^0(\mathbb{T}^{2n}, \mathbf{L}_{(x, \check{y})} \otimes \sqrt{K_\Omega}) = \left( \mathbf{\check{E}}_{\alpha, \Omega} \right)_{(x, \check{y})},
		\end{equation*}
		is the Blattner--Kostant--Sternberg pairing map.
	\end{definition}
	
	Similar to the family orthogonal projection \eqref{Equation: family orthogonal projection}, for any $s \in \left( \mathbf{\check{E}}_{\alpha, \Omega'} \right)_{(x, \check{y})}$,
	\begin{equation*}
		\operatorname{BKS}_{\Omega, \Omega'}^\alpha(s) = \sum_{l \in \mathbb{Z}_{k\boldsymbol{H}}^n} \langle s, \sigma_{\alpha, \Omega, l}(x, \,\cdot\,, \check{y}) \rangle \cdot \sigma_{\alpha, \Omega, l}(x, \,\cdot\,, \check{y}),
	\end{equation*}
	where $\sigma_{\alpha, \Omega, l}(x, \,\cdot\,, \check{y})$ are defined in Definition \ref{Definition: Guassian base}, forming an orthonormal basis of $\left( \mathbf{\check{E}}_{\alpha, \Omega'} \right)_{(x, \check{y})}$, and $\langle \,\cdot\,,\,\cdot\, \rangle$ denotes the sesquilinear pairing of $\left( \mathbf{\check{E}}_{\alpha, \Omega'} \right)_{(x, \check{y})}$ and $\left( \mathbf{\check{E}}_{\alpha, \Omega} \right)_{(x, \check{y})}$.
	
	\begin{proposition}
		\label{Proposition: family BKS pairing map}
		For all index $\alpha$ and elements $\Omega, \Omega' \in \mathbb{H}_n$, $\operatorname{BKS}_{\Omega, \Omega'}^\alpha$ is an isomorphism of Hermitian vector bundles with unitary connection.
	\end{proposition}
	\begin{proof}
		By Proposition \ref{Proposition: orthonormal basis of geometric quantization} (c.f. Subsection 4.1 of \cite{BaiMouNun2010}), for all $l \in \mathbb{Z}_{k\boldsymbol{H}}^n$ and all point $(x, \check{y})$,
		\begin{equation*}
			\operatorname{BKS}_{\Omega, \Omega'}^\alpha(\sigma_{\alpha, \Omega', l}(x, \,\cdot\,, \check{y})) = \sigma_{\alpha, \Omega, l}(x, \,\cdot\,, \check{y}).
		\end{equation*}
		It then follows from Proposition \ref{Proposition: sections of local mirror B-brane} that, for all $s \in \Gamma\left(U_\alpha \times \check{\mathbb{T}}^{2n}, \mathbf{\check{E}}_{\alpha, \Omega'}\right)$,
		\begin{equation*}
			\left\langle \operatorname{BKS}_{\Omega, \Omega'}^\alpha s \right\rangle_\Omega = \langle s \rangle_{\Omega'}.
		\end{equation*}
		Then applying \eqref{Equation: Hermitian metric on local mirror B-brane} and Theorem \ref{Theorem: unitary connection on the local mirror B brane} completes the proof.
	\end{proof}
	
	In the second step, we fix $\Omega \in \mathbb{H}_n$ and consider two indices $\alpha, \beta$ such that $U_{\alpha\beta} \neq \emptyset$. Let $A_{\alpha\beta} \cdot \Omega$ be the action of $A_{\alpha\beta}$ on $\Omega$ given as in \eqref{Equation: twisted action on upper half space}. The tensor product of $\phi_{\alpha\beta}$ with the transform:
	\begin{equation*}
		\sqrt{K_{\beta, \Omega}} \to \sqrt{K_{\alpha, A_{\alpha\beta} \cdot \Omega}}, \quad \frac{\sqrt{d^n w_{\beta, \Omega}}}{\sqrt[4]{\det \left( 2\operatorname{Im} \Omega \right)}} \mapsto \frac{\sqrt{d^n w_{\alpha, A_{\alpha\beta} \cdot \Omega}}}{\sqrt[4]{\det \left( 2\operatorname{Im} (A_{\alpha\beta} \cdot \Omega) \right)}},
	\end{equation*}
	restricts to an isomorphism of Hermitian vector bundles with unitary connections:
	\begin{equation*}
		\phi_{\alpha\beta; \Omega}: \left( \Gamma\left( U_{\alpha\beta} \times \check{\mathbb{T}}^{2n}, \mathbf{\check{E}}_{\beta, \Omega} \right), \check{\nabla}^{\beta, \Omega} \right) \to \left( \Gamma\left( U_{\alpha\beta} \times \check{\mathbb{T}}^{2n}, \mathbf{\check{E}}_{\alpha, A_{\alpha\beta} \cdot \Omega} \right), \check{\nabla}^{\alpha, A_{\alpha\beta} \cdot \Omega} \right)
	\end{equation*}
	Now, we combine the two steps to construct a general transition map with the aid of the following proposition.
	
	\begin{proposition}
		\label{Proposition: commutativity of transitions}
		Suppose $\alpha, \beta$ are indices with $U_{\alpha\beta} \neq \emptyset$ and $\Omega, \Omega' \in \mathbb{H}_n$. Then the following diagram commutes:
		\begin{center}
			\begin{tikzcd}
				\Gamma\left( U_{\alpha\beta} \times \check{\mathbb{T}}^{2n}, \mathbf{\check{E}}_{\beta, \Omega'} \right) \ar[rr, "\operatorname{BKS}_{A_{\alpha\beta}^{-1} \cdot \Omega, \Omega'}^\beta"] \ar[d, "\phi_{\alpha\beta; \Omega'}"'] && \Gamma\left( U_{\alpha\beta} \times \check{\mathbb{T}}^{2n}, \mathbf{\check{E}}_{\beta, A_{\alpha\beta}^{-1} \cdot \Omega} \right) \ar[d, "\phi_{\alpha\beta; A_{\alpha\beta}^{-1} \cdot \Omega}"] \\
				\Gamma\left( U_{\alpha\beta} \times \check{\mathbb{T}}^{2n}, \mathbf{\check{E}}_{\alpha, A_{\alpha\beta} \cdot \Omega'} \right) \ar[rr, "\operatorname{BKS}_{\Omega, A_{\alpha\beta} \cdot \Omega'}^\alpha"'] && \Gamma\left( U_{\alpha\beta} \times \check{\mathbb{T}}^{2n}, \mathbf{\check{E}}_{\alpha, \Omega} \right)
			\end{tikzcd}
		\end{center}
	\end{proposition}
	\begin{proof}
		By construction, the transition map $\phi_{\alpha\beta;\Omega'}$ is induced from the isomorphism $\phi_{\alpha\beta}$ of the underlying family of prequantum line bundles --- covering the symplectomorphism of the torus determined by the coordinate change --- together with the natural transformation of half-forms corresponding to $\Omega' \mapsto A_{\alpha\beta}\cdot \Omega'$. The same holds for $\phi_{\alpha\beta; A_{\alpha\beta}^{-1} \cdot \Omega}$. In particular, these maps are unitary, preserve the connections, and respect the spaces of holomorphic sections.\par
		On the other hand, the operators $\operatorname{BKS}_{\Omega, A_{\alpha\beta} \cdot \Omega'}^\alpha$ and $\operatorname{BKS}_{A_{\alpha\beta}^{-1} \cdot \Omega, \Omega'}^\beta$ are defined fiberwise via the sesquilinear pairing induced by the Hermitian and holomorphic structures on the half-form–corrected prequantum line bundles. Consequently, each map is canonical and completely determined by these structures.\par
		Combining these two facts, the diagram commutes: applying a BKS pairing map followed by the half-form–corrected prequantum line bundle isomorphism for the target complex structure gives the same result as first applying the isomorphism for the original complex structure and then the corresponding BKS pairing map. This holds because the isomorphisms intertwine the Hermitian structures, the connections, and the spaces of holomorphic sections.
	\end{proof}
	
	\begin{definition}
		For all indices $\alpha, \beta$ with $U_{\alpha\beta} \neq \emptyset$ and $\Omega, \Omega' \in \mathbb{H}_n$, define the isomorphism of Hermitian vector bundles with unitary connections
		\begin{equation*}
			\check{\phi}_{\alpha\beta; \Omega, \Omega'}: \left( \Gamma\left( U_{\alpha\beta} \times \check{\mathbb{T}}^{2n}, \mathbf{\check{E}}_{\beta, \Omega'} \right), \check{\nabla}^{\beta, \Omega'} \right) \to \left( \Gamma\left( U_{\alpha\beta} \times \check{\mathbb{T}}^{2n}, \mathbf{\check{E}}_{\alpha, \Omega} \right), \check{\nabla}^{\alpha, \Omega} \right)
		\end{equation*}
		by
		\begin{equation*}
			\check{\phi}_{\alpha\beta; \Omega, \Omega'} := \phi_{\alpha\beta; A_{\alpha\beta}^{-1} \cdot \Omega} \circ \operatorname{BKS}_{A_{\alpha\beta}^{-1} \cdot \Omega, \Omega'}^\beta = \operatorname{BKS}_{\Omega, A_{\alpha\beta} \cdot \Omega'}^\alpha \circ \phi_{\alpha\beta; \Omega'}.
		\end{equation*}
	\end{definition}
	
	In the statement of the following theorem, we explicitly include the superscript $(k)$ to highlight the dependence of all relevant objects on the level $k$.
	
	\begin{theorem}
		\label{Proposiiton: Global mirror B-brane}
		For each level $k \in \mathbb{Z}_{>0}$, the isomorphisms $\check{\phi}_{\alpha\beta; \Omega, \Omega'}^{(k)}$ satisfy the cocycle condition; hence the local data $(\mathbf{\check{E}}_{\alpha, \Omega}^{(k)}, \check{\nabla}^{(k), \alpha, \Omega})$ glue to a semi-affine B-brane
		\begin{equation*}
			\check{\mathcal{B}}_{\operatorname{cc}}^{(k)} := (\check{X}^{(k)}, \check{E}^{(k)}, \nabla^{\check{E}^{(k)}})
		\end{equation*}
		on $\check{X}^{(k)}$, which we call the \emph{mirror B-brane}  of $\mathcal{B}_{\operatorname{cc}}^{(k)}$.
	\end{theorem}
	\begin{proof}
		Let $\alpha, \beta, \gamma$ be indices with $U_{\alpha\beta\gamma} \neq \emptyset$ and $\Omega, \Omega', \Omega'' \in \mathbb{H}_n$. Using the compatibility of the family BKS pairing maps with the transition maps (Proposition \ref{Proposition: commutativity of transitions}) together with the cocycle condition $A_{\alpha\beta} A_{\beta\gamma} A_{\gamma\alpha} = \operatorname{Id}$, the triple composition
		\begin{equation}
			\label{Equation: cocyle condition of mirror B brane}
			\check{\phi}_{\alpha\beta; \Omega, \Omega'} \circ \check{\phi}_{\beta\gamma; \Omega', \Omega''} \circ \check{\phi}_{\gamma\alpha; \Omega'', \Omega}
		\end{equation}
		decomposes as the composition of the following two terms:
		\begin{enumerate}
			\item the triple composition of family BKS pairing maps
			\begin{equation*}
				\operatorname{BKS}_{\Omega, A_{\alpha\beta} \cdot \Omega'}^\alpha \circ \operatorname{BKS}_{A_{\alpha\beta} \cdot \Omega', A_{\alpha\beta} A_{\beta\gamma} \cdot \Omega''}^\alpha \circ \operatorname{BKS}_{A_{\alpha\beta} A_{\beta\gamma} \cdot \Omega'', \Omega}^\alpha,
			\end{equation*}
			which is the identity by the proof of Proposition \ref{Proposition: family BKS pairing map}; and
			\item the triple composition $\phi_{\alpha\beta; A_{\beta\gamma} A_{\gamma\alpha} \cdot \Omega} \circ \phi_{\beta\gamma; A_{\gamma\alpha} \cdot \Omega} \circ \phi_{\gamma\alpha; \Omega}$, which is also the identity due to the cocycle condition $\phi_{\alpha\beta} \circ \phi_{\beta\gamma} \circ \phi_{\gamma\alpha} = \operatorname{Id}$.
		\end{enumerate}
		Therefore, the composition \eqref{Equation: cocyle condition of mirror B brane} is the identity, as required.
	\end{proof}
	
	\section{The mirror theorem}
	\label{Section: Brane quantization via SYZ transforms}
	This section is devoted to the main objective of the paper: establishing a mirror statement asserting the existence of a graded algebra isomorphism
	\begin{center}
		\begin{tikzcd}
			\operatorname{Hom}_A(\mathcal{B}_{\operatorname{cc}}^{(k)}, \mathcal{B}_{\operatorname{cc}}^{(k)}) \ar[r, "\cong"] & \operatorname{Hom}_B(\check{\mathcal{B}}_{\operatorname{cc}}^{(k)}, \check{\mathcal{B}}_{\operatorname{cc}}^{(k)}).
		\end{tikzcd}
	\end{center}
	Recall from Theorem \ref{Proposiiton: Global mirror B-brane} that the mirror B-brane $\check{\mathcal{B}}_{\operatorname{cc}}^{(k)} = (\check{X}^{(k)}, \check{E}^{(k)})$ is obtained by gluing local data defined on open subsets of the form $X_U \subset X$, where $U \subset B$ is as in Proposition \ref{Proposition: skew smith form of A brane}. For each such subset, one chooses a parameter $\Omega \in \mathbb{H}_n$, which determines a family of torus-invariant K\"ahler polarizations and thereby provides an explicit local model of the B-brane.\par
	A natural first approach in constructing the desired algebra isomorphism is via a family of Toeplitz operators, which encode --- at least asymptotically --- the action of deformation quantization on geometric quantization in K\"ahler geometry \cite{BorMeiSch1994, Gui1995, Sch2000}.\par
	However, the non-formal quantization constructed in Section \ref{Section: Holomorphic deformation quantization}, which realizes the A-model morphism space $\operatorname{Hom}_A(\mathcal{B}_{\operatorname{cc}}^{(k)}, \mathcal{B}_{\operatorname{cc}}^{(k)})$, is of Moyal type. As shown in \cite{And2005} (see also \cite{LeuYau2023}), the Berezin–Toeplitz star product and the Moyal product on a symplectic torus are gauge equivalent, with the gauge transformation governed by the Dolbeault Laplacian. This observation motivates our actual construction: we obtain the desired isomorphism by introducing a suitable twist of the family of Toeplitz operators. Our main result is formulated as follows.
	
	\begin{theorem}
		\label{Theorem: mirror statement}
		Let $k \in \mathbb{Z}_{>0}$. Then the mirror transform constructed via family twisted Toeplitz operators
		\begin{equation*}
			\Phi_U^{(k)}: \left( \Omega^{0, *}(X_U), \star_{k^{-1}}, \overline{\partial}_I \right) \to \left( \Omega^{0, *}(\check{X}_U^{(k)}, \operatorname{End}(\check{E}^{(k)})), \circ, \overline{\partial}_{\operatorname{End}( \check{E}^{(k)})} \right)
		\end{equation*}
		defines an isomorphism of sheaves of differential graded algebras on $B$. Consequently, this induces a graded algebra isomorphism
		\begin{equation*}
			(H^*(X,\mathcal{O}_X), \star_{k^{-1}}) \cong (H^*(\check{X},\operatorname{End}(\check{E}^{(k)})), \circ).
		\end{equation*}
	\end{theorem}
	
	The construction proceeds via a local-to-global argument (local in the base direction). In Subsection \ref{Subsection: local construction of mirror map}, we construct the mirror transform $\Phi^{(k)}$ locally. In Subsection \ref{Subsection: global construction of mirror map}, we glue these local constructions to obtain a global transform. Finally, in Subsection \ref{Subsection: proof of mirror statement}, we complete the proof of Theorem \ref{Theorem: mirror statement}.
	
	\subsection{Local construction of the mirror transform}
	\quad\par
	\label{Subsection: local construction of mirror map}
	Throughout this subsection, we restrict attention to an open subset $U \subset B$ over which $\mathcal{B}_{\operatorname{cc}}$ admits a skew-Smith form on $X_U \cong U \times \mathbb{T}^{2n}$, and we fix a choice $\Omega \in \mathbb{H}_n$. It is instructive to first examine the construction of family Toeplitz operators.
	
	\subsubsection{Motivation: family Toeplitz operators}
	\quad\par
	
	\begin{definition}
		Let $f \in \mathcal{C}^\infty(U \times \mathbb{T}^{2n})$. The \emph{family Toeplitz operator} associated with $f$ is defined by
		\begin{equation*}
			T_f := \Pi \circ \operatorname{m}_{(\pi_U^*f)}: \Gamma(U \times \check{\mathbb{T}}^{2n}, \mathbf{\check{E}}) \to \Gamma(U \times \check{\mathbb{T}}^{2n}, \mathbf{\check{E}}),
		\end{equation*}
		where $\Pi$ denotes the family orthogonal projection and $\operatorname{m}_{(\pi_U^*f)}$ is multiplication by the pullback of $f$ along the canonical projection $\pi_U: U \times \mathbb{T}^{2n} \times \check{\mathbb{T}}^{2n} \to U \times \mathbb{T}^{2n}$.
	\end{definition}
	
	By construction, for each $f \in \mathcal{C}^\infty(U \times \mathbb{T}^{2n})$, the operator $T_f$ defines a smooth section of $\operatorname{End}(\mathbf{\check{E}})$ with the property that, for every $(x, \check{y})$, its action on the fiber $\mathbf{\check{E}}_{(x, \check{y})}$ coincides with the Toeplitz operator on $H^0(\mathbb{T}^{2n}, \mathbf{L}_{(x, \check{y})} \otimes \sqrt{K})$ associated with the restriction of $f$ to the fiber $\{x\} \times \mathbb{T}^{2n}$. Consequently, we obtain a $\mathcal{C}^\infty(U)$-linear map
	\begin{equation*}
		T: \mathcal{C}^\infty(U \times \mathbb{T}^{2n}) \to \Gamma(U \times \check{\mathbb{T}}^{2n}, \operatorname{End}(\mathbf{\check{E}}))
	\end{equation*}
	given by $f \mapsto T_f$.\par
	Recall that each fiber $\{x\} \times \mathbb{T}^{2n}$ carries the symplectic form $\tfrac{\sqrt{-1}}{2} dw \wedge (\operatorname{Im} \Omega)^{-1} d\overline{w}$, where $w$ are the complex coordinates described in \eqref{Equation: complex coordinates on fibre torus}. To define a gauge transformation connecting the non-formal Moyal product $\star_{k^{-1}}$ to the fiberwise Berezin--Toeplitz product, we introduce the following second order differential operator on $U \times \mathbb{T}^{2n}$:
	\begin{equation}
		\label{Equation: fibrewise Dolbeault Laplacian}
		\Delta = -\frac{1}{2\pi} \frac{\partial}{\partial w} \cdot (\operatorname{Im} \Omega) \frac{\partial}{\partial \overline{w}},
	\end{equation}
	viewed as a family of Dolbeault Laplacians up to normalization. For each $k \in \mathbb{Z}_{>0}$, the expected non-formal gauge transformation is $e^{\frac{1}{k} \Delta}$. However, this operator corresponds to a backward heat flow and is ill-posed as a map $\mathcal{C}^\infty(U \times \mathbb{T}^{2n}) \to \mathcal{C}^\infty(U \times \mathbb{T}^{2n})$.\par
	Nonetheless, its composition with the family Toeplitz map, $T \circ e^{\frac{1}{k} \Delta}$, can be computed explicitly on individual Fourier modes.
	
	\begin{lemma}
		\label{Lemma: Toeplitz operator}
		Let $m = (m_1, ..., m_{2n}) \in \mathbb{Z}^{2n}$ and $f \in \mathcal{C}^\infty(U \times \mathbb{T}^{2n})$ be defined by
		\begin{equation*}
			f(x, y) = \exp \left( 2\pi\sqrt{-1} m \cdot u \right).
		\end{equation*}
		Let $l, \widetilde{l} \in \mathbb{Z}_{k\boldsymbol{H}}^n$ be such that $\widetilde{q} := (k\boldsymbol{H})^{-1} (l + \boldsymbol{m}_1 - \widetilde{l}) \in \mathbb{Z}^n$. Then for all $(x, \check{y})$,
		\begin{equation*}
			T_{\exp \left( \frac{1}{k} \Delta \right)(f)} (\sigma_l(x, \,\cdot\,, \check{y})) = \exp \left( \tfrac{2\pi\sqrt{-1}}{k} \left( -\boldsymbol{m}_2 \cdot \boldsymbol{H}^{-1} (\boldsymbol{\check{u}}_1 + \tfrac{1}{2} \boldsymbol{m}_1 - \widetilde{l}) + k \widetilde{q} \cdot \boldsymbol{\check{u}}_2 \right) \right) \cdot \sigma_{\widetilde{l}}(x, \,\cdot\,, \check{y}),
		\end{equation*}
		where $\sigma_l(x, \,\cdot\,, \check{y})$ is given in Definition \ref{Definition: Guassian base}.
	\end{lemma}
	
	By the definition of the Toeplitz operator,
	\begin{equation*}
		T_{f_\Omega} (\sigma_l) = \sum_{\widetilde{l} \in \mathbb{Z}_{k\boldsymbol{H}}^n} \left\langle f_\Omega \sigma_l, \sigma_{\widetilde{l}} \right\rangle \sigma_{\widetilde{l}},
	\end{equation*}
	where, for brevity, we write $\sigma_l = \sigma_l(x, \,\cdot\,, \check{y})$ and $f_\Omega = \exp \left( \frac{1}{k} \Delta \right)(f)$. Hence, the statement reduces to the computation of the matrix coefficients $\left\langle f_\Omega \sigma_l, \sigma_{\widetilde{l}} \right\rangle$, which are expressed as Gaussian integrals. These computations are analogous to those carried out in the proof of Theorem 5 in \cite{And2005}, although the present setting involves additional parameters and more cumbersome notations. We therefore omit the details.\par
	We also provide an explicit formula for the action of a single Fourier mode on sections of $\mathbf{\check{E}}$ via the composition $T \circ e^{\frac{1}k\Delta}$.
	
	\begin{lemma}
		\label{Lemma: family Toeplitz operator}
		Let $m \in \mathbb{Z}^{2n}$ and define the function $f \in \mathcal{C}^\infty(U \times \mathbb{T}^{2n})$ by
		\begin{equation*}
			f(x, y) = \exp \left( 2\pi\sqrt{-1} m \cdot u \right).
		\end{equation*}
		Then for all $s \in \Gamma(U \times \check{\mathbb{T}}^{2n}, \mathbf{\check{E}})$,
		\begin{equation}
			\label{Equation: family Toeplitz operator}
			\left\langle T_{\exp\left(\frac{1}{k} \Delta\right)(f)} s \right\rangle (x, \boldsymbol{\check{y}}_1) = \exp \left( \tfrac{-2\pi\sqrt{-1}}{k} \boldsymbol{m}_2 \cdot \boldsymbol{H}^{-1} \left(\boldsymbol{\check{y}}_1 + \tfrac{1}{2} \boldsymbol{m}_1 \right) \right) \cdot \langle s \rangle (x, \boldsymbol{\check{y}}_1 + \boldsymbol{m}_1).
		\end{equation}
	\end{lemma}
	\begin{proof}
		Recall from Proposition \ref{Proposition: sections of local mirror B-brane} that
		\begin{equation*}
			s(x, y, \check{y}) = \sum_{l \in \mathbb{Z}_{k\boldsymbol{H}}^n} \sigma_l(x, y, \check{y}) \sum_{\check{q} \in \mathbb{Z}^n} \langle s \rangle \left( x, \boldsymbol{\check{u}}_1 - l - k\boldsymbol{H}\check{q} \right) \cdot e^{2\pi\sqrt{-1} \check{q} \cdot \boldsymbol{\check{u}}_2}.
		\end{equation*}
		For all $l \in \mathbb{Z}_{k\boldsymbol{H}}^n$, let $\widetilde{l}$ be the unique element in $\mathbb{Z}_{k\boldsymbol{H}}^n$ such that $l + \boldsymbol{m}_1 - \widetilde{l} \in k\boldsymbol{H}\mathbb{Z}^n$, and let $\widetilde{q}(l) = (k\boldsymbol{H})^{-1} (l + \boldsymbol{m}_1 - \widetilde{l}) \in \mathbb{Z}^n$. By Lemma \ref{Lemma: Toeplitz operator},
		\begin{align*}
			\left( T_{\exp\left(\frac{1}{k} \Delta\right)(f)} s \right) (x, y, \check{y}) = & \sum_{l \in \mathbb{Z}_{k\boldsymbol{H}}^n} \sigma_{\widetilde{l}}(x, y, \check{y}) \sum_{\check{q} \in \mathbb{Z}^n} \langle s \rangle \left( x, \boldsymbol{\check{u}}_1 - l - k\boldsymbol{H}\check{q} \right)\\
			& \cdot e^{\frac{2\pi\sqrt{-1}}{k} \left( -\boldsymbol{m}_2 \cdot \boldsymbol{H}^{-1} (\boldsymbol{\check{u}}_1 + \frac{1}{2} \boldsymbol{m}_1 - \widetilde{l}) + k \widetilde{q}(l) \cdot \boldsymbol{\check{u}}_2 \right) + 2\pi\sqrt{-1} \check{q} \cdot \boldsymbol{\check{u}}_2}.
		\end{align*}
		For the right hand side of the above equality, we first replace the index $\check{q}$ by $\check{q} - \widetilde{q}(l)$, then apply the equality $- l + k\boldsymbol{H} \widetilde{q}(l) = \boldsymbol{m}_1 - \widetilde{l}$, and finally multiply the summand by the factor $e^{-2\pi\sqrt{-1} \boldsymbol{m}_2 \cdot (k\boldsymbol{H})^{-1} (-k\boldsymbol{H}\check{q})}$, which is indeed equal to $1$. As a result, it yields the term
		\begin{align*}
			& \sum_{l \in \mathbb{Z}_{k\boldsymbol{H}}^n} \sigma_{\widetilde{l}}(x, y, \check{y}) \sum_{\check{q} \in \mathbb{Z}^n} \langle s \rangle \left( x, \boldsymbol{\check{u}}_1 + \boldsymbol{m}_1 - \widetilde{l} - k\boldsymbol{H}\check{q} \right) e^{\frac{2\pi\sqrt{-1}}{k} \left( -\boldsymbol{m}_2 \cdot \boldsymbol{H}^{-1} (\boldsymbol{\check{u}}_1 + \frac{1}{2} \boldsymbol{m}_1 - \widetilde{l} - k\boldsymbol{H} \check{q}) \right) + 2\pi\sqrt{-1} \check{q} \cdot \boldsymbol{\check{u}}_2}.
		\end{align*}
		Since $l \mapsto \widetilde{l}$ forms a bijection $\mathbb{Z}_{k\boldsymbol{H}}^n \to \mathbb{Z}_{k\boldsymbol{H}}^n$, we can replace $\widetilde{l}$ by $l$ in the summand of the above term, concluding the proof.
	\end{proof}
	
	In particular, for $m \in \mathbb{Z}^{2n}$ and the function $f \in \mathcal{C}^\infty(U \times \mathbb{T}^{2n})$ defined by
	\begin{equation*}
		f(z) = e^{2\pi (-\boldsymbol{m}_2 \cdot \boldsymbol{H}^{-1} \boldsymbol{z}^1 + \boldsymbol{m}_1 \cdot \boldsymbol{H}^{-1} \boldsymbol{z}^2)},
	\end{equation*}
	we have for all $s \in \Gamma(U \times \check{\mathbb{T}}^{2n}, \mathbf{\check{E}})$,
	\begin{equation*}
		\left\langle T_{\exp\left(\frac{1}{k} \Delta\right)(f)} s \right\rangle (x, \boldsymbol{\check{y}}_1) = e^{2\pi \left( -\boldsymbol{m}_2 \cdot \boldsymbol{H}^{-1} \left( \boldsymbol{\check{z}}_1 + \frac{\sqrt{-1}}{2k} \boldsymbol{m}_1 \right) + \boldsymbol{m}_1 \cdot \boldsymbol{H}^{-1} \boldsymbol{x}_2 \right)} \cdot \langle s \rangle (x, \boldsymbol{\check{y}}_1 + \boldsymbol{m}_1).
	\end{equation*}
	
	Our next goal is to show that the composition $T \circ e^{\frac{1}k\Delta}$ extends to a map from $\mathcal{C}^\infty(U \times \mathbb{T}^{2n})$ to $\Gamma(U \times \check{\mathbb{T}}^{2n}, \operatorname{End}(\mathbf{\check{E}}))$. We begin with some preparatory observations.
	
	\subsubsection{Operator-theoretic characterization of $\Gamma(U \times \check{\mathbb{T}}^{2n}, \operatorname{End}(\mathbf{\check{E}}))$}
	\quad\par
	In Proposition \ref{Proposition: sections of local mirror B-brane}, we established an isomorphism of topological vector spaces
	\begin{equation*}
		\Gamma( U \times \check{\mathbb{T}}^{2n}, \mathbf{\check{E}} ) \to \mathcal{C}^\infty(U, \mathcal{S}(\check{\mathbb{R}}^n)), \quad s \mapsto \langle s \rangle.
	\end{equation*}
	This induces an injective $\mathbb{C}$-linear map
	\begin{equation}
		\label{Equation: characterization of sectino of endo-bundle}
		\Gamma(U \times \check{\mathbb{T}}^{2n}, \operatorname{End}(\mathbf{\check{E}})) \to \operatorname{End}_\mathbb{C} ( \mathcal{C}^\infty(U, \mathcal{S}(\check{\mathbb{R}}^n)) ), \quad \Psi \mapsto \langle \Psi \rangle,
	\end{equation}
	where $\langle \Psi \rangle$ is the unique $\mathbb{C}$-linear operator satisfying $\langle \Psi \rangle \langle s \rangle = \langle \Psi(s) \rangle$ for all $s \in \Gamma( U \times \check{\mathbb{T}}^{2n}, \mathbf{\check{E}} )$.\par
	We now reformulate Lemma \ref{Lemma: family Toeplitz operator}. For $\boldsymbol{m}_1 \in \mathbb{Z}^n$, define the shift operator
	\begin{equation}
		\label{Equation: shift operator}
		\operatorname{S}_{\boldsymbol{m}_1}: \mathcal{C}^\infty(U \times \check{\mathbb{R}}^n) \to \mathcal{C}^\infty(U \times \check{\mathbb{R}}^n), \quad \operatorname{S}_{\boldsymbol{m}_1}(s)(x, \boldsymbol{\check{y}}_1) = s(x, \boldsymbol{\check{y}}_1 + \boldsymbol{m}_1).
	\end{equation}
	For a function $f \in \mathcal{C}^\infty(U \times \mathbb{T}^{2n})$, consider its fiberwise Fourier expansion
	\begin{equation}
		\label{Equation: half-fibrewise Fourier expansion}
		f(x, y) = \sum_{\boldsymbol{m}_1 \in \mathbb{Z}^n} f_{\boldsymbol{m}_1}(x, \boldsymbol{u}^2) \cdot e^{2\pi\sqrt{-1} \boldsymbol{m}_1 \cdot \boldsymbol{u}^1},
	\end{equation}
	where $f_{\boldsymbol{m}_1} \in \mathcal{C}^\infty(U \times \mathbb{T}^n)$. Define the rescaled torus $\check{\mathbb{T}}_{k\boldsymbol{H}}^n := \check{\mathbb{R}}^n / (k\boldsymbol{H} \mathbb{Z}^n)$, and associate to each $f_{\boldsymbol{m}_1}$ a function $\check{f}_{\boldsymbol{m}_1} \in \mathcal{C}^\infty(U \times \check{\mathbb{T}}_{k\boldsymbol{H}}^n)$ by
	\begin{equation}
		\label{Equation: half-fibrewise Fourier coefficient}
		\check{f}_{\boldsymbol{m}_1}(x, \boldsymbol{\check{y}}_1) := f_{\boldsymbol{m}_1}( x, -\tfrac{1}{k} \boldsymbol{H}^{-1} (\boldsymbol{\check{y}}_1 + \tfrac{1}{2} \boldsymbol{m}_1) ).
	\end{equation}
	By Lemma \ref{Lemma: family Toeplitz operator}, if only finitely many $f_{\boldsymbol{m}_1}$'s are nonzero and each has finitely many nonzero fiberwise Fourier modes in the $\boldsymbol{u}^2$-variable, then
	\begin{equation}
		\label{Equation: series of shift operators}
		\left\langle T_{\exp\left(\frac{1}{k} \Delta\right)(f)} \right\rangle = \sum_{\boldsymbol{m}_1 \in \mathbb{Z}^n} \check{f}_{\boldsymbol{m}_1} \operatorname{S}_{\boldsymbol{m}_1}.
	\end{equation}
	Here, by a slight abuse of notation, $\check{f}_{\boldsymbol{m}_1}$ denotes the operator of multiplication by this function, and $\operatorname{S}_{\boldsymbol{m}_1}$ its restriction to $\mathcal{C}^\infty(U, \mathcal{S}(\check{\mathbb{R}}^n))$.\par
	Our strategy for constructing the local mirror transform is to show that, for a general function $f \in \mathcal{C}^\infty(U \times \mathbb{T}^{2n})$, the right-hand side of \eqref{Equation: series of shift operators} remains well-defined and lies in the image of $\Gamma(U \times \check{\mathbb{T}}^{2n}, \operatorname{End}(\mathbf{\check{E}}))$ under the injective map \eqref{Equation: characterization of sectino of endo-bundle}. To this end, we first characterize this image. As before, we identify functions $f \in \mathcal{C}^\infty(U \times \check{\mathbb{T}}^n)$ with periodic smooth $\mathbb{C}$-valued functions on $U \times \check{\mathbb{R}}^n$.
	
	\begin{lemma}
		\label{Lemma: characterization of sections of endo-bundle}
		Let $\Xi: \mathcal{C}^\infty(U, \mathcal{S}(\check{\mathbb{R}}^n)) \to \mathcal{C}^\infty(U, \mathcal{S}(\check{\mathbb{R}}^n))$ be a $\mathbb{C}$-linear operator. Then there exists $\Psi \in \Gamma(U \times \check{\mathbb{T}}^{2n}, \operatorname{End}(\mathbf{\check{E}}))$ such that $\Xi = \langle \Psi \rangle$ if and only if $\Xi$ satisfies the following conditions:
		\begin{itemize}
			\item $\Xi$ is continuous;
			\item for all $f \in \mathcal{C}^\infty(U \times \check{\mathbb{T}}^n)$ and $s \in \mathcal{C}^\infty(U, \mathcal{S}(\check{\mathbb{R}}^n))$, $\Xi(f \cdot s) = f \cdot \Xi(s)$;
			\item for all $\check{q} \in \mathbb{Z}^n$, $\Xi \circ \operatorname{S}_{k\boldsymbol{H}\check{q}} = \operatorname{S}_{k\boldsymbol{H}\check{q}} \circ \Xi$, where $\operatorname{S}_{k\boldsymbol{H}\check{q}}$ is the shift operator defined in \eqref{Equation: shift operator}.
		\end{itemize}
	\end{lemma}
	\begin{proof}
		We first record a common observation. Using \eqref{Equation: sections in term of non-smooth unitary frame}, commutation with multiplication by smooth $\mathbb{C}$-valued functions on $U \times \check{\mathbb{T}}^{2n}$ independent of $\boldsymbol{\check{y}}_2$ together with commutation with multiplication by functions of the form $e^{2\pi\sqrt{-1} \boldsymbol{\check{m}}^2 \cdot \boldsymbol{\check{y}}_2}$ with $\boldsymbol{\check{m}}^2 \in \mathbb{Z}^n$ is equivalent to the module property with respect to $\mathcal{C}^\infty(U \times \check{\mathbb{T}}^n)$ and to commutation with the shift operators $\operatorname{S}_{k\boldsymbol{H}\check{q}}$.\par
		Suppose $\Psi \in \Gamma(U \times \check{\mathbb{T}}^{2n}, \operatorname{End}(\mathbf{\check{E}}))$. Via the isomorphism \eqref{Equation: charaterization of sections of mirror brane}, the operator
		\begin{equation*}
			\Psi: \Gamma(U \times \check{\mathbb{T}}^{2n}, \mathbf{\check{E}}) \to \Gamma(U \times \check{\mathbb{T}}^{2n}, \mathbf{\check{E}})
		\end{equation*}
		induces a continuous operator $\Xi := \langle \Psi \rangle$. Moreover, $\Psi$ is $\mathcal{C}^\infty(U \times \check{\mathbb{T}}^{2n})$-linear. it commutes with multiplication by all smooth functions, and in particular with those appearing in the above observation. The desired properties of $\Xi$ therefore follow.\par
		Conversely, suppose $\Xi$ satisfies the stated conditions, and let $\Psi$ be the corresponding operator under \eqref{Equation: charaterization of sections of mirror brane}. By the assumptions on $\Xi$ and the observation above, $\Psi$ commutes with multiplication by the same class of functions. Since these functions generate a dense subalgebra of $\mathcal{C}^\infty(U \times \check{\mathbb{T}}^{2n})$, continuity implies that $\Psi$ commutes with multiplication by all smooth functions. Hence $\Psi$ is $\mathcal{C}^\infty(U \times \check{\mathbb{T}}^{2n})$-linear, so $\Psi \in \Gamma(U \times \check{\mathbb{T}}^{2n}, \operatorname{End}(\mathbf{\check{E}}))$ and $\Xi = \langle \Psi \rangle$.
	\end{proof}
	
	\subsubsection{Family twisted Toeplitz operators}
	\quad\par
	\label{Subsubsection: family twisted Toeplitz operators on local branes}
	
	We equip both $\mathcal{C}^\infty(U \times \mathbb{T}^{2n})$ and $\Gamma(U \times \check{\mathbb{T}}^{2n}, \mathbf{\check{E}})$ with the compact-open $\mathcal{C}^\infty$-topology. On $\mathcal{C}^\infty(U \times \mathbb{T}^{2n})$, this topology may be described in terms of the fiberwise Fourier transform. It is induced by the family of seminorms defined as follows: for every compact subset $K \subset U$ and multi-indices $\mu \in \mathbb{N}^{2n}$ and $\nu, \eta \in \mathbb{N}^n$,
	\begin{equation*}
		\lVert f \rVert_{K, \mu, \nu, \eta} := \sup_{(x, \boldsymbol{\check{y}}_1, \boldsymbol{m}_1) \in K \times \check{\mathbb{R}}^n \times \mathbb{Z}^n} \left\lvert (\boldsymbol{m}_1)^\eta \cdot \frac{\partial^{\lvert \mu \rvert + \lvert \nu \rvert} \check{f}_{\boldsymbol{m}_1}}{\partial x^\mu \partial (\boldsymbol{\check{y}}_1)^\nu} (x, \boldsymbol{\check{y}}_1) \right\rvert.
	\end{equation*}
	On the other hand, recall that the standard Fr\'echet topology on $\mathcal{C}^\infty(U, \mathcal{S}(\check{\mathbb{R}}^n))$ is induced by the seminorms: for every compact subset $K \subset U$ and multi-indices $\mu \in \mathbb{N}^{2n}$ and $\nu, \eta \in \mathbb{N}^n$,
	\begin{equation*}
		\lVert s \rVert_{K, \mu, \nu, \eta} := \sup_{(x, \boldsymbol{\check{y}}_1) \in K \times \check{\mathbb{R}}^n} \left\lvert (\boldsymbol{\check{y}}_1)^\eta \cdot \frac{\partial^{\lvert \mu \rvert + \lvert \nu \rvert} s}{\partial x^\mu \partial (\boldsymbol{\check{y}}_1)^\nu} (x, \boldsymbol{\check{y}}_1) \right\rvert.
	\end{equation*}
	
	\begin{proposition}
		\label{Proposition: absolute convergence}
		Let $f \in \mathcal{C}^\infty(U \times \mathbb{T}^{2n})$ and $s \in \Gamma(U \times \check{\mathbb{T}}^{2n}, \mathbf{\check{E}})$. Then the series
		\begin{equation}
			\label{Equation: absolutely convergent series}
			\sum_{m \in \mathbb{Z}^{2n}}  \check{f}_{\boldsymbol{m}_1} \cdot \operatorname{S}_{\boldsymbol{m}_1}(\langle s \rangle)
		\end{equation}
		converges absolutely in $\mathcal{C}^\infty(U, \mathcal{S}(\check{\mathbb{R}}^n))$, where $\check{f}_{\boldsymbol{m}_1}$ are defined in \eqref{Equation: half-fibrewise Fourier coefficient}. Consequently, there is a well defined $\mathcal{C}^\infty(U)$-linear map
		\begin{equation}
			\label{Equation: quantization map}
			\Phi: \mathcal{C}^\infty(U \times \mathbb{T}^{2n}) \to \Gamma(U \times \check{\mathbb{T}}^{2n}, \operatorname{End}(\mathbf{\check{E}})), \quad f \mapsto \Phi_f,
		\end{equation}
		where for all $s \in \Gamma(U \times \check{\mathbb{T}}^{2n}, \mathbf{\check{E}})$, $\langle \Phi_f (s) \rangle$ is given by \eqref{Equation: absolutely convergent series}, and the map
		\begin{equation}
			\label{Equation: continuity of mirror map}
			\mathcal{C}^\infty(U \times \mathbb{T}^{2n}) \times \Gamma(U \times \check{\mathbb{T}}^{2n}, \mathbf{\check{E}}) \to \Gamma(U \times \check{\mathbb{T}}^{2n}, \mathbf{\check{E}}), \quad (f, s) \mapsto \Phi_f s,
		\end{equation}
		is continuous.
	\end{proposition}
	
	\begin{definition}
		\label{Definition: family twisted Toeplitz operator}
		Let $f \in \mathcal{C}^\infty(U \times \mathbb{T}^{2n})$. The \emph{family twisted Toeplitz operator associated with} $f$ is the $\mathcal{C}^\infty(U \times \check{\mathbb{T}}^{2n})$-linear map
		\begin{equation*}
			\Phi_f: \Gamma(U \times \check{\mathbb{T}}^{2n}, \mathbf{\check{E}}) \to \Gamma(U \times \check{\mathbb{T}}^{2n}, \mathbf{\check{E}})
		\end{equation*}
		defined in \eqref{Equation: quantization map}.
	\end{definition}
	
	\begin{proof}[\myproof{Proposition}{\ref{Proposition: absolute convergence}}]
		Fix a compact subset $K \subset U$ and multi-indices $\mu \in \mathbb{N}^{2n}$ and $\nu, \eta \in \mathbb{N}^n$. For each $q \in \mathbb{Z}^n$, define a multi-index $\boldsymbol{1}_q \in \mathbb{N}^n$ by $\boldsymbol{1}_q(i) = 1$ for $q_i \neq 0$; $\boldsymbol{1}_q(i) = 0$ otherwise.\par
		Let $C_f$ be the maximum of all seminorms $\lVert f \rVert_{K, \mu', \nu', \eta'}$ with $\mu' \leq \mu$, $\nu' \leq \nu$ and $\lvert \eta' \rvert \leq \lvert \eta \rvert + 2n$. Let $C_s$ be the maximum of all seminorms $\lVert \langle s \rangle \rVert_{K, \mu', \nu', \eta'}$ with $\mu' \leq \mu$, $\nu' \leq \nu$ and $\eta' \leq \eta$. Then
		\begin{equation*}
			\sup_{(x, \boldsymbol{\check{y}}_1) \in K \times \mathbb{R}^n} \left\lvert (\boldsymbol{m}_1)^{\eta'} \cdot \frac{\partial^{\lvert \mu' \rvert + \lvert \nu' \rvert} \check{f}_{\boldsymbol{m}_1}}{\partial x^{\mu'} \partial (\boldsymbol{\check{y}}_1)^{\nu'}} (x, \boldsymbol{\check{y}}_1) \right\rvert \leq \frac{1}{\lvert (\boldsymbol{m}_1)^{2\boldsymbol{1}_{\boldsymbol{m}_1}} \rvert} \cdot C_f
		\end{equation*}
		for all $\boldsymbol{m}_1 \in \mathbb{Z}^n$ and all multi-indices $\mu' \in \mathbb{N}^{2n}$ and $\nu', \eta' \in \mathbb{N}^n$ with $\mu' \leq \mu$, $\nu' \leq \nu$ and $\eta' \leq \eta$. Then by Leibniz rule and Multinomial Theorem, we can show that
		\begin{equation*}
			\lVert \check{f}_{\boldsymbol{m}_1} \cdot \operatorname{S}_{\boldsymbol{m}_1}(\langle s \rangle) \rVert_{K, \mu, \nu, \eta} \leq \frac{2^{\lvert \mu \rvert + \lvert \nu \rvert + \lvert \eta \rvert}}{\lvert (\boldsymbol{m}_1)^{2\boldsymbol{1}_{\boldsymbol{m}_1}} \rvert} \cdot C_f \cdot C_s.
		\end{equation*}
		Then the fact that
		\begin{equation*}
			\sum_{\boldsymbol{m}_1 \in \mathbb{Z}^n} \frac{2^{\lvert \mu \rvert + \lvert \nu \rvert + \lvert \eta \rvert}}{\lvert (\boldsymbol{m}_1)^{2\boldsymbol{1}_{\boldsymbol{m}_1}} \rvert} < \infty,
		\end{equation*}
		implies that
		\begin{equation*}
			\sum_{\boldsymbol{m}_1 \in \mathbb{Z}^n} \left\lVert \check{f}_{\boldsymbol{m}_1} \cdot \operatorname{S}_{\boldsymbol{m}_1}(\langle s \rangle) \right\rVert_{K, \mu, \nu, \eta} < \infty.
		\end{equation*}
		Thus, \eqref{Equation: absolutely convergent series} converges absolutely (and hence unconditionally) in the Fr\'echet space $\mathcal{C}^\infty(U, \mathcal{S}(\check{\mathbb{R}}^n))$. We can then well define an element $\Phi_f(s) \in \Gamma(U \times \check{\mathbb{T}}^{2n}, \mathbf{\check{E}})$ by setting $\langle \Phi_f(s) \rangle$ equal to \eqref{Equation: absolutely convergent series}. The above estimate also proves the continuity of \eqref{Equation: continuity of mirror map}.\par
		Now, fix $f \in \mathcal{C}^\infty(U \times \mathbb{T}^{2n})$ again. We show that $\Phi_f$ defines a section of $\operatorname{End}(\mathbf{\check{E}})$ by checking the criterion in Lemma \ref{Lemma: characterization of sections of endo-bundle}. Clearly, the map $\langle \Phi_f \rangle: \langle s \rangle \mapsto \langle \Phi_f(s) \rangle$ is $\mathbb{C}$-linear. The continuity condition in Lemma \ref{Lemma: characterization of sections of endo-bundle} is already verified. Next, fix $s \in \mathcal{C}^\infty(U, \mathcal{S}(\check{\mathbb{R}}^n))$, $g \in \mathcal{C}^\infty(U \times \check{\mathbb{T}}^n)$ and $\check{q} \in \mathbb{Z}^n$. We have
		\begin{equation*}
			\langle \Phi_f \rangle(g \cdot s) = \sum_{\boldsymbol{m}_1 \in \mathbb{Z}^n} \check{f}_{\boldsymbol{m}_1} \cdot \operatorname{S}_{\boldsymbol{m}_1}(g \cdot s) = \sum_{\boldsymbol{m}_1 \in \mathbb{Z}^n} \check{f}_{\boldsymbol{m}_1} \cdot \operatorname{S}_{\boldsymbol{m}_1}(g) \cdot \operatorname{S}_{\boldsymbol{m}_1}(s) = g \cdot \langle \Phi_f \rangle(s),
		\end{equation*}
		since $\operatorname{S}_{\boldsymbol{m}_1}(g) = g$. On the other hand,
		\begin{align*}
			(\langle \Phi_f \rangle \circ \operatorname{S}_{k\boldsymbol{H}\check{q}})(s) = & \sum_{\boldsymbol{m}_1 \in \mathbb{Z}^n} \check{f}_{\boldsymbol{m}_1} \cdot (\operatorname{S}_{\boldsymbol{m}_1} \circ \operatorname{S}_{k\boldsymbol{H}\check{q}})(s)\\
			= & \sum_{\boldsymbol{m}_1 \in \mathbb{Z}^n} \operatorname{S}_{k\boldsymbol{H}\check{q}}(\check{f}_{\boldsymbol{m}_1} \cdot \operatorname{S}_{\boldsymbol{m}_1}(s)) = (\operatorname{S}_{k\boldsymbol{H}\check{q}} \circ \langle \Phi_f \rangle)(s),
		\end{align*}
		since $\operatorname{S}_{k\boldsymbol{H}\check{q}}(\check{f}_{\boldsymbol{m}_1}) = \check{f}_{\boldsymbol{m}_1}$. To conclude, by Lemma \ref{Lemma: characterization of sections of endo-bundle} and injectivity of the map $\Psi \to \langle \Psi \rangle$ , $\Phi_f \in \Gamma(U \times \check{\mathbb{T}}^{2n}, \operatorname{End}(\mathbf{\check{E}}))$.
		The proof is complete.
	\end{proof}
	
	Extend the map $\Phi$ in \eqref{Equation: quantization map} to a map $\Omega^{0, *}(U \times \mathbb{T}^{2n}) \to \Omega^{0, *}(U \times \check{\mathbb{T}}^{2n}, \operatorname{End} (\mathbf{\check{E}}))$ by setting
	\begin{equation*}
		\Phi(d\overline{z}^{i_1} \wedge \cdots \wedge d\overline{z}^{i_l}) = d\overline{\check{z}}_{i_1} \wedge \cdots \wedge d\overline{\check{z}}_{i_l} \otimes \operatorname{Id}_{\mathbf{\check{E}}}.
	\end{equation*}
	Here, $\check{z}_i = x_i + \tfrac{\sqrt{-1}}{k} \check{y}_i$ where we have suppressed the superscript $(k)$ in the coordinates $\check{z}$.
	
	\subsection{Global construction of the mirror transform}
	\quad\par
	\label{Subsection: global construction of mirror map}
	Fix $k \in \mathbb{Z}_{>0}$ and work in the setting of Subsection \ref{Subsection: Global construction of mirror brane via gluing}. For each index $\alpha$ and $\Omega \in \mathbb{H}_n$, we have constructed the family twisted Toeplitz operator
	\begin{equation*}
		\Phi_{\alpha, \Omega}: \Omega^{0, *}(U_\alpha \times\mathbb{T}^{2n}) \to \Omega^{0, *}\left(U_\alpha \times \check{\mathbb{T}}^{2n}, \operatorname{End}(\mathbf{\check{E}}_{\alpha, \Omega}) \right).
	\end{equation*}
	On the other hand, we extend the family BKS pairing map $\operatorname{BKS}_{\Omega, \Omega'}^\alpha$ (Definition \ref{Definition: family BKS pairing map}) to a linear map
	\begin{equation}
		\Omega^{0, *}\left(U_\alpha \times \check{\mathbb{T}}^{2n}, \mathbf{\check{E}}_{\alpha, \Omega'}\right) \to  \Omega^{0, *}\left(U_\alpha \times \check{\mathbb{T}}^{2n}, \mathbf{\check{E}}_{\alpha, \Omega}\right)
	\end{equation}
	by its action on the $(0, *)$-forms of the form $d\overline{z}_\alpha^{i_1} \wedge \cdots \wedge d\overline{z}_\alpha^{i_l}$ as the identity map.
	
	\begin{proposition}
		\label{Proposition: Commutativity one}
		For all index $\alpha$ and $\Omega, \Omega' \in \mathbb{H}_n$, the following diagram commutes:
		\begin{center}
			\begin{tikzcd}
				\Omega^{0, *}(U_\alpha \times\mathbb{T}^{2n}) \ar[rr, "\Phi_{\alpha, \Omega'}"] \ar[d, "="'] && \Omega^{0, *}\left(U_\alpha \times \check{\mathbb{T}}^{2n}, \operatorname{End}(\mathbf{\check{E}}_{\alpha, \Omega'})\right) \ar[d, "\text{conjugation by } \operatorname{BKS}_{\Omega, \Omega'}^\alpha"]\\
				\Omega^{0, *}(U_\alpha \times\mathbb{T}^{2n}) \ar[rr, "\Phi_{\alpha, \Omega}"'] && \Omega^{0, *}\left(U_\alpha \times \check{\mathbb{T}}^{2n}, \operatorname{End}(\mathbf{\check{E}}_{\alpha, \Omega})\right)
			\end{tikzcd}
		\end{center}
	\end{proposition}
	\begin{proof}
		Throughout the proof, we suppress the index $\alpha$. By construction of the family twisted Toeplitz operators and the family BKS pairing maps, the statement is immediate on the differential form component. It therefore suffices to verify the claim on the function component.\par
		By continuity and $\mathcal{C}^\infty(U)$-linearity, it is enough to check that
		\begin{equation}
			\Phi_{\Omega, f} \circ \operatorname{BKS}_{\Omega, \Omega'} = \operatorname{BKS}_{\Omega, \Omega'} \circ \Phi_{\Omega', f}
		\end{equation}
		for functions of the form
		\begin{equation*}
			f(x, y) = e^{2\pi\sqrt{-1} m \cdot u} \quad \text{with } m \in \mathbb{Z}^{2n}.
		\end{equation*}
		It suffices to compare both sides on an orthonormal basis of each fiber. Fix $(x, \check{y})$ and consider the bases $\{ \sigma_l(x, \,\cdot\,, \check{y}) \}_{l \in \mathbb{Z}_{k\boldsymbol{H}}^n}$ and $\{ \sigma'_l(x, \,\cdot\,, \check{y}) \}_{l \in \mathbb{Z}_{k\boldsymbol{H}}^n}$ corresponding to $\Omega$ and $\Omega'$, respectively, as given in Proposition \ref{Proposition: orthonormal basis of geometric quantization}. For simplicity, write $\sigma
		_l = \sigma_l(x, \,\cdot\,, \check{y})$ and $\sigma'
		_l = \sigma'_l(x, \,\cdot\,, \check{y})$.\par
		Fix $l \in \mathbb{Z}_{k\boldsymbol{H}}^n$. Let $\widetilde{l} \in \mathbb{Z}_{k\boldsymbol{H}}^n$ be such that $\widetilde{q} := (k\boldsymbol{H})^{-1} (l + \boldsymbol{m}_1 - \widetilde{l}) \in \mathbb{Z}^n$. Let $\Delta$ and $\Delta'$ denote the differential operators defined in \eqref{Equation: fibrewise Dolbeault Laplacian} corresponding to $\Omega$ and $\Omega'$, respectively, and set
		\begin{equation*}
			f_\Omega = \exp\left(\tfrac{1}{k} \Delta \right) f \quad \text{and} \quad f_{\Omega'} = \exp\left(\tfrac{1}{k} \Delta' \right) f.
		\end{equation*}
		By the definitions of $\Phi_{\Omega}, \Phi_{\Omega'}$ and $\operatorname{BKS}_{\Omega, \Omega'}$, together with Lemma \ref{Lemma: Toeplitz operator}, we compute
		\begin{align*}
			({\Phi_{\Omega, f}} \circ \operatorname{BKS}_{\Omega, \Omega'}) (\sigma'_l) = & \Phi_{\Omega, f} ( \sigma_l ) = \langle f_\Omega \sigma_l, \sigma_{\widetilde{l}} \rangle \sigma_{\widetilde{l}},\\
			(\operatorname{BKS}_{\Omega, \Omega'} \circ \Phi_{\Omega', f}) (\sigma'_l) = & \operatorname{BKS}_{\Omega, \Omega'} \left( \langle f_{\Omega'} \sigma'_l, \sigma'_{\widetilde{l}} \rangle \sigma'_{\widetilde{l}} \right) = \langle f_{\Omega'} \sigma'_l, \sigma'_{\widetilde{l}} \rangle \sigma_{\widetilde{l}}.
		\end{align*}
		The explicit formula in Lemma \ref{Lemma: Toeplitz operator} shows that the coefficients $\langle f_\Omega \sigma_l, \sigma_{\widetilde{l}} \rangle$ and $\langle f_{\Omega'} \sigma'_l, \sigma'_{l'} \rangle$ coincide. Hence the two compositions agree on a basis, and the diagram commutes.
	\end{proof}
	
	In what follows, for indices $\alpha, \beta$ with $U_{\alpha\beta} \neq \emptyset$ and $\Omega \in \mathbb{H}_n$, we denote by
	\begin{equation*}
		\operatorname{Ad}(\phi_{\alpha\beta; \Omega}): \Gamma(U_{\alpha\beta} \times \check{\mathbb{T}}^{2n}, \operatorname{End}(\mathbf{\check{E}}_{\beta, \Omega})) \to \Gamma(U_{\alpha\beta} \times \check{\mathbb{T}}^{2n}, \operatorname{End}(\mathbf{\check{E}}_{\alpha, A_{\alpha\beta} \cdot \Omega}))
	\end{equation*}
	the conjugation by $\phi_{\alpha\beta; \Omega}$. We extend this map to
	\begin{equation*}
		\operatorname{Ad}(\phi_{\alpha\beta; \Omega}): \Omega^{0, *}(U_{\alpha\beta} \times \check{\mathbb{T}}^{2n}, \operatorname{End}(\mathbf{\check{E}}_{\beta, \Omega})) \to \Omega^{0, *}(U_{\alpha\beta} \times \check{\mathbb{T}}^{2n}, \operatorname{End}(\mathbf{\check{E}}_{\alpha, A_{\alpha\beta} \cdot \Omega}))
	\end{equation*}
	by combining conjugation on the endomorphism part with the tensorial transformation of
	\begin{equation*}
		d\overline{\check{z}}_{\beta, i_1} \wedge \cdots \wedge d\overline{\check{z}}_{\beta, i_l}
	\end{equation*}
	under the coordinate change $(x_\beta, y_\beta) \mapsto (x_\alpha, y_\alpha)$.
	
	\begin{proposition}
		\label{Proposition: Commutativity two}
		For all indices $\alpha, \beta$ with $U_{\alpha\beta} \neq \emptyset$ and $\Omega \in \mathbb{H}_n$, the following diagram commutes:
		\begin{center}
			\begin{tikzcd}
				\Omega^{0, *}(U_{\alpha\beta} \times \mathbb{T}^{2n}) \ar[rr, "\Phi_{\beta, \Omega}"] \ar[d] && \Omega^{0, *}(U_{\alpha\beta} \times \check{\mathbb{T}}^{2n}, \operatorname{End}(\mathbf{\check{E}}_{\beta, \Omega})) \ar[d, "\operatorname{Ad}(\phi_{\alpha\beta; \Omega})"]\\
				\Omega^{0, *}(U_{\alpha\beta} \times \mathbb{T}^{2n}) \ar[rr, "\Phi_{\alpha, A_{\alpha\beta} \cdot \Omega}"'] && \Omega^{0, *}(U_{\alpha\beta} \times \check{\mathbb{T}}^{2n}, \operatorname{End}(\mathbf{\check{E}}_{\alpha, A_{\alpha\beta} \cdot \Omega}))
			\end{tikzcd}
		\end{center}
		where the left vertical arrow is the transform under the coordinate change $(x_\beta, y_\beta) \mapsto (x_\alpha, y_\alpha)$.
	\end{proposition}
	\begin{proof}
		We first observe that the Dolbeault form components transform tensorially under the coordinate change associated with the constant transition matrix $A_{\alpha\beta} \in \operatorname{GL}(2n, \mathbb{Z})$:
		\begin{equation*}
			d\overline{z}_\beta = A_{\alpha\beta} d\overline{z}_\alpha \quad \text{and} \quad d\overline{\check{z}}_\beta = A_{\alpha\beta} d\overline{\check{z}}_\alpha.
		\end{equation*}
		It follows that both vertical arrows act on the $(0, *)$-form part in the same way. Therefore, it suffices to verify the commutativity of the diagram on the level of functions.\par
		Let $f \in \mathcal{C}^\infty(U_{\alpha\beta} \times \mathbb{T}^{2n})$, and denote by $f'$ its transform under the coordinate change. By $\mathcal{C}^\infty(U_{\alpha\beta})$-linearity and continuity, we may assume that $f$ is a single Fourier mode:
		\begin{equation*}
			f(x_\beta, y_\beta) = e^{2\pi\sqrt{-1} m_\beta \cdot u_\beta}, \quad m_\beta \in \mathbb{Z}^{2n}.
		\end{equation*}
		Then $f'$ is again a single Fourier mode in the $(x_\alpha, y_\alpha)$-coordinates. By construction of the family Toeplitz operators, together with the compatibility of the transition isomorphisms $\phi_{\alpha\beta; \Omega}$ with the Hermitian structures, connections, and polarizations, we have
		\begin{equation*}
			T_{\alpha, \Omega', f'} \circ \phi_{\alpha\beta; \Omega} = \phi_{\alpha\beta; \Omega} \circ T_{\beta, \Omega, f},
		\end{equation*}
		where $\Omega' := A_{\alpha\beta} \cdot \Omega$. Similar to the proof of Proposition \ref{Proposition: commutativity of transitions}, this follows from the naturality of the Toeplitz construction, since $\phi_{\alpha\beta; \Omega}$ intertwines all geometric data entering the fiberwise definition.\par
		Consider the operators $\Delta$ and $\Delta'$ defined in \eqref{Equation: fibrewise Dolbeault Laplacian} corresponding to $\Omega$ and $\Omega'$, respectively. The operator $\Delta$ transforms naturally under the coordinate change, so that $\Delta'$ is its pushforward. Consequently, $\widetilde{f}' := \exp (\tfrac{1}{k} \Delta') f'$ is the transform of $\widetilde{f} := \exp (\tfrac{1}{k} \Delta) f$. In particular, both $\widetilde{f}$ and $\widetilde{f}'$ are scalar multiples of $f$ and $f'$ by the same nonzero constant. By definition,
		\begin{equation*}
			\Phi_{\beta, \Omega, f} = T_{\beta, \Omega, \widetilde{f}} \quad \text{and} \quad \Phi_{\alpha, \Omega', f'} = T_{\beta, \Omega, \widetilde{f}'}.
		\end{equation*}
		Combining this with the intertwining relation above, we obtain
		\begin{equation*}
			\Phi_{\alpha, \Omega', f'} \circ \phi_{\alpha\beta; \Omega} = \phi_{\alpha\beta; \Omega} \circ \Phi_{\beta, \Omega, f}.
		\end{equation*}
		Finally, extending by $\mathcal{C}^\infty(U_{\alpha\beta})$-linearity and continuity to arbitrary functions $f$, we conclude that the diagram commutes.
	\end{proof}
	
	Combining Propositions \ref{Proposition: Commutativity one} and \ref{Proposition: Commutativity two}, we deduce that the maps $\Phi_{\alpha, \Omega}$ are compatible on overlaps and therefore glue to a family of maps
	\begin{equation*}
		\Phi_U: \Omega^{0, *}(X_U) \to \Omega^{0, *}(\check{X}_U, \operatorname{End}(\check{E})),
	\end{equation*}
	for each open subset $U \subset B$, defining a morphism of sheaves of $\mathcal{C}_B^\infty$-modules.
	
	\subsection{Proof of the mirror theorem}
	\quad\par
	\label{Subsection: proof of mirror statement}
	Finally, we prove our main result, Theorem \ref{Theorem: mirror statement}. Since the statement is sheaf-theoretic over the base $B$, it suffices to carry out the argument on an open subset $U \subset B$ of the type considered in Proposition \ref{Proposition: skew smith form of A brane}. We also fix, once and for all, an auxiliary choice of $\Omega \in \mathbb{H}_n$. The proof is organized into a sequence of lemmas.
	
	\begin{lemma}\label{Lemma: holomorphicity}
		The map $\Phi$ is a homomorphism of differential graded algebras:
		\begin{equation*}
			\Phi: (\Omega^{0, *}(U \times \mathbb{T}^{2n}), \star_{k^{-1}}, \overline{\partial}) \to ( \Omega^{0, *}(U \times \check{\mathbb{T}}^{2n}, \operatorname{End}(\mathbf{\check{E}})), \circ, \overline{\partial}).
		\end{equation*}
	\end{lemma}
	\begin{proof}
		By construction, $\Phi$ preserves degree. We first verify that it is an algebra homomorphism. It is immediate that $\Phi$ sends the constant one function to the identity map. Thus, it suffices to prove that
		\begin{equation}
			\label{Equation: algebra_hom}
			\Phi_f \circ \Phi_g = \Phi_{f \star_{k^{-1}} g}
		\end{equation}
		for all $f, g \in \mathcal{C}^\infty(U \times \mathbb{T}^{2n})$. Indeed, once \eqref{Equation: algebra_hom} holds for functions, the general case for forms $\alpha, \beta \in \Omega^{0, *}(U \times \mathbb{T}^{2n})$ follows from the formula $\Phi(d\overline{z}^i) = d\overline{\check{z}}_i$.
		By continuity and $\mathcal{C}^\infty(U)$-(bi)linearity, it is enough to verify \eqref{Equation: algebra_hom} on Fourier modes of the form
		\begin{equation*}
			f(x, y) = e^{2\pi\sqrt{-1} m \cdot u}, \quad g(x, y) = e^{2\pi\sqrt{-1} m' \cdot u},
		\end{equation*}
		with $m, m' \in \mathbb{Z}^{2n}$. In this case, \eqref{Equation: algebra_hom} follows by a direct computation using \eqref{Equation: family Toeplitz operator} ((see also \cite{And2005}, Theorem 5).\par
		It remains to to show that $\Phi$ intertwines the Dolbeault operators. For similar reasons, it suffices to check that for each $1 \leq i \leq 2n$,
		\begin{equation}
			\Phi_{\frac{\partial f}{\partial z^i}} = \check{\nabla}_{\partial_{\overline{\check{z}_i}}} \circ \Phi_f - \Phi_f \circ \check{\nabla}_{\partial_{\overline{\check{z}_i}}}
		\end{equation}
		for functions of the form $f(x, y) = f_m(x) e^{2\pi\sqrt{-1} m \cdot u}$ with $m \in \mathbb{Z}^{2n}$ and a function $f_m \in \mathcal{C}^\infty(U)$. Fix $s \in \Gamma(U \times \check{\mathbb{T}}^{2n}, \mathbf{\check{E}})$. By Lemma \ref{Lemma: family Toeplitz operator},
		\begin{equation*}
			\left\langle \Phi_f s \right\rangle (x, \boldsymbol{\check{y}}_1) = F_m(x, \boldsymbol{\check{y}}_1) \cdot \langle s \rangle (x, \boldsymbol{\check{y}}_1 + \boldsymbol{m}_1),
		\end{equation*}
		where
		\begin{equation*}
			F_m(x, \boldsymbol{\check{y}}_1) := f_m(x) \exp \left( -\tfrac{2\pi\sqrt{-1}}{k} \boldsymbol{m}_2 \cdot \boldsymbol{H}^{-1} \left(\boldsymbol{\check{y}}_1 + \tfrac{1}{2} \boldsymbol{m}_1 \right) \right).
		\end{equation*}
		Using the formulas established in the proof of Theorem \ref{Theorem: unitary connection on the local mirror B brane}, we compute for $1 \leq i \leq n$:
		\begin{equation*}
			\left\langle \check{\nabla}_{\partial_{\overline{\check{z}_i}}} \Phi_f s - \Phi_f \check{\nabla}_{\partial_{\overline{\check{z}_i}}} s \right\rangle(x, \boldsymbol{\check{y}}_1) = \frac{\partial F_m}{\partial \overline{\check{z}_i}} (x, \boldsymbol{\check{y}}_1) \cdot \langle s \rangle (x, \boldsymbol{\check{y}}_1 + \boldsymbol{m}_1),
		\end{equation*}
		and
		\begin{align*}
			& \left\langle \check{\nabla}_{\partial_{\overline{\check{z}_{n+i}}}} \Phi_f s - \Phi_f \check{\nabla}_{\partial_{\overline{\check{z}_{n+i}}}} s \right\rangle (x, \boldsymbol{\check{y}}_1)\\
			= & \left( \frac{1}{2} \frac{\partial F_m}{\partial x_{n+i}} (x, \boldsymbol{\check{y}}_1) - \frac{k\pi\sqrt{-1}}{h_i} \left( - \sqrt{-1} \frac{m_i}{k} \right) \cdot F_m(x, \boldsymbol{\check{y}}_1) \right) \cdot \langle s \rangle (x, \boldsymbol{\check{y}}_1 + \boldsymbol{m}_1)\\
			= & \left( \frac{1}{2} \frac{\partial F_m}{\partial x_{n+i}} (x, \boldsymbol{\check{y}}_1) - \frac{\pi m_i}{h_i} \cdot F_m(x, \boldsymbol{\check{y}}_1) \right) \cdot \langle s \rangle (x, \boldsymbol{\check{y}}_1 + \boldsymbol{m}_1).
		\end{align*}
		A direct computation using \eqref{Equation: first complex coordinate derivative} and \eqref{Equation: second complex coordinate derivative} shows that for $j = i$ and $j = n + i$,
		\begin{equation*}
			\left\langle \Phi_{\partial_{\overline{z}^j} f} s \right\rangle(x, \boldsymbol{\check{y}}_1) = \left\langle \check{\nabla}_{\partial_{\overline{\check{z}_j}}} \Phi_f s - \Phi_f \check{\nabla}_{\partial_{\overline{\check{z}_j}}} s \right\rangle(x, \boldsymbol{\check{y}}_1).
		\end{equation*}
		Therefore, $\Phi$ is a morphism of differential graded algebras.
	\end{proof}
	
	\begin{lemma}
		\label{Lemma: injectivity}
		The map $\Phi: \Omega^{0, *}(U \times \mathbb{T}^{2n}) \to \Omega^{0, *}(U \times \check{\mathbb{T}}^{2n}, \operatorname{End}(\mathbf{\check{E}}))$ is injective.
	\end{lemma}
	\begin{proof}
		It suffices to prove injectivity on functions. Let $f \in \mathcal{C}^\infty(U \times \mathbb{T}^{2n})$ and assume $\Phi_f = 0$. We show that $f = 0$. By \eqref{Equation: half-fibrewise Fourier expansion}, it is enough to prove that $\check{f}_{\boldsymbol{m}_1} = 0$ for all $\boldsymbol{m}_1 \in \mathbb{Z}^n$, where $\check{f}_{\boldsymbol{m}_1}$ is defined in \eqref{Equation: half-fibrewise Fourier coefficient}.\par
		Fix $\boldsymbol{m}_1 \in \mathbb{Z}^n$ and $(x, \boldsymbol{\check{y}}_1) \in U \times \check{\mathbb{R}}^n$. Choose a section 
		$s \in \Gamma(U \times \check{\mathbb{T}}^{2n}, \mathbf{\check{E}})$ 
		such that its associated function $\langle s \rangle \in \mathcal{C}^\infty(U, \mathcal{S}(\check{\mathbb{R}}^n))$ satisfies $\langle s \rangle(x,\boldsymbol{\check{y}}_1 + \boldsymbol{m}_1) = 1$ and $\langle s \rangle(x,\boldsymbol{\check{y}}_1 + \check{q}) = 0$ for all $\check{q} \in \mathbb{Z}^n$ with $\check{q} \neq \boldsymbol{m}_1$. Then Proposition \ref{Proposition: absolute convergence} implies
		\begin{equation*}
			\check{f}_{\boldsymbol{m}_1}(x, \boldsymbol{\check{y}}_1) = \langle \Phi_f(s) \rangle (x, \boldsymbol{\check{y}}_1) = 0.
		\end{equation*}
		Since $(x, \boldsymbol{\check{y}}_1)$ and $\boldsymbol{m}_1$ were arbitrary, \eqref{Equation: half-fibrewise Fourier expansion} yields $f=0$.
	\end{proof}
	
	\begin{lemma}
		\label{Lemma: surjectivity}
		The map $\Phi: \Omega^{0, *}(U \times \mathbb{T}^{2n}) \to \Omega^{0, *}(U \times \check{\mathbb{T}}^{2n}, \operatorname{End}(\mathbf{\check{E}}))$ is surjective.
	\end{lemma}
	\begin{proof}
		It suffices to prove surjectivity on functions. Fix $\Psi \in \Gamma(U \times \check{\mathbb{T}}^{2n}, \operatorname{End}(\mathbf{\check{E}}))$. For each $\boldsymbol{m}_1 \in \mathbb{Z}^n$, define a function $\check{f}_{\boldsymbol{m}_1}: U \times \check{\mathbb{R}}^n \to \mathbb{C}$ by
		\begin{equation}
			\check{f}_{\boldsymbol{m}_1}(x, \boldsymbol{\check{y}}_1) := (\langle \Psi \rangle s) (x, \boldsymbol{\check{y}}_1),
		\end{equation}
		where $s \in \mathcal{C}^\infty(U, \mathcal{S}(\check{\mathbb{R}}^n))$ is any function satisfying $\langle s \rangle(x,\boldsymbol{\check{y}}_1 + \boldsymbol{m}_1) = 1$ and $\langle s \rangle(x,\boldsymbol{\check{y}}_1 + \check{q}) = 0$ for all $\check{q} \in \mathbb{Z}^n$ with $\check{q} \neq \boldsymbol{m}_1$. By Lemma \ref{Lemma: dependence on lattice}, $\check{f}_{\boldsymbol{m}_1}$ is well defined.\par
		We first show that $\check{f}_{\boldsymbol{m}_1}$ is smooth. Fix $(x, \boldsymbol{\check{y}}_1) \in U \times \check{\mathbb{R}}^n$. Choose a function $s \in \mathcal{C}^\infty(U, \mathcal{S}(\check{\mathbb{R}}^n))$ which equals $1$ on a sufficiently small neighbourhood of $(x, \boldsymbol{\check{y}}_1 + \boldsymbol{m}_1)$ and vanishes away from a slightly larger neighbourhood. Then on a sufficiently small neighbourhood of $(x, \boldsymbol{\check{y}}_1)$,
		\begin{equation*}
			\check{f}_{\boldsymbol{m}_1} = \langle \Psi \rangle (s).
		\end{equation*}
		Since the right-hand side is smooth at $(x, \boldsymbol{\check{y}}_1)$, so is $\check{f}_{\boldsymbol{m}_1}$.\par
		Next, $\check{f}_{\boldsymbol{m}_1}$ is invariant under the lattice $k \boldsymbol{H} \mathbb{Z}^n$. Indeed, fix $(x, \boldsymbol{\check{y}}_1) \in U \times \check{\mathbb{R}}^n$ and $\check{q} \in \mathbb{Z}^n$. Choosing $s \in \mathcal{C}^\infty(U, \mathcal{S}(\check{\mathbb{R}}^n))$ which equals $1$ at $(x, \boldsymbol{\check{y}}_1 + k\boldsymbol{H} \check{q} + \boldsymbol{m}_1)$ and vanishes away from a sufficiently small neighbourhood of $(x, \boldsymbol{\check{y}}_1 + k\boldsymbol{H} \check{q} + \boldsymbol{m}_1)$, we obtain
		\begin{equation*}
			\check{f}_{\boldsymbol{m}_1}(x, \boldsymbol{\check{y}}_1 + k\boldsymbol{H} \check{q}) = (\operatorname{S}_{k\boldsymbol{H} \check{q}} \circ \langle \Psi \rangle)(s)(x, \boldsymbol{\check{y}}_1) = (\langle \Psi \rangle \circ \operatorname{S}_{k\boldsymbol{H} \check{q}})(s)(x, \boldsymbol{\check{y}}_1) = \check{f}_{\boldsymbol{m}_1}(x, \boldsymbol{\check{y}}_1)
		\end{equation*}
		by Lemma \ref{Lemma: characterization of sections of endo-bundle}, where $\operatorname{S}_{k\boldsymbol{H} \check{q}}$ denotes the shift operator defined in \eqref{Equation: shift operator}. So $\check{f}_{\boldsymbol{m}_1}$ descends to a smooth function on $U \times \check{\mathbb{T}}_{k\boldsymbol{H}}^n$, denoted by the same symbol.\par
		Moreover, Lemma \ref{Lemma: Fibrewise Schwartz property}, together with the finiteness of $\mathbb{Z}_{k\boldsymbol{H}}^n$ and the $k\boldsymbol{H}\mathbb{Z}^n$-invariance of $\check{f}$, implies that for all compact subset $K \subset U$ and multi-indices $\mu \in \mathbb{N}^{2n}$ and $\nu, \eta \in \mathbb{N}^n$, $\lVert \check{f} \rVert_{K, \mu, \nu, \eta} < \infty$. Consequently, the Fourier series
		\begin{equation*}
			f(x, y) := \sum_{\boldsymbol{m}_1 \in \mathbb{Z}^n} \check{f}_{\boldsymbol{m}_1}(x, -k\boldsymbol{H}\boldsymbol{u}^2 - \tfrac{1}{2} \boldsymbol{m}_1) \cdot e^{2\pi\sqrt{-1} \boldsymbol{m}_1 \cdot \boldsymbol{u}^1}
		\end{equation*}
		defines a smooth function $f \in \mathcal{C}^\infty(U \times \mathbb{T}^{2n})$. By construction, the associated section $\Phi_f$ coincides with $\Psi$. Hence $\Phi$ is surjective.
	\end{proof}
	
	\begin{lemma}[Lattice dependence]
		\label{Lemma: dependence on lattice}
		Let $\Psi \in \Gamma(U \times \check{\mathbb{T}}^{2n}, \operatorname{End}(\mathbf{\check{E}}))$ and $(x, \boldsymbol{\check{y}}_1) \in U \times \check{\mathbb{R}}^n$. If $s \in \mathcal{C}^\infty(U, \mathcal{S}(\check{\mathbb{R}}^n))$ satisfies $s(x, \boldsymbol{\check{y}}_1 + \boldsymbol{m}_1) = 0$ for all $\boldsymbol{m}_1 \in \mathbb{Z}^n$, then $(\langle \Psi \rangle s)(x, \boldsymbol{\check{y}}_1) = 0$.
	\end{lemma}
	\begin{proof}
		Approximate $s$ in $\mathcal{C}^\infty(U, \mathcal{S}(\check{\mathbb{R}}^n))$ by compactly supported functions with the same vanishing property. By continuity, it suffices to treat the compactly supported case.\par
		In this case, write $s = \sum_{i=1}^r f_i \cdot s_i$, where each $s_i$ is compactly supported (and hence in $\mathcal{C}^\infty(U, \mathcal{S}(\check{\mathbb{R}}^n))$) and each $f_i \in \mathcal{C}^\infty(U \times \check{\mathbb{T}}^n)$ vanishes at $(x, \boldsymbol{\check{y}}_1)$. By Lemma \ref{Lemma: characterization of sections of endo-bundle},
		\begin{equation*}
			(\langle \Psi \rangle s)(x, \boldsymbol{\check{y}}_1) = \sum_{i=1}^r f_i(x, \boldsymbol{\check{y}}_1) \cdot (\langle \Psi \rangle s_i)(x, \boldsymbol{\check{y}}_1) = 0.
		\end{equation*}
	\end{proof}
	
	\begin{lemma}[Fiberwise Schwartz estimate]
		\label{Lemma: Fibrewise Schwartz property}
		Let $K \subset U$ be a compact subset, and $\mu \in \mathbb{N}^{2n}$ and $\nu, \eta \in \mathbb{N}^n$ be multi-indices. Let $\boldsymbol{\check{y}}_1 \in \check{\mathbb{R}}^n$ and $l \in \mathbb{Z}_{k\boldsymbol{H}}^n$. Then there exists $\varepsilon > 0$ such that
		\begin{equation}
			\label{Equation: seminorm estimate}
			\sup_{(x, \boldsymbol{\check{y}'}_1, \check{q}) \in K \times \overline{\operatorname{B}}(\boldsymbol{\check{y}}_1, \varepsilon) \times \mathbb{Z}^n} \left\lvert (l + k\boldsymbol{H} \check{q})^\eta \cdot \frac{\partial^{\lvert \mu \rvert + \lvert \nu \rvert} \check{f}_{l + k\boldsymbol{H} \check{q}}}{\partial x^\mu \partial (\boldsymbol{\check{y}'}_1)^\nu} (x, \boldsymbol{\check{y}'}_1) \right\rvert < \infty,
		\end{equation}
		where $\overline{\operatorname{B}}(\boldsymbol{\check{y}}_1, \varepsilon)$ denotes the closed ball in $\check{\mathbb{R}}^n$ of radius $\varepsilon$ centred at $\boldsymbol{\check{y}}_1$.
	\end{lemma}
	\begin{proof}
		Fix $\varepsilon = \tfrac{1}{4}$ and choose a fiberwise cutoff function $s \in \mathcal{C}^\infty(U, \mathcal{S}(\check{\mathbb{R}}^n))$ such that
		\begin{equation*}
			s \equiv 1 \quad \text{on } U \times \overline{\operatorname{B}}(\boldsymbol{\check{y}}_1 + l, \varepsilon) \quad \text{and} \quad \operatorname{supp} (s) \subset U \times \overline{\operatorname{B}}(\boldsymbol{\check{y}}_1 + l, 2\varepsilon).
		\end{equation*}
		Then for all $(x, \boldsymbol{\check{y}'}_1, \check{q}) \in K \times \overline{\operatorname{B}}(\boldsymbol{\check{y}}_1, \varepsilon) \times \mathbb{Z}^n$,
		\begin{equation*}
			(\langle \Psi \rangle s)(x, \boldsymbol{\check{y}'}_1 - k\boldsymbol{H}\check{q}) = \check{f}_{l + k\boldsymbol{H}\check{q}}(x, \boldsymbol{\check{y}'}_1 - k\boldsymbol{H}\check{q}) = \check{f}_{l + k\boldsymbol{H}\check{q}}(x, \boldsymbol{\check{y}'}_1).
		\end{equation*}
		Hence the seminorm bound $\lVert \langle \Psi \rangle s \rVert_{K, \mu, \nu, \eta} < \infty$ for $\langle \Psi \rangle s$ implies
		\begin{equation*}
			\sup_{(x, \boldsymbol{\check{y}'}_1, \check{q}) \in K \times \overline{\operatorname{B}}(\boldsymbol{\check{y}}_1, \varepsilon) \times \mathbb{Z}^n} \left\lvert (\boldsymbol{\check{y}'}_1 - k\boldsymbol{H} \check{q})^\eta \cdot \frac{\partial^{\lvert \mu \rvert + \lvert \nu \rvert} \check{f}_{l + k\boldsymbol{H} \check{q}}}{\partial x^\mu \partial (\boldsymbol{\check{y}'}_1)^\nu} (x, \boldsymbol{\check{y}'}_1) \right\rvert < \infty.
		\end{equation*}
		Since $\boldsymbol{\check{y}'}_1$ ranges in a bounded set, the weights $(\boldsymbol{\check{y}'}_1 - k\boldsymbol{H} \check{q})^\eta$ are comparable to $(l + k\boldsymbol{H} \check{q})^\eta$ outside a finite subset of $\mathbb{Z}^n$. The desired estimate follows by combining this comparison with compactness.
	\end{proof}
	
	\begin{proof}[\myproof{Theorem}{\ref{Theorem: mirror statement}}]
		By construction, the collection $\{ \Phi_U^{(k)} \}$ defines a morphism of sheaves of $\mathcal{C}_B^\infty$-modules. For any open subset $U \subset B$, Lemma \ref{Lemma: holomorphicity} ensures that $\Phi_U^{(k)}$ is a morphism of differential graded algebras. Since $\Phi_U^{(k)}$ is bijective (Lemmas \ref{Lemma: injectivity} and \ref{Lemma: surjectivity}), it constitutes an isomorphism of sheaves of DGAs. This isomorphism naturally induces the claimed isomorphism on the resulting cohomology groups, which completes the proof.
	\end{proof}

	\appendix
	
	\section{Function Spaces with Fiberwise Rapid Decay}
	\label{Sectoin: Remarks on fibrewise Schwartz spaces}
	Let $(x^1, ..., x^n)$ be the standard coordinates on $\mathbb{R}^n$, and let $U \subset \mathbb{R}^n$ be an open subset. Let $(y^1, ..., y^m)$ denote the standard coordinate on $\mathbb{R}^m$, and write $\mathcal{S}(\mathbb{R}^m)$ for the Schwartz space on $\mathbb{R}^m$. We define
	\begin{equation*}
		\mathcal{C}^\infty(U, \mathcal{S}(\mathbb{R}^m))
	\end{equation*}
	to be the space of smooth $\mathbb{C}$-valued functions $s \in \mathcal{C}^\infty(U \times \mathbb{R}^m)$ satisfying the following conditions: for all compact subset $K \subset U$ and multi-indices $\mu \in \mathbb{N}^n$ and $\nu, \eta \in \mathbb{N}^m$,
	\begin{equation*}
		\lVert s \rVert_{K, \mu, \nu, \eta} := \sup_{(x, y) \in K \times \mathbb{R}^m} \left\lvert y^\eta \cdot \frac{\partial^{\lvert \mu \rvert + \lvert \nu \rvert} s}{\partial x^\mu \partial y^\nu} (x, y) \right\rvert < \infty.
	\end{equation*}
	We refer to $\mathcal{C}^\infty(U, \mathcal{S}(\mathbb{R}^m))$ as the \emph{space of smooth $\mathbb{C}$-valued functions on $U \times \mathbb{R}^m$ that are rapidly decreasing in the $\mathbb{R}^m$-variables}. We equip this space with the locally convex topology generated by the seminorms $\lVert \,\cdot\, \rVert_{K, \mu, \nu, \eta}$.
	
	\begin{remark}
		This space can be identified with the completed projective tensor product $\mathcal{C}^\infty(U) \widehat{\otimes}_\pi \mathcal{S}(\mathbb{R}^m)$, where $\mathcal{C}^\infty(U)$ is equipped with the usual compact-open $\mathcal{C}^\infty$-topology and $\mathcal{S}(\mathbb{R}^m)$ is equipped with its Schwartz topology.
	\end{remark}

	\section{Twisted integral symplectic groups}
	\label{Section: Twisted integral symplectic groups}
	Let $n \in \mathbb{Z}$, $I_n$ denote the $n \times n$ identity matrix and
	\begin{equation}
		\label{Equation: symplectic involution}
		J_n := \begin{pmatrix}
			0 & -I_n\\
			I_n & 0
		\end{pmatrix}.
	\end{equation}
	Recall that the real symplectic group $\operatorname{Sp}(2n, \mathbb{R}) = \left\{ A \in \operatorname{GL}(2n, \mathbb{R}): A J_n A^T = J_n \right\}$ has a canonical action on the Segal upper half space $\mathbb{H}_n$, which is given by
	\begin{equation}
		\label{Equation: standard action of symplectic group on upper half space}
		A \cdot \Omega = 
		(\boldsymbol{A}_1^1\Omega + \boldsymbol{A}_1^2)(\boldsymbol{A}_2^1\Omega + \boldsymbol{A}_2^2)^{-1}, \quad \text{where }
		A =
		\begin{pmatrix}
			\boldsymbol{A}_1^1 & \boldsymbol{A}_1^2\\
			\boldsymbol{A}_2^1 & \boldsymbol{A}_2^2
		\end{pmatrix}.
	\end{equation}
	We are interested in a variant of the standard integral symplectic group $\operatorname{Sp}(2n, \mathbb{Z}) < \operatorname{Sp}(2n, \mathbb{R})$, determined by a finite sequence $h_1, ..., h_n \in \mathbb{Z}_{>0}$ with $h_1 \mid h_2 \mid \cdots \mid h_n$ (for instance, $h_1, ..., h_n$ are invariant factors of $\mathcal{B}_{\operatorname{cc}}$), and in the induced action of it on $\mathbb{H}_n$ through a particular embedding into $\operatorname{Sp}(2n, \mathbb{R})$. We first define $\boldsymbol{H} = \operatorname{diag}(h_1, ..., h_n)$ and $H$ as in \eqref{Equation: Skew-Smith form of H}. The group of our interest is as follows:
	\begin{equation*}
		\operatorname{Sp}(\mathbb{Z}^{2n}, H) := \left\{ A \in \operatorname{GL}(2n, \mathbb{Z}): A H A^T = H \right\}.
	\end{equation*}
	This group has an automorphism $A \mapsto A^{-T}$. There is then a group embedding
	\begin{equation}
		\label{Equation: group embedding of integral symplectic group}
		\operatorname{Sp}(\mathbb{Z}^{2n}, H) \hookrightarrow \operatorname{Sp}(2n, \mathbb{R}), \quad A \mapsto
		\begin{pmatrix}
			\boldsymbol{H}^{-1} & 0\\
			0 & I_n
		\end{pmatrix}
		A^{-T}
		\begin{pmatrix}
			\boldsymbol{H} & 0\\
			0 & I_n
		\end{pmatrix},
	\end{equation}
	which, together with \eqref{Equation: standard action of symplectic group on upper half space}, induces an action
	\begin{equation}
		\label{Equation: twisted action on upper half space}
		\operatorname{Sp}(\mathbb{Z}^{2n}, H) \times \mathbb{H}_n \to \mathbb{H}_n, \quad (A, \Omega) \mapsto A \cdot \Omega.
	\end{equation}
	This action satisfies the following property. If
	\begin{equation*}
		y = (y^1, ..., y^{2n}) \in \mathbb{R}^{2n}, \quad A = \begin{pmatrix}
			\boldsymbol{A}_1^1 & \boldsymbol{A}_1^2\\
			\boldsymbol{A}_2^1 & \boldsymbol{A}_2^2
		\end{pmatrix}
		\in \operatorname{Sp}(\mathbb{Z}^{2n}, H) \quad \text{and} \quad y' = A^{-T} y,
	\end{equation*}
	then for all $\Omega \in \mathbb{H}_n$, letting $\Omega' = A \cdot \Omega$,
	\begin{equation}
		(\boldsymbol{y}')^1 - \boldsymbol{H}^{-1} \Omega' (\boldsymbol{y}')^2 = ( \boldsymbol{A}_1^1 - \boldsymbol{H}^{-1} \Omega \boldsymbol{A}_1^2 )^{-1} ( \boldsymbol{y}^1 - \boldsymbol{H}^{-1} \Omega \boldsymbol{y}^2 ).
	\end{equation}
	
	\bibliographystyle{amsplain}
	\bibliography{References}
	
\end{document}